\documentclass[final,a4paper, 11pt]{article}
\RequirePackage{etex} 
\usepackage{authblk}
\usepackage{blindtext}
\usepackage{comment}
\usepackage[utf8]{inputenc}
\usepackage{amsmath}
\usepackage{amsthm}
\usepackage{amsfonts}
\usepackage{lipsum}
\usepackage{amssymb}
\usepackage{mathrsfs}
\usepackage{empheq}
\usepackage{authblk}
\usepackage{mathtools}
\usepackage{siunitx}
\usepackage{soul} 
\usepackage{amsbsy}
\usepackage{braket}
\usepackage[font=footnotesize]{caption}
\usepackage[export]{adjustbox}
\usepackage{xcolor}
\usepackage{multicol}
\usepackage{multirow}
\usepackage{graphicx}
\graphicspath{{IMG/}}
\usepackage{tikz}
\usepackage{pgfplots}
\pgfplotsset{compat=1.6}
\usepackage{placeins}
\usepackage[english]{babel}
\usepackage[margin=0.8in]{geometry}

\usepackage{float}
\usepackage{longtable}
\usepackage{colortbl}
\usepackage{subcaption}
\usepackage{listings}
\usepackage{hyperref}
\usepackage{cleveref}
\usepackage{autonum}
\usepackage{realboxes}
\theoremstyle{plain}
\newtheorem{definition}{Definition}
\numberwithin{definition}{section}

\newtheorem{assumption}{Assumption}
\numberwithin{assumption}{section}

\newtheorem{theorem}{Theorem}
\numberwithin{theorem}{section}

\newtheorem{proposition}{Proposition}
\numberwithin{proposition}{section}
\theoremstyle{plain}
\newtheorem{remark}
{Remark}
\theoremstyle{plain}
\newtheorem{lemma}{Lemma}[section]
\usepackage{tikz}

\usepackage{enumitem}
\usepackage[makeroom]{cancel}

\usepackage[square, numbers, sort&compress]{natbib} 
\newcommand\T{\rule{0pt}{2.6ex}}
\newcommand\B{\rule[-1.2ex]{0pt}{0pt}}
\newcommand{\jump}[1]{\left[\mkern-1.5mu \left[#1\right] \mkern-1.5mu\right]}
\newcommand{\avg}[1]{\left\{ \mkern-5mu \left\{#1 \right\} \mkern-5mu \right\}}
\newcommand{\wavg}[1]{\left\{ \mkern-5mu \left\{#1 \right\} \mkern-5mu \right\}_{\omega}}
\newcommand{\wTavg}[1]{\left\{ \mkern-5mu \left\{#1 \right\} \mkern-5mu \right\}_{\omega_{\boldsymbol{\Theta}}}}
\newcommand{\wPavg}[1]{\left\{ \mkern-5mu \left\{#1 \right\} \mkern-5mu \right\}_{\omega_{\mathbf{K}}}}
\newcommand{\wUavg}[1]{\left\{ \mkern-5mu \left\{#1 \right\} \mkern-5mu \right\}_{\omega_{\mu}}}

\renewcommand{\div}[1]{\nabla \mkern-2.5mu \cdot \mkern-2.5mu {#1}}
\newcommand{\divh}[1]{\nabla_h \mkern-2.5mu \cdot \mkern-2.5mu {#1}}
\newcommand{\specialcell}[2][c]{%
  \begin{tabular}[#1]{@{}c@{}}#2\end{tabular}}

\newcommand\nnfootnote[1]{%
  \begin{NoHyper}
  \renewcommand\thefootnote{}\footnote{#1}%
  \addtocounter{footnote}{-1}%
  \end{NoHyper}
}

\definecolor{myred}{rgb}{0.9, 0.0, 0.0}
\definecolor{myblue}{rgb}{0.0, 0.28, 0.67}
\definecolor{mygreen}{rgb}{0.0, 0.7, 0.0}
\definecolor{myyellow}{rgb}{1.0, 0.55, 0.0}

\begin{document}

\title{\textbf{Robust discontinuous Galerkin-based scheme for the fully-coupled non-linear thermo-hydro-mechanical problem}}
\author[a]{Stefano Bonetti\textsuperscript{*,}}
\author[a]{Michele Botti}
\author[a]{Paola F. Antonietti}
\affil[a]{\small{\textit{MOX, Dipartimento di Matematica, Politecnico di Milano, P.zza Leondardo da Vinci 32, 20133 Milano, Italy.}}}
\date{}

\maketitle

\begin{center}
\begin{minipage}[c]{1\textwidth}
\textbf{Abstract}
\smallskip
\newline
We present and analyze a discontinuous Galerkin method for the numerical modeling of the non-linear fully-coupled thermo-hydro-mechanic problem. We propose a high-order symmetric weighted interior penalty scheme that supports general polytopal grids and is robust with respect to strong heteorgeneities in the model coefficients. We focus on the treatment of the non-linear convective transport term in the energy conservation equation and we propose suitable stabilization techniques that make the scheme robust for advection-dominated regimes. The stability analysis of the problem and the convergence of the fixed-point linearization strategy are addressed theoretically under mild requirements on the problem's data. A complete set of numerical simulations is presented in order to assess the convergence and robustness properties of the proposed method.
\end{minipage}
\end{center}

\nnfootnote{* Corresponding author}
\nnfootnote{\textit{E-mail addresses:} \texttt{stefano.bonetti@polimi.it} (Stefano Bonetti), \texttt{michele.botti@polimi.it} (Michele Botti), \texttt{paola.antonietti@polimi.it} (Paola F. Antonietti)}

\section{Introduction}
\label{sec:RobQSTPE_introduction}
The thermo-hydro-mechanical (THM) coupling refers to the coupled interactions between temperature, fluid flow, and mechanical deformations. This coupling occurs in many natural and engineered systems. THM is relevant in several processes, such as environmental science, civil engineering, and material science. Applications of THM have an impact on many fields: environmental sustainability (e.g. $CO_2$ sequestration or geothermal energy production), public safety, economy. Moreover, it finds application also in seismicity/induced seismicity studies.

The THM problem is often formulated starting from Biot's equation of poroelasticity. Poroelasticity inspects the interaction between fluid flow and elastic deformation within a porous medium, indeed it is a suitable modelization of the subsoil. In fact, in the context of geophysical applications, we model it as a fully-saturated poroelastic material, and -- if the focus is not the study of seismic effects -- it is often assumed that we are in the small deformations regime. Thus, through this article we assume to be in the quasi-static regime. For the aforementioned applications, the temperature plays a key role in the description of the phenomena and its evolution, then it is included in the model via an energy conservation equation, leading to a fully-coupled THM system of equations \cite{Brun2019, Brun2020, AntoniettiBonetti2022}:

\begin{equation}
	\label{eq:RobQSTPE_ModelProblem}
	\left\{
	\begin{aligned}
		& \partial_t \left(a_0 T - b_0 p + \beta \div{\mathbf{u}} \right) - c_f \nabla T \cdot (\mathbf{K} \nabla p) - \div{(\boldsymbol{\Theta} \nabla T )} = H && \text{in} \ \Omega \times (0, T_f], \qquad (i)\\
		& \partial_t \left(c_0 p - b_0 T + \alpha \div{\mathbf{u}} \right) - \div{(\mathbf{K} \nabla p)} = g && \text{in} \ \Omega \times (0, T_f], \qquad (ii)\\
		& -\div{\boldsymbol{\sigma}(\mathbf{u},p,T)} = \mathbf{f} && \text{in} \ \Omega \times (0, T_f], \qquad (iii)\\
	\end{aligned}
	\right.
\end{equation}
where $\Omega$ is the computational domain, $T_f >0$ is the final simulation time, and $\boldsymbol{\sigma}(\mathbf{u}, p, T)$ is the total Cauchy stress tensor accounting for the effects of pressure and temperature. Our model is constituted by three equations: $(i)$ the energy conservation equation, that is solved in the temperature unknown $T$, $(ii)$ the mass conservation equation in the pore pressure unknown $p$, and $(iii)$ the momentum conservation equation in the solid displacement unknown $\mathbf{u}$. In $(ii)$, $(iii)$ we can recognize the structure of the Biot's system with the additional contribution of the temperature, which plays a role similar to the one of the pressure. Equation $(i)$ is similar to $(ii)$ but with the presence of both conduction and convection terms at the same time. The convection term is the one that makes the two equations different, moreover, we highlight that it is a nonlinear term. The correct handling of this non-linearity is one of the main points of this work: we propose a discretization technique that can handle the case of vanishing thermal conductivity. To this aim, we refer to the strategies used in the discontinuous Galerkin treatment of advection-diffusion problems in the advection-dominated regimes. Moreover, reasoning as in the hyperbolic framework, we add boundary terms that are in charge of enforcing boundary conditions in the \textit{inflow} part of the boundary of the domain. To ensure the robustness of the scheme for the quasi-incompressible limit -- i.e. with respect to volumetric locking phenomena -- we consider one additional scalar equation that is solved in the so-called total pressure auxiliary variable. The introduction of the total pressure also ensures the inf-sup stability of the problem. 

For the spatial discretization of the problem, we propose a discontinuous Galerkin (dG) finite element method on polytopal grids (PolyDG \cite{Cangiani2014}). Examples of PolyDG schemes can be found in  \cite{Antonietti2013, Bassi2012, Antonietti2019} for elliptic problems,  in \cite{Cangiani2017} for parabolic problems, and in   \cite{Antonietti.Botti.ea:21, Botti.Di-Pietro.ea:19, Botti2021} for poroelasticity. Moreover, in \cite{AntoniettiBonetti2022} a PolyDG method for thermo-hydro-mechanics model is analyzed. The PolyDG schemes are appealing in this context because they are geometrically flexible, i.e. they allow for polygonal/polyhedral elements (possibly agglomerated), local mesh refinement and coarsening, and allow to handle highly heterogeneous media by facilitating the representation of inner discontinuities. To further enhance robustness with respect to heterogeneities we consider a Symmetric Weighted Interior Penalty (WSIP) version of the PolyDG scheme, in which weighted averages \cite{Heinrich2002, Heinrich2003, Heinrich2005, Stenberg1998} are introduced in place of the standard averages operators of the discontinuous Galerkin methods \cite{Ern2009}. Finally, they are also suitable for high-order approximations and through the article, we show that this may be an important property for dealing with the robust treatment of the non-linearity. 

The major highlights of this paper are: \textit{(a)} a detailed description of the PolyDG-WSIP method for the THM problem; \textit{(b)} an in-depth discussion of the advection term, with suitable stabilization techniques to deal with the advection-dominated regime; and \textit{(c)} numerical verification of the robustness of the proposed method both in terms of stability for degenerating coefficients and in terms of capability to handle heterogeneities in the physical parameters. The numerical analysis presented in this work, compared to \cite{AntoniettiBonetti2022}, focuses on establishing stability estimates which are robust with respect to hydraulic and thermal conductivities. Additionally, a novel linearization strategy is proposed and analyzed by deriving a condition on the model data which ensures convergence. 

The rest of the paper is organized as follows: the model problem, the assumptions on the model's coefficients, and the four-fields formulation is presented in Section~\ref{sec:RobQSTPE_model_problem}. In Section~\ref{sec:RobQSTPE_discretization}, we design the PolyDG-WSIP space discretization, we detail the treatment of the non-linear advection term (cf. Section~\ref{sec:RobQSTPE_advection_dg}), introduce the stabilization techniques (cf. Section~\ref{sec:RobQSTPE_advection_stabilization}), and present the linearization algorithm. In Section~\ref{sec:RobQSTPE_stability}, we study the stability of the semi-discrete problem and study the convergence of the fixed-point iteration scheme. In Section~\ref{sec:RobQSTPE_conv_test} we assess the performance of the method in terms of accuracy and in Section~\ref{sec:RobQSTPE_rob_test} in terms of robustness with respect to the model coefficients.

\section{Model Problem}
\label{sec:RobQSTPE_model_problem}
Let $\Omega \subset \mathbb{R}^d$, $d \in \{2,3\}$, be an open bounded Lipschitz polygonal/polyhedral domain. We consider the coupled thermo-hydro-mechanical problem reading: \textit{find $(\mathbf{u}, p, T)$ such that it holds} 
\begin{subequations}
	\label{eq:RobQSTPE_QS_TPE_system}
	\begin{empheq}[left=\empheqlbrace]{align}
		& a_0 T - b_0 p + \beta \div{\mathbf{u}} - c_f \nabla T \cdot (\mathbf{K} \nabla p) - \div{(\boldsymbol{\Theta} \nabla T )} = H && \text{in} \ \Omega, \label{eq:RobQSTPE_energy_cons} \\
		& c_0 p - b_0 T + \alpha \div{\mathbf{u}} - \div{(\mathbf{K} \nabla p)} = g && \text{in} \ \Omega, \label{eq:RobQSTPE_mass_cons} \\
		& -\div{\boldsymbol{\sigma}(\mathbf{u},p,T)} = \mathbf{f} && \text{in} \ \Omega, \label{eq:RobQSTPE_momentum_cons}
	\end{empheq}
\end{subequations}
Note that $T$ represents the variation of the temperature with respect to a reference value \cite{Coussy2003}. The source terms $H$, $g$, $\mathbf{f}$ are a heat source, a fluid mass source, and a body force, respectively and are given. The constitutive law for the stress tensor $\boldsymbol{\sigma}$ (in \eqref{eq:RobQSTPE_momentum_cons}) in the linear elasticity framework is given by
\begin{equation}
	\label{eq:RobQSTPE_const_law_sigma}
	\boldsymbol{\sigma}(\mathbf{u},p,T) = 2 \mu \boldsymbol{\epsilon}(\mathbf{u}) + \lambda \div{\mathbf{u}} \mathbf{I} - \alpha p \mathbf{I} - \beta T \mathbf{I}, 
\end{equation}
where $\mathbf{I}$ is the identity tensor and $\boldsymbol{\epsilon}(\mathbf{u})= \frac{1}{2}(\nabla \mathbf{u} + \nabla \mathbf{u}^T)$ is the strain tensor. For the sake of simplicity, we supplement \eqref{eq:RobQSTPE_QS_TPE_system} 
with homogeneous Dirichlet conditions, namely $\mathbf{u} = \boldsymbol{0}$, $p = 0$, and $T=0$ on $\partial\Omega$. 

We remark that Problem~\ref{eq:RobQSTPE_QS_TPE_system} can be seen as one step of an implicit time advancing scheme (e.g. backward Euler method) applied to \eqref{eq:RobQSTPE_ModelProblem}. In this case, the conductivity tensors $\mathbf{K}$ in \eqref{eq:RobQSTPE_energy_cons} and $\boldsymbol{\Theta}$ in \eqref{eq:RobQSTPE_mass_cons} are scaled by the time-step $\delta t$, namely $\mathbf{K}=\delta t \tilde{\mathbf{K}}$ and $\boldsymbol{\Theta}=\delta t \tilde{\boldsymbol{\Theta}}$, where $\tilde{\mathbf{K}}$ and $\tilde{\boldsymbol{\Theta}}$ are the actual hydraulic mobility and heat conductivity of the medium, respectively.
For a detailed derivation of the quasi-static model we refer the reader to \cite{Brun2018}.
In Table~\ref{tab:RobQSTPE_TPE_params} we detail the parameters characterizing problem \eqref{eq:RobQSTPE_QS_TPE_system}-\eqref{eq:RobQSTPE_const_law_sigma} specifying their physical interpretation and their corresponding unit of measure. For a detailed discussion on the parameters and the relations among them we refer to \cite{AntoniettiBonetti2022}.
\begin{table}[ht]
	\centering 
	\footnotesize
	\begin{tabular}{ c | c | l }
		\textbf{Notation} & \textbf{Quantity} & \textbf{Unit} \\[3pt]
		$a_0$ & thermal capacity & \si[per-mode = symbol]{\pascal \per \kelvin \squared} \\
		$b_0$ & thermal dilatation coefficient & \si{\per \kelvin} \\
		$c_0$ & specific storage coefficient & \si{\per\pascal} \\
		$\alpha$ & Biot--Willis constant & - \\
		$\beta$ & thermal stress coefficient & \si[per-mode = symbol]{\pascal \per \kelvin} \\
		$c_f$  & fluid volumetric heat capacity divided by reference temperature & \si[per-mode = symbol]{\pascal \per \kelvin\squared} \\ 
		$\mu, \lambda$ & Lamé parameters & \si{\pascal} \\
		$\tilde{\mathbf{K}}$ & permeability divided by fluid viscosity & \si[per-mode = symbol]
		{\metre\squared \per \pascal \per \second} \\
		$\tilde{\boldsymbol{\Theta}}$ & effective thermal conductivity & \si[per-mode = symbol]{\metre\squared \pascal \per \kelvin\squared \per \second} \\
		$\phi$ & porosity & - \\
	\end{tabular}
	\caption{Thermo-hydro-mechanics coefficients appearing in \eqref{eq:RobQSTPE_QS_TPE_system}, \eqref{eq:RobQSTPE_const_law_sigma}.}
	\label{tab:RobQSTPE_TPE_params}
\end{table}

\subsection{Notation and assumptions}
For $X\subseteq\Omega$, we denote by $L^p(X)$ the standard Lebesgue space of index $p\in [1, \infty]$ and by $H^q(X)$ the Sobolev space of index $q \geq 0$ of real-valued functions defined on $X$, with the convention that $H^0(X)=L^2(X)$. 
The notation $\mathbf{L}^2(X)$ and $\mathbf{H}^q(X)$ is adopted in place of $\left[ L^2(X) \right]^d$ and $\left[ H^q(X) \right]^d$, respectively. 
In addition, we denote by $\mathbf{H}(\textrm{div},X)$ the space of $\mathbf{L}^2(X)$ vector fields whose divergence is square integrable. These spaces are equipped with natural inner products and norms denoted by $(\cdot, \cdot)_X = (\cdot, \cdot)_{L^2(X)}$ and $||\cdot||_X = ||\cdot||_{L^2(X)}$, with the convention that the subscript can be omitted in the case $X=\Omega$.
For the sake of brevity, in what follows, we make use of the symbol $x \lesssim y$ to denote $x \le C y$, where $C$ is a positive constant independent of the discretization parameters.

Following \cite{Brun2019}, we introduce suitable assumptions on the problem data, both on the forcing and boundary data and on the problem parameters:
\begin{assumption}[Assumptions on the problem data]
	\hspace{0pt}
	We assume that:
	\label{assumption:RobQSTPE_model_problem}
	\begin{enumerate}
		\item the hydraulic mobility $\textup{\textbf{K}}= (K)^d_{i,j=1}$ and heat conductivity $\boldsymbol{\Theta}=(\Theta)^d_{i,j=1}$ are symmetric tensor fields which, for strictly positive real numbers $k_M>k_m$ and $\theta_M> \theta_m$, satisfy
		\begin{equation}
			k_m |\zeta|^2 \leq \zeta^T \textup{\textbf{K}}(x)\zeta \leq k_M |\zeta|^2
			\quad \text{and} \quad
			\theta_m |\zeta|^2 \leq \zeta^T \boldsymbol{\Theta}(x) \zeta \leq \theta_M |\zeta|^2, \quad\forall\zeta \in \mathbb{R}^d, \ \text{a.e.} \ x \in \Omega;
		\end{equation}
		\item The shear modulus $\mu$ and the fluid heat capacity $c_f$ are scalar fields such that $\mu:\Omega\to[\mu_m,\mu_M]$ and $c_f:\Omega\to[0, c_{fM}]$ with $0<\mu_m\le\mu_M$ and $0\le c_{fM}$;
		\item 
		the coupling parameters $\alpha:\Omega\to (\phi,1]$ and $\beta:\Omega\to (0,\beta_M]$ are strictly positive;
		\item the scalar fields $\lambda$, $c_0$, $b_0$, and $a_0$ are such that $\lambda\ge 0$ and $a_0, c_0 \geq b_0\geq 0$;
		\item the forcing terms are chosen such that $g,H \in L^2(\Omega)$ and $\mathbf{f} \in \mathbf{L}^2(\Omega)$.
	\end{enumerate}
\end{assumption}

\subsection{Four-field formulation}
\label{sec:RobQSTPE_RobQSTPE_4field_form}

As in \cite{AntoniettiBonetti2022}, we refer to the four-field formulation of the THM problem obtained by introducing the total pressure auxiliary variable $\varphi$ (\cite{Oyarzua2016, Botti2021}), in which we include all the volumetric contributions to the stress tensor, namely:
$$
\varphi = \lambda\div{\mathbf{u}}-\alpha p -\beta T.
$$
We introduce the functional spaces $\mathbf{V} = \mathbf{H}^1_0(\Omega),\ V = H^1_0(\Omega)$, and $Q = L^2(\Omega)$. Then, the weak formulation of \eqref{eq:RobQSTPE_QS_TPE_system} reads: \textit{find $(\mathbf{u},p,T,\varphi) \in \mathbf{V} \times V \times V \times Q$ such that:}
\begin{equation}
	\label{eq:RobQSTPE_semi_discrete_cont_not_lin}
	\begin{aligned}
		& \mathcal{M}((p, T, \varphi), (q, S, \psi)) + (\boldsymbol{\Theta} \nabla T,\nabla S) - (c_f \nabla T \cdot (\mathbf{K} \nabla p),S) + (\mathbf{K} \nabla p,\nabla q) +  (2 \mu \boldsymbol{\epsilon}(\mathbf{u}),\boldsymbol{\epsilon}(\mathbf{v})) \\
		& + (\varphi,\nabla \cdot \mathbf{v}) - (\nabla \cdot \mathbf{u},\psi) = (H,s) + (g,q) + (\mathbf{f},\mathbf{v}) \qquad \forall \ (\mathbf{v},q,S,\psi) \in \mathbf{V} \times V \times V \times Q.
	\end{aligned}
\end{equation}
where the bilinear form $\mathcal{M}:V\times V \times Q\to\mathbb{R}$ is given by:
$$
\begin{aligned}
	\mathcal{M}((p, T, \varphi), (q, S, \psi)) = (b_0 (p - T),q-S) &+ ((a_0-b_0)T,S) + ((c_0-b_0)p,q) \\ 
	& + (\lambda^{-1}(\varphi + \alpha p + \beta T),\psi + \alpha q + \beta S).
\end{aligned}
$$
\begin{remark}
	The convection term $c_f \nabla T \cdot (\mathbf{K} \nabla p)$ should not be tested by an $H^1$-regular function, since it is only in $L^1(\Omega)$. However, it can be inferred from the thermal energy equation \eqref{eq:RobQSTPE_energy_cons} and the assumption on the problem data that $c_f \nabla T \cdot (\mathbf{K} \nabla p) \in H^{-1}(\Omega)$, where $H^{-1}(\Omega)$ is the dual space of $V$. Therefore, the third term in the left-hand side of \eqref{eq:RobQSTPE_semi_discrete_cont_not_lin} has to be intended as the duality product 
	$$
	\langle c_f \nabla T \cdot (\mathbf{K} \nabla p),S \rangle_{H^{-1}(\Omega),H^1_0(\Omega)}.
	$$
	For additional details on the well-posedness of coupled Darcy and heat equations, we refer to \cite{Bernardi.Dib:18}.
\end{remark}

\section{Discretization}
\label{sec:RobQSTPE_discretization}
The aim of this section is to derive the PolyDG approximation of problem \eqref{eq:RobQSTPE_QS_TPE_system}. We start by introducing some preliminaries on the discretization. Then, we show the PolyDG scheme, in which we exploit the Symmetric Weighted Interior Penalty method (WSIP) \cite{Ern2009}, in order to make the method able to cope with strong heterogeneities in the physical coefficients. A particular focus will be devoted to the linearization and stabilization procedures.

\subsection{Preliminaries}
\label{sec:RobQSTPE_DG_preliminaries}
The aim of this section is to introduce some instrumental assumptions and results on the PolyDG method. For designing the PolyDG discretization of Problem~\eqref{eq:RobQSTPE_QS_TPE_system} we start by introducing a polytopic subdivision $\mathcal{T}_h$ of the computational domain $\Omega$ and its features. An interface is defined as a planar, simplitial subset of the intersection of the boundaries of any two neighbouring elements of $\mathcal{T}_h$. 
In the following, we denote with $\mathcal{F}$, $\mathcal{F}_I$, and $\mathcal{F}_B$ the set of faces, interior faces, and boundary faces, respectively.
In what follows, we introduce the main assumptions on the mesh $\mathcal{T}_h$ \cite{Cangiani2014, CangianiDong:17}.
\begin{definition}[Polytopic regular mesh]
	\label{def:unif_regular}
	A mesh $\mathcal{T}_h$ is said to be polytopic regular if $\forall \kappa \in \mathcal{T}_h$, there exist a set of non-overlapping $d$-dimensional simplices contained in $\kappa$, denoted by $\{S_{\kappa}^F\}_{F \subset \partial \kappa}$, such that, for any face $F \subset \partial \kappa$ the following condition holds: $h_{\kappa} \lesssim d \ |S_{\kappa}^F| \ |F|^{-1}$.
\end{definition}

\begin{assumption}
	\label{ass:RobQSTPE_mesh_Th1}
	Given $\{\mathcal{T}_h\}_h, h>0$, we assume that the following properties are uniformly satisfied:
	\begin{enumerate}[start=1,label={\bfseries A.\arabic* }]
		\item \label{ass:RobQSTPE_A1} $\mathcal{T}_h$ is polytopic-regular;
			\item \label{ass:RobQSTPE_A3} For any neighbouring elements $\kappa^+, \kappa^- \in \mathcal{T}_h$, hp-local bounded variation property holds, i.e. \\$h_{\kappa^+} \lesssim h_{\kappa^-} \lesssim h_{\kappa^+}$, $p_{\kappa^+} \lesssim p_{\kappa^-} \lesssim p_{\kappa^+}$.
		\end{enumerate}
	\end{assumption}
	Note that the bounded variation hypothesis \ref{ass:RobQSTPE_A3} is introduced to avoid technicalities. Under \ref{ass:RobQSTPE_A1} the following inequality (\textit{trace-inverse} inequality) holds \cite{Cangiani2017}: 
	\begin{equation}
		\label{eq:RobQSTPE_trace_inverse_ineq}
		\| v\|_{L^2(\partial \kappa)} \le C_{\mathrm{tr}}  
		h^{-\frac12}\, \ell\, \|v\|_{L^2(\kappa)} \quad \forall v \in \mathbb{P}^{\ell}(\kappa),
	\end{equation}
	where $\mathbb{P}^{\ell}(\kappa)$ is the space of polynomials of maximum degree equal to $\ell$ in $\kappa$ and $C_{\mathrm{tr}}>0$ is independent of $\ell, h$, the number of faces per element, and the relative size of a face compared to the diamater of the element it belongs to.
	
	For the sake of simplicity, we assume that the parameters $\boldsymbol{\Theta}, \mathbf{K}$, $\mu$ and $c_f$ are element-wise constant. Then, we can introduce the following quantities:
	\begin{equation}
		\Theta_{\kappa} = \left(|\sqrt{\boldsymbol{\Theta}\rvert_{\kappa}}|_2^2 \right), \qquad K_{\kappa} = \left(|\sqrt{\mathbf{K}\rvert_{\kappa}}|_2^2 \right), \qquad \mu_{\kappa} = \mu \rvert_{\kappa},
		\qquad c_{f,\kappa} = c_f \rvert_{\kappa},
	\end{equation}
	where $|\cdot|_2$ is the $\ell^2$-norm in $\mathbb{R}^{d \times d}$. We remark that this assumption is reasonable in the context of groundwater flow models, where the data are derived through local measurements.
	
	\subsection{The PolyDG-WSIP discrete problem}
	\label{sec:RobQSTPE_DG_problem}
	In this section, we present the WSIP method \cite{Ern2009} and discuss its application to the THM problem. The key ingredient of the method is to use weighted averages instead of arithmetic ones. The use of weighted averages has been introduced for elliptic problems in \cite{Stenberg1998} and then developed for discontinuous Galerkin methods (dG-WSIP) for dealing with advection-diffusion problems with locally vanishing diffusion \cite{Ern2009}. As one of the aims of this work is to inspect the robustness with respect to the model's coefficients, we use this modification of the standard PolyDG scheme. Other than ensuring numerical robustness with respect to large heterogeneities, another advantage of dG-WSIP is that it requires minimal modifications with respect to standard dG schemes both in terms of analysis and coding.
	
	For the definition of the WSIP method we introduce the weight function $\omega^+:\mathcal{F}_I\to [0,1]$ \cite{Heinrich2002, Heinrich2003, Heinrich2005, Stenberg1998, Ern2009}. Given an interior face $F \in \mathcal{F}_I$, we denote the values taken by $\omega^+$ and $\omega^- = 1-\omega^+$ on the face $F$ as $\omega\rvert_F^{+}$ and $\omega\rvert_F^{-}$, respectively. Given the function $\omega$ we can introduce the notion of weighted averages and jump operators, denoted with $\wavg{\cdot}$ and $\jump{\cdot}$, and of normal jump, denoted by $\jump{\cdot}_n$ \cite{Arnold2002, Ern2009}:
	\begin{equation}
		\label{eq:RobQSTPE_avg_jump_operators}
		\begin{aligned}
			& \jump{a} = a^+ \mathbf{n^+} + a^- \mathbf{n^-}, \quad && \jump{\mathbf{a}} = \mathbf{a}^+ \odot \mathbf{n^+} + \mathbf{a}^- \odot \mathbf{n^-}, \quad &&\jump{\mathbf{a}}_n = \mathbf{a}^+ \cdot \mathbf{n^+} + \mathbf{a}^- \cdot \mathbf{n^-}, \\ 
			& \wavg{a} = \omega^+ a^+ + \omega^- a^, \quad && \wavg{\mathbf{a}} = \omega^+ \mathbf{a}^+ + \omega^- \mathbf{a}^-, \quad && \wavg{\mathbf{A}} = \omega^+ \mathbf{A}^+ + \omega^- \mathbf{A}^-,
		\end{aligned}
	\end{equation}
	where $\mathbf{a} \odot \mathbf{n} = \mathbf{a}\mathbf{n}^T$, and $a, \ \mathbf{a}, \ \mathbf{A}$ are (regular enough) scalar-valued, vector-valued, and tensor-valued functions, respectively. The subscript $\omega$ in the weighted-average operator is omitted whenever $\omega^+ = \omega^- = 1/2$. On boundary faces $F\in\mathcal{F}_B$, we set
	$ \jump{a} = a \mathbf{n},\ \wavg{a} = a,\ \jump{\mathbf{a}} = \mathbf{a} \odot \mathbf{n},\ \wavg{\mathbf{a}} = \mathbf{a},\ \jump{\mathbf{a}}_n = \mathbf{a} \cdot \mathbf{n},\ \wavg{\mathbf{A}} = \mathbf{A}.$ For the averages, this corresponds to consider $\omega^\pm$ single-valued and equal to $1$.
	
	We start deriving the PolyDG-WSIP formulation of problem \eqref{eq:RobQSTPE_semi_discrete_cont_not_lin} by intoducing the discrete spaces that are used in the following. Given $m,\ell \geq 1$, we define:
	\begin{equation}
		\begin{aligned}
			Q_h^{m} &= \left\{ \psi \in L^2(\Omega) : \psi |_{\kappa} \in \mathbb{P}^{m}(\kappa) \  \forall \kappa \in \mathcal{T}_h \right\}\hspace{-0.5mm}, \;
			V_h^{\ell} = \left\{ v \in L^2(\Omega) : v |_{\kappa} \in \mathbb{P}^{\ell}(\kappa) \  \forall \kappa \in \mathcal{T}_h \right\}\hspace{-0.5mm}, \;
			\mathbf{V}_h^{\ell} = \left[V_h^{\ell}\right]^d\hspace{-1mm}.
		\end{aligned}
	\end{equation}
	The PolyDG-WSIP discretization of problem \eqref{eq:RobQSTPE_semi_discrete_cont_not_lin} reads:
	\textit{find $(\mathbf{u}_h,p_h,T_h,\varphi_h) \in \mathbf{V}_h^{\ell} \times V_h^{\ell} \times V_h^{\ell} \times Q_h^m$ such that $\forall \ (\mathbf{v}_h,q_h,S_h,\varphi_h) \in \mathbf{V}_h^{\ell} \times V_h^{\ell} \times V_h^{\ell} \times Q_h^m$ it holds}
	\begin{equation}
		\label{eq:RobQSTPE_discrete_weak_form_Dtime_1.2}
		\begin{aligned}
			& \mathcal{M}((p_h, T_h, \varphi_h),(q_h,S_h,\psi_h)) + \mathcal{A}_h^{T}(T_h,S_h) + \widetilde{\mathcal{C}}_h^{\text{stab}}(T_h,p_h,S_h)
			+ \mathcal{A}_h^{p}(p_h,q_h) + \mathcal{A}_h^{e}(\mathbf{u}_h,\mathbf{v}_h) \\
			&\; - \mathcal{B}_h(\varphi_h,\mathbf{v}_h) +  \mathcal{B}_h(\psi_h, \mathbf{u}_h) +  \mathcal{D}_h(\varphi_h,\psi_h)
			= \left((\mathbf{f}, g, H),(\mathbf{v}_h, q_h, S_h)\right),
		\end{aligned}
	\end{equation}
	where the bilinear and trilinear forms are defined by:
	\begin{equation}
		\label{eq:RobQSTPE_bilinear_forms_discr}
		\begin{aligned}
			& \mathcal{A}_h^T(T,S) = \left(\boldsymbol{\Theta}\nabla_h T, \nabla_h S\right) - \sum_{F \in \mathcal{F}} \int_F \left(\wTavg{\boldsymbol{\Theta}\nabla_h T} \mkern-2.5mu \cdot \mkern-2.5mu \jump{S} + \jump{T} \mkern-2.5mu \cdot \mkern-2.5mu \wTavg{\boldsymbol{\Theta}\nabla_h S} - \sigma \jump{T} \mkern-2.5mu \cdot \mkern-2.5mu \jump{S}\right),\\
			& \mathcal{A}_h^p(p,q) = (\mathbf{K} \nabla_h p,\nabla_h q) - \sum_{F \in \mathcal{F}} \int_F \left(\wPavg{\mathbf{K} \nabla_h p} \mkern-2.5mu \cdot \mkern-2.5mu \jump{q} + \jump{p} \mkern-2.5mu \cdot \mkern-2.5mu \wPavg{\mathbf{K} \nabla_h q} - \xi \jump{p} \mkern-2.5mu \cdot \mkern-2.5mu \jump{q}\right),\\
			& \mathcal{A}_h^e(\mathbf{u},\mathbf{v}) = (2 \mu\boldsymbol{\epsilon}_h(\mathbf{u}),\boldsymbol{\epsilon}_h(\mathbf{v}))
			- \sum_{F \in \mathcal{F}} \int_F \left( \wUavg{2 \mu\boldsymbol{\epsilon}_h(\mathbf{u})} \mkern-2.5mu : \mkern-2.5mu \jump{\mathbf{v}} + \jump{\mathbf{u}} \mkern-2.5mu : \mkern-2.5mu \wUavg{2 \mu\boldsymbol{\epsilon}_h(\mathbf{v})} -  \zeta \jump{\mathbf{u}} \mkern-2.5mu : \mkern-2.5mu \jump{\mathbf{v}}\right),\\
			& \mathcal{B}_h(\varphi,\mathbf{v}) = - (\varphi,\nabla_h \mkern-2.5mu \cdot \mkern-2.5mu \mathbf{v}) + \sum_{F \in \mathcal{F}} \int_F \avg{ \varphi} \mkern-2.5mu \cdot \mkern-2.5mu \jump{\mathbf{v}}_n,\\
			& \begin{aligned}
				\widetilde{\mathcal{C}}^{\text{stab}}_h (T,p,S) = \ & \big( -c_f \left(\mathbf{K} \ \nabla_h p\right) \cdot \nabla_h T, S \big) - \sum_{F \in \mathcal{F}_I} \int_{F} \left( \avg{ - c_f \ \mathbf{K} \nabla_h p } \cdot \jump{T} \right) \avg{S} \\
				& + \frac{1}{2}\sum_{F \in \mathcal{F}} \int_F \ \left| \avg{ - c_f \ \mathbf{K} \nabla_h p} \cdot \mathbf{n} \right| \jump{T} \mkern-2.5mu \cdot \mkern-2.5mu \jump{S} - \frac{1}{2}\sum_{F \in \mathcal{F}_B} \int_F \  (- c_f \ \mathbf{K} \nabla_h p) \cdot \mathbf{n} \, T \, S \end{aligned} \\
			&\mathcal{D}_h(\varphi,\psi) = \sum_{F \in \mathcal{F}_I} \int_F \varrho \jump{\varphi} \mkern-2.5mu \cdot \mkern-2.5mu \jump{\psi}.
		\end{aligned}
	\end{equation}
	For all $w\in V_h^{\ell}$ and $\mathbf{w}\in \mathbf{V}_h^{\ell}$, $\nabla_h w$ and $\divh{\mathbf{w}}$ denote the broken differential operators whose restrictions to each element $k \in \mathcal{T}_h$ are defined as $\nabla w_{|k}$ and $\div{\mathbf{w}}_{|k}$, respectively, and $\boldsymbol{\epsilon}_h(\mathbf{u}) = \left(\nabla_h \mathbf{u} + \nabla_h \mathbf{u}^T\right)/2$. In \eqref{eq:RobQSTPE_bilinear_forms_discr} we set:
	\begin{equation}
		\omega_{\boldsymbol{\Theta}}^{\pm} = \frac{\delta_{\boldsymbol{\Theta}_n}^{\mp}}{\delta_{\boldsymbol{\Theta}_n}^{+} + \delta_{\boldsymbol{\Theta}_n}^{-}}, \qquad \omega_{\mathbf{K}}^{\pm} = \frac{\delta_{\mathbf{K}_n}^{\mp}}{\delta_{\mathbf{K}_n}^{+} + \delta_{\mathbf{K}_n}^{-}}, \qquad 
		\omega_{\mu}^{\pm} = \frac{\mu^{\mp}}{\mu^{+} + \mu^{-}},
	\end{equation}
	where $\delta_{\boldsymbol{\Theta}_n}^{\pm} = \mathbf{n}^{{\pm}^T} \, \boldsymbol{\Theta}^{\pm} \, \mathbf{n}^{{\pm}}$, $\delta_{\mathbf{K}_n}^{\pm} = \mathbf{n}^{{\pm}^T} \, \mathbf{K}^{\pm} \, \mathbf{n}^{{\pm}}$. Note that, the PolyDG-WSIP method requires also a different definition of penalty coefficients with respect to standard IP method \cite{Arnold1982, Wheeler1978, Ern2021, Cangiani2017}. Thus, the stabilization functions $\sigma, \xi, \zeta, \varrho \in L^{\infty}(\mathcal{F}_h)$ appearing in \eqref{eq:RobQSTPE_bilinear_forms_discr} are defined according to \cite{Ern2009} as:
	\begin{equation}
		\label{eq:RobQSTPE_stabilization_func}
		\begin{aligned}
			\sigma &= \left\{\begin{aligned}
				&\alpha_1 \gamma_{\boldsymbol{\Theta}} \underset{\kappa \in \{\kappa^+,\kappa^-\}}{\mbox{max}} \left(\frac{\ell^2}{h_{\kappa}}\right) \quad & F \in \mathcal{F}_I,\\
				&\alpha_1 \overline{\Theta}_{\kappa} \frac{\ell^2}{h_{\kappa}} \quad & F \in \mathcal{F}_B,\\
			\end{aligned}
			\right.
			\qquad
			\xi &&= \left\{\begin{aligned}
				&\alpha_2 \gamma_{\mathbf{K}} \underset{\kappa \in \{\kappa^+,\kappa^-\}}{\mbox{max}}\left(\frac{\ell^2}{h_{\kappa}}\right) \quad & F \in \mathcal{F}_I,\\
				&\alpha_2 \overline{K}_{\kappa} \frac{\ell^2}{h_{\kappa}} \quad & F \in \mathcal{F}_B,\\
			\end{aligned}
			\right.\\
			\zeta &= \left\{\begin{aligned}
				&\alpha_3 \gamma_{\mu} \underset{\kappa \in \{\kappa^+,\kappa^-\}}{\mbox{max}}\left(\frac{\ell^2}{h_{\kappa}}\right) \quad & F \in \mathcal{F}_I,\\
				&\alpha_3 \mu_{\kappa} \frac{\ell^2}{h_{\kappa}} \quad & F \in \mathcal{F}_B,\\
			\end{aligned}
			\right.
			\qquad
			\varrho &&= \left\{\begin{aligned}
				&\alpha_4 \underset{\kappa \in \{\kappa^+,\kappa^-\}}{\mbox{min}}\left(\frac{h_{\kappa}}{m}\right) \quad & F \in \mathcal{F}_I,\\
				&\alpha_4 \frac{h_{\kappa}}{m} \quad & F \in \mathcal{F}_B,\\
			\end{aligned}
			\right.\\
		\end{aligned}   
	\end{equation}
	where $\alpha_1, \alpha_2, \alpha_3, \alpha_4 \in \mathbb{R}$ are positive constants to be properly defined, $h_{\kappa}$ is the diameter of the element $\kappa \in \mathcal{T}_h$, and the coefficients $\gamma_{\boldsymbol{\Theta}}$, $\gamma_{\mathbf{K}}$, $\gamma_{\mu}$ are given by:
	\begin{equation}
		\gamma_{\boldsymbol{\Theta}} = \frac{ \delta_{\boldsymbol{\Theta}_n}^{+} \, \delta_{\boldsymbol{\Theta}_n}^{-}}{\delta_{\boldsymbol{\Theta}_n}^{+} + \delta_{\boldsymbol{\Theta}_n}^{-}}, \qquad \gamma_{\mathbf{K}} = \frac{\delta_{\mathbf{K}_n}^{+} \, \delta_{\mathbf{K}_n}^{-}}{\delta_{\mathbf{K}_n}^{+} + \delta_{\mathbf{K}_n}^{-}}, \qquad \gamma_{\mu}^{\pm} = \frac{\mu^{+} \, \mu^{-}}{\mu^{+} + \mu^{-}}.
	\end{equation}
	We point out that in the discrete formulation above, we have decided to consider the same polynomial degree for the spaces $V_h^{\ell}$ and $\mathbf{V}_h^{\ell}$, because we are mainly interested in approximation schemes yielding the same accuracy for all the primary variables.
	\begin{remark}
		Note that, in \eqref{eq:RobQSTPE_discrete_weak_form_Dtime_1.2}, we have added an additional weakly consistent stabilization term $\mathcal{D}_h$ for the total pressure following the dG discretization of the Stokes problem analyzed in \cite{antonietti2020_stokesDG}.
	\end{remark}
	
	\subsubsection{Treatment of the advection term in the dG framework}
	\label{sec:RobQSTPE_advection_dg}
	
	In this section, we detail the derivation and the analysis of the first two terms that appear in the trilinear form $\widetilde{\mathcal{C}}^{\text{stab}}_h$. We observe that, assuming that the pressure field is known, the non-linear convective term in \eqref{eq:RobQSTPE_semi_discrete_cont_not_lin} reduces to an advection-like term.
	Thus, we treat the transport term in the dG framework following approach of \cite{Ayuso2009, Brezzi2004, DiPietro2012, Houston2002}. 
	For the sake of exposition, we introduce the trilinear form 
	\begin{equation}
		\label{eq:RobQSTPE_Ctilde}
		\begin{aligned}
			\widetilde{\mathcal{C}}_h (T,p,S) = \ & \big( -c_f \left(\mathbf{K} \ \nabla_h p\right) \cdot \nabla_h T, S \big) - \sum_{F \in \mathcal{F}_I} \int_{F} \left( \avg{ -c_f \ \mathbf{K} \nabla_h p } \mkern-2.5mu \cdot \mkern-2.5mu \jump{T} \right) \avg{S} \quad T, p, S \in V_h^{\ell}.
		\end{aligned}
	\end{equation}
		%
		The following proposition will be instrumental for deriving the a-priori analysis of problem \eqref{eq:RobQSTPE_discrete_weak_form_Dtime_1.2}.
		\begin{proposition}
			\label{prop:RobQSTPE_trasporto}
			The trilinear form $\widetilde{\mathcal{C}}_h$ defined in \eqref{eq:RobQSTPE_Ctilde} satisfies for any  $T, p \in V_h^{\ell}$:
			\begin{equation}
				\label{eq:RobQSTPE_trasporto}
				\begin{aligned}
					\widetilde{\mathcal{C}}_h(T, p, T) = - \frac{1}{2} \big(\divh{\left( - c_f  \mathbf{K} \nabla_h p \right)}, T^2 \big) + \frac{1}{2} \sum_{F \in \mathcal{F}} \int_F \jump{ -c_f \mathbf{K} \nabla_h p}_n \avg{T^2}. 
				\end{aligned}
			\end{equation}
		\end{proposition}
		
		\begin{proof}
			We start the proof by recalling the following property that holds for any given product of scalar- and vector-valued functions $\mathbf{w} \in \mathbf{V}_h^{\ell-1}$, $v \in V_h^{\ell}$: $\divh{(\mathbf{w} \, v)} = \mathbf{w} \cdot \nabla_h v + (\divh{\mathbf{w}}) \, v$.
			
			We exploit the previous identity and element-wise integration by parts to rewrite $\widetilde{C}_h$ as:
			\begin{equation}
				\label{eq:RobQSTPE_positivity_1}
				\begin{aligned}
					\widetilde{\mathcal{C}}_h (T,p,S) = \ & \frac{1}{2} \sum_{\kappa \in \mathcal{T}_h} \int_{\kappa} \bigg( (-c_f \mathbf{K} \ \nabla_h p ) \cdot \nabla T \bigg) S + \frac{1}{2} \sum_{\kappa \in \mathcal{T}_h} \int_{\kappa} \divh{\bigg(\left(-c_f \mathbf{K} \ \nabla_h p \right) T\bigg)} \, S \\
					& - \frac{1}{2} \sum_{\kappa \in \mathcal{T}_h} \int_{\kappa} \divh{\left(-c_f \mathbf{K} \ \nabla_h p \right)} \, T \, S - \sum_{F \in \mathcal{F}_I} \int_F \left( \avg{ -c_f \mathbf{K} \ \nabla_h p } \mkern-2.5mu \cdot \mkern-2.5mu \jump{T} \right) \avg{S}. \\
					= \ & \frac{1}{2} \sum_{\kappa \in \mathcal{T}_h} \int_{\kappa} \bigg( (-c_f \mathbf{K} \ \nabla_h p ) \cdot \nabla T \bigg) S - \frac{1}{2} \sum_{\kappa \in \mathcal{T}_h} \int_{\kappa} \bigg(\left(-c_f \mathbf{K} \ \nabla_h p \right) T\bigg) \cdot \nabla S \\
					& + \frac{1}{2} \sum_{\kappa \in \mathcal{T}_h} \int_{\partial \kappa} \bigg(\left(-c_f \mathbf{K} \ \nabla_h p \right) \cdot \mathbf{n} \bigg) \,  T \, S - \frac{1}{2} \sum_{\kappa \in \mathcal{T}_h} \int_{\kappa} \divh{\left(-c_f \mathbf{K} \ \nabla_h p \right)} \, T \, S \\
					& - \sum_{F \in \mathcal{F}_I} \int_F \left( \avg{ -c_f \mathbf{K} \ \nabla_h p } \mkern-2.5mu \cdot \mkern-2.5mu \jump{T} \right) \avg{S}. \\
				\end{aligned}
			\end{equation}
			Reasoning as in \cite{Arnold1982}, the third term in the right-hand side of \eqref{eq:RobQSTPE_positivity_1} can be written as:
			\begin{equation}
				\label{eq:RobQSTPE_positivity_magicformula}
				\begin{aligned}
					\frac{1}{2} \sum_{\kappa \in \mathcal{T}_h} & \int_{\partial \kappa} \bigg(\left(-c_f \mathbf{K} \ \nabla_h p \right) \cdot \mathbf{n} \bigg) \,  T \, S = \\
					= \ & \frac{1}{2} \sum_{F \in \mathcal{F}_I} \bigg[ \int_F \avg{-c_f \mathbf{K} \ \nabla_h p} \cdot \jump{T \, S} + \int_F \jump{-c_f \mathbf{K} \ \nabla_h p}_n \avg{T \, S} \bigg] \\
					& + \frac{1}{2} \sum_{F \in \mathcal{F}_B} \int_F \bigg(\left(-c_f \mathbf{K} \ \nabla_h p \right) \cdot \mathbf{n} \bigg) \, T \, S \\
					= \ & \frac{1}{2} \sum_{F \in \mathcal{F}_I} \int_F \left( \avg{-c_f \mathbf{K} \ \nabla_h p} \cdot \jump{T} \right) \, \avg{S} + \frac{1}{2} \sum_{F \in \mathcal{F}_I} \int_F \left( \avg{-c_f \mathbf{K} \ \nabla_h p} \cdot \jump{S} \right) \, \avg{T} \\
					& + \frac{1}{2} \sum_{F \in \mathcal{F}_I} \int_F \jump{-c_f \mathbf{K} \ \nabla_h p}_n \avg{T \, S} \bigg] + \frac{1}{2} \sum_{F \in \mathcal{F}_B} \int_F \bigg(\left(-c_f \mathbf{K} \ \nabla_h p \right) \cdot \mathbf{n} \bigg) \, T \, S.
				\end{aligned}
			\end{equation}
			By plugging this expression into \eqref{eq:RobQSTPE_positivity_1} and grouping together the terms that involve the integral on the faces, we have:
			\begin{equation}
				\label{eq:RobQSTPE_positivity_3}
				\begin{aligned}
					\widetilde{\mathcal{C}}_h (T,p,S) = \ & \frac{1}{2} \sum_{\kappa \in \mathcal{T}_h} \int_{\kappa} \bigg( (-c_f \mathbf{K} \ \nabla_h p ) \cdot \nabla T \bigg) S - \frac{1}{2} \sum_{\kappa \in \mathcal{T}_h} \int_{\kappa} \bigg(\left(-c_f \mathbf{K} \ \nabla_h p \right) T\bigg) \cdot \nabla S \\
					& + \frac{1}{2} \sum_{F \in \mathcal{F}_I} \int_F \left( \avg{-c_f \mathbf{K} \ \nabla_h p} \cdot \jump{S} \right) \, \avg{T} - \frac{1}{2} \sum_{F \in \mathcal{F}_I} \int_F \left( \avg{-c_f \mathbf{K} \ \nabla_h p} \cdot \jump{T} \right) \, \avg{S}\\
					& + \frac{1}{2} \sum_{F \in \mathcal{F}} \int_F \jump{-c_f \mathbf{K} \ \nabla_h p}_n \avg{T \, S} - \frac{1}{2} \sum_{\kappa \in \mathcal{T}_h} \int_{\kappa} \divh{\left(-c_f \mathbf{K} \ \nabla_h p \right)} \, T \, S. \\
				\end{aligned}
			\end{equation}
			Then, we conclude the proof by taking $S = T$ in \eqref{eq:RobQSTPE_positivity_3} to infer:
			\begin{equation}
				\label{eq:RobQSTPE_positivity_4}
				\begin{aligned}
					\widetilde{\mathcal{C}}_h (T,p,T) = - \frac{1}{2} \sum_{\kappa \in \mathcal{T}_h} \int_{\kappa} \divh{\left(-c_f \mathbf{K} \ \nabla_h p \right)} \, T^2 + \frac{1}{2} \sum_{F \in \mathcal{F}} \int_F \jump{-c_f \mathbf{K} \ \nabla_h p}_n \avg{T^2}. 
				\end{aligned}
			\end{equation}
		\end{proof}
		\begin{remark}
			\label{rem:skew_ChBh}
			Owing to the definition of the bilinear form $\mathcal{B}_h$ in \eqref{eq:RobQSTPE_bilinear_forms_discr}, we can also express \eqref{eq:RobQSTPE_trasporto} as $$
			\widetilde{\mathcal{C}}_h(T, p, T) = \frac{1}{2} \mathcal{B}_h(T^2, -c_f \, \mathbf{K} \, \nabla_h p ), \quad T, p, \in V_h^{\ell}.
			$$
		\end{remark}
		
		\subsubsection{Stabilization of the trilinear form}
		\label{sec:RobQSTPE_advection_stabilization}
		In this section, we focus on the last two terms appearing in $\widetilde{\mathcal{C}}_h^{\text{stab}}$ \eqref{eq:RobQSTPE_bilinear_forms_discr}. The formulation $\widetilde{\mathcal{C}}_h$ for the transport term may suffer when dealing with high Peclèt number, namely when the Darcy velocity $-c_f(\mathbf{K} \ \nabla_h p_h)$ is large compared to thermal diffusion. 
		In \cite{AntoniettiBonetti2022}, it has been observed that this may be the case when the thermal conductivity coefficient is significantly smaller than the permeability coefficient. In order to circumvent this issue, we introduce two suitable stabilizations.
		
		\medskip
		\textit{Upwind.} First, we introduce an upwinding stabilization following \cite{Ern2009, Houston2002, Ayuso2009}. In principle, this stabilization would consist in a different definition of the average, however is has been proven that it is essentially equivalent to an additional jump-jump stabilization term scaled by an appropriate coefficient \cite{Brezzi2004, DiPietro2012}. We denote by $\widetilde{\mathcal{C}}^{\text{uw}}_h$ the upwinded version of the trilinear form $\widetilde{\mathcal{C}}_h$:
		\begin{equation}
			\label{eq:RobQSTPE_upwinding}
			\begin{aligned}
				\widetilde{\mathcal{C}}^{\text{uw}}_h(T,p,S) = \widetilde{\mathcal{C}}_h(T,p,S) + s^{\text{uw}}_h(T,p,S) = \widetilde{\mathcal{C}}_h(T,p,S) + \sum_{F \in \mathcal{F}_I} \int_F \varpi \ \frac{\left| \avg{ -c_f \ \mathbf{K} \nabla_h p} \cdot \mathbf{n} \right|}{2} \jump{T} \mkern-2.5mu \cdot \mkern-2.5mu \jump{S},
			\end{aligned}
		\end{equation}
		with $T, p, S \in V_h^{\ell}$. In \eqref{eq:RobQSTPE_upwinding}, $\varpi$ is a user-dependent parameter; if we set $\varpi = 1$, then we get the classical upwind fluxed in the finite-volume scheme (this will be our choice in the rest of the work). We remark that, considering the PolyDG formulation of our problem, the use of jump-jump stabilization \eqref{eq:RobQSTPE_upwinding} is very convenient from the coding point of view, as we already have the same kind of stabilization coming from the IP formulation of the bilinear form $\mathcal{A}_h^T$.
		
		\medskip
		\textit{Enforce the inflow condition.} In the limit of vanishing thermal conductivity, we notice that the upwinding correction is not sufficient for ensuring robustness. Our idea, for further stabilizing the scheme, consists in mimicking the imposition of inflow boundary conditions in the hyperbolic case \cite{Suli2000}. Indeed, 
		Dirichlet conditions are weakly enforced through the bilinear form $\mathcal{A}_h^T$ and, as a result, we lose control of them for vanishing $\boldsymbol{\Theta}$. To this aim, we add two boundary terms to $\widetilde{\mathcal{C}}^{\text{uw}}_h(T,p,S)$ and we obtain the definition of the stabilized trilinear form given in \eqref{eq:RobQSTPE_bilinear_forms_discr}:
		\begin{equation}
			\label{eq:RobQSTPE_inflow}
			\begin{aligned}
				\widetilde{\mathcal{C}}^{\text{stab}}_h(T,p,S) = \ & \widetilde{\mathcal{C}}^{\text{uw}}_h(T,p,S) + s^{\text{inflow}}_h(T,p,S) \\
				= \ & \widetilde{\mathcal{C}}^{\text{uw}}_h(T,p,S) + \sum_{F \in \mathcal{F}_B} \int_F \left(-c_f \ \mathbf{K} \nabla_h p \cdot \mathbf{n} \right)^\ominus \, T \, S,
			\end{aligned}
		\end{equation}
		with $T, p, S \in V_h^{\ell}$ and with $(\cdot)^\ominus = \frac12(|\cdot|-(\cdot))$ denoting the negative part operator.
		\begin{remark}
			Given a face $F \in \mathcal{F}_B$ we observe that the coefficients appearing in the integrand function coming from $s^{\mathrm{inflow}}_h(T,p,S)$ read:
			$$
			\left(-c_f \ \mathbf{K} \nabla_h p \cdot \mathbf{n} \right)^\ominus = 
			\begin{cases}
				-\left(-c_f \ \mathbf{K} \nabla_h p \right) \cdot \mathbf{n} & \text{if $\left(-c_f \ \mathbf{K} \nabla_h p \right) \cdot \mathbf{n} < 0$ (inflow)}, \\
				0, & \text{if $\left(-c_f \ \mathbf{K} \nabla_h p \right) \cdot \mathbf{n} \geq 0$ (outflow)}. \\
			\end{cases}
			$$
			Note that the coefficient $-\left(-c_f \ \mathbf{K} \nabla_h p \right) \cdot \mathbf{n}$ is positive when $\left(-c_f \ \mathbf{K} \nabla_h p \right) \cdot \mathbf{n} < 0$. This means that via $s^{\mathrm{inflow}}_h$ we are effectively enforcing the boundary conditions at the inflow boundary.
		\end{remark}
		
		As a corollary to Proposition~\ref{prop:RobQSTPE_trasporto} and Remark ~\ref{rem:skew_ChBh}, we have the following useful result:
		\begin{proposition}
			\label{prop:RobQSTPE_trasporto_uw}
			The trilinear form $\widetilde{\mathcal{C}}^{\text{stab}}_h$ defined in \eqref{eq:RobQSTPE_bilinear_forms_discr} satisfies for any $T, p \in V_h^{\ell}.$:
			\begin{equation}
				\label{eq:RobQSTPE_trasporto_stab}
				\begin{aligned}
					\widetilde{\mathcal{C}}^{\mathrm{stab}}_h(T, p, T) = \ & - \frac{1}{2} \big(\divh{\left( - c_f  \mathbf{K} \nabla_h p \right)}, T^2 \big) + \frac{1}{2} \sum_{F \in \mathcal{F}} \int_F \jump{ -c_f \mathbf{K} \nabla_h p}_n \avg{T^2} \\
					& + \sum_{F \in \mathcal{F}} \int_F \ \frac{\left| \avg{ - c_f \ \mathbf{K} \nabla_h p} \cdot \mathbf{n} \right|}{2} \jump{T}^2 - \sum_{F \in \mathcal{F}_B} \int_F \ \frac{ (- c_f \ \mathbf{K} \nabla_h p) \cdot \mathbf{n}}{2} \, T^2 \\
					= \ & \frac{\mathcal{B}_h(T^2, -c_f \, \mathbf{K} \, \nabla_h p )}{2} +\hspace{-1mm} \sum_{F \in \mathcal{F}_I}\hspace{-0.5mm} \int_F \hspace{-0.5mm}\frac{\left| \avg{ - c_f \mathbf{K} \nabla_h p} \cdot \mathbf{n} \right|}{2} \jump{T}^2 
					+\hspace{-1.5mm} \sum_{F \in \mathcal{F}_B}\hspace{-1mm} \int_F (- c_f \mathbf{K} \nabla_h p \cdot \mathbf{n})^{\ominus} \, T^2,
				\end{aligned}
			\end{equation}
			with $(\cdot)^\ominus = \frac12(|\cdot|-(\cdot))$ denoting the negative part operator.
		\end{proposition}
		
		\subsection{Linearization}
		\label{sec:RobQSTPE_linearization}
		For tackling the non-linear convective term, we introduce a fixed-point iteration algorithm, as in \cite{Brun2018, AntoniettiBonetti2022}. 
		Differently from these works, we design a linearization of the trilinear form $\widetilde{\mathcal{C}}^{\text{stab}}_h(T_h, p_h, S_h)$ based on the replacement of the Darcy velocity with a known function $\boldsymbol{\eta} \in \mathbf{V}_h^{\ell-1}$. Indeed, when solving the $k\textsuperscript{th}$-step of the algorithm, we set $\boldsymbol{\eta} = -c_f \, \mathbf{K} \, \nabla_h (p_h^{k-1})$, with $p_h^{k-1}$ denoting the pressure approximation computed at the $(k-1)\textsuperscript{th}$-step. With this choice, we get better results in terms of robustness.  
		According to this linearization procedure, the trilinear form $\widetilde{\mathcal{C}}^{\text{stab}}_h(T, p, S)$ becomes bilinear and we will refer to it as:
		\begin{equation}
			\label{eq:RobQSTPE_linearized_C}
			\begin{aligned}
				\mathcal{C}^{\text{stab}}_h (T,S) = \ & \big( \boldsymbol{\eta} \cdot \nabla_h T, S \big) - \sum_{F \in \mathcal{F}_I} \int_{F} \left( \avg{\boldsymbol{\eta}} \cdot \jump{T} \right) \avg{S} \\
				& + \sum_{F \in \mathcal{F}_I} \int_F \ \frac{\left| \avg{\boldsymbol{\eta}} \cdot \mathbf{n} \right|}{2} \jump{T} \mkern-2.5mu \cdot \mkern-2.5mu \jump{S} + \sum_{F \in \mathcal{F}_B} \int_F \ \left(\boldsymbol{\eta} \cdot \mathbf{n} \right)^\ominus \, T \, S,
			\end{aligned}
		\end{equation}
		with $T, S \in V_h^{\ell}$.
		\begin{remark}
			We observe that, by setting $\boldsymbol{\eta} = -c_f \, \mathbf{K} \, \nabla_h p_h^{k-1}$, the energy conservation equation \eqref{eq:RobQSTPE_energy_cons} essentially becomes an advection-diffusion equation, where the advective velocity depends on the previously computed pressure field. 
		\end{remark}
		The linearized semi-discrete formulation becomes:
		\textit{find $(\mathbf{u}_h,p_h,T_h,\varphi_h) \in \mathbf{V}_h^{\ell} \times V_h^{\ell} \times V_h^{\ell} \times Q_h^m$ such that}
		\begin{equation}
			\label{eq:RobQSTPE_discrete_weak_form_Dtime_2}
			\begin{aligned}
				\mathcal{M}((p_h, T_h, \varphi_h),(q_h, S_h, \psi_h)) + \mathcal{A}_h^{T}(T_h,S_h) + \mathcal{C}^{\text{stab}}_h(T_h,S_h)
				+ \mathcal{A}_h^{p}(p_h,q_h) + \mathcal{A}_h^{e}(\mathbf{u}_h,\mathbf{v}_h) \\
				- \mathcal{B}_h(\varphi_h,\mathbf{v}_h) +  \mathcal{B}_h(\psi_h,\mathbf{u}_h)
				+  \mathcal{D}_h(\varphi_h,\psi_h)
				= (H,S_h) + (g,q_h) + (\mathbf{f},\mathbf{v}_h).
			\end{aligned}
		\end{equation}
		\textit{for all $(\mathbf{v}_h,q_h,S_h,\psi_h) \in \mathbf{V}_h^{\ell} \times V_h^{\ell} \times V_h^{\ell} \times Q_h^m$}. The linearization algorithm is initialized by taking $\boldsymbol{\eta}=\boldsymbol0$ at the first iteration. Then, the computed pressure is used to update $\boldsymbol{\eta}$ and the linear discrete problem \eqref{eq:RobQSTPE_discrete_weak_form_Dtime_2} is solved again until a prescribed stopping criterion is verified.
		The convergence of this algorithm is established in Section \ref{sec:RobQSTPE_conv_fixedpoint} under suitable requirements on the discrete pressure solution.

\section{Stability analysis}
\label{sec:RobQSTPE_stability}
In this section, we establish a priori estimates for the solution to the discrete problem \eqref{eq:RobQSTPE_discrete_weak_form_Dtime_2} and we investigates the conditions under which the linearization procedure designed yields convergence.

\subsection{Well-posedness of the linearized problem}\label{sec:RobQSTPE_apriori}
For carrying out the analysis of problem \eqref{eq:RobQSTPE_discrete_weak_form_Dtime_2} we need to introduce some further notation. 
The $dG$-norms that will be used in our analysis are defined such that
\begin{equation}
	\label{eq:RobQSTPE_DG_norms}
	\begin{aligned}
		&\|S\|^2_{dG,T} = \|\sqrt{\boldsymbol{\Theta}} \ \nabla_h S\|^2 +\sum_{F\in\mathcal{F}} \|\sigma^{1/2} \jump{S} \ \|_F^2 \quad &&  \forall \ S \in V_h^{\ell},\\ 
		&\|q\|^2_{dG,p} = \|\sqrt{\mathbf{K}} \ \nabla_h q\|^2 + \sum_{F\in\mathcal{F}} \|\xi^{1/2} \jump{q} \ \|_F^2 \quad && \forall \ q \in V_h^{\ell},\\ 
		&\|\mathbf{v}\|^2_{dG,e} = \|\sqrt{2 \mu} \ \boldsymbol{\epsilon}_h(\mathbf{v})\|^2 + \sum_{F\in\mathcal{F}}\|\zeta^{1/2} \jump{\mathbf{v}} \ \|_F^2 \quad && \forall \ \mathbf{v} \in \mathbf{V}_h^{\ell}.
	\end{aligned}
\end{equation}

We can now state the following results that establishes the main key properties of the bilinear forms defined in \eqref{eq:RobQSTPE_bilinear_forms_discr}. The proof follows the lines of \cite[Chapter 4]{DiPietro2012}, \cite[Section 3]{Ern2009}. 
\begin{lemma}
	\label{lem:RobQSTPE_boundcoerc_bil_forms}
	Let Assumptions~\ref{assumption:RobQSTPE_model_problem} and ~\ref{ass:RobQSTPE_mesh_Th1} be satisfied and assume that the parameters $\alpha_1$, $\alpha_2$, and $\alpha_3$ appearing in \eqref{eq:RobQSTPE_stabilization_func} are chosen large enough. Then, the following bounds hold:
	\begin{equation}
		\begin{aligned}
			\mathcal{A}_h^T(T,S) \lesssim \ & \|T\|_{dG,T} \|S\|_{dG,T}, \qquad && \mathcal{A}_h^T(T,T) \gtrsim \|T\|_{dG,T}^2 \qquad &&\forall \ T,S \in V_h^{\ell},\\
			\mathcal{A}_h^p(p,q) \lesssim \ & \|p\|_{dG,p} \|q\|_{dG,p}, \qquad && \mathcal{A}_h^p(p,p) \gtrsim \|p\|_{dG,p}^2 \qquad &&\forall \ p,q \in V_h^{\ell},\\
			\mathcal{A}_h^e(\mathbf{u},\mathbf{v}) \lesssim \ & \|\mathbf{u}\|_{dG,e} \|\mathbf{v}\|_{dG,e}, \qquad && \mathcal{A}_h^e(\mathbf{u},\mathbf{u}) \gtrsim \|\mathbf{u}\|_{dG,e}^2 \qquad &&\forall \ \mathbf{u},\mathbf{v} \in \mathbf{V}_h^{\ell},
		\end{aligned}
	\end{equation}
	where the hidden constants do not depend on the material properties and the discretization parameters.
\end{lemma}

\begin{lemma}
	\label{lem:RobQSTPE_Bh}
	Under Assumption~\ref{ass:RobQSTPE_mesh_Th1}, the following inequalities hold true with hidden constants independent of the model parameters and the mesh size $h$:
	\begin{enumerate}
		\item[(i) ] assuming that the polynomial degrees $\ell$ and $m$  satisfy $\ell+1 \geq m$ and the parameter $\alpha_4$ in \eqref{eq:RobQSTPE_stabilization_func} is large enough, there exists $\mathbb{B}>0$, possibly depending on $\ell$ and $m$, such that
		\begin{equation}
			\label{eq:RobQSTPE_gen_inf_sup}
			\underset{\mathbf{0} \neq \mathbf{v}_h \in \mathbf{V}^{\ell}_h}{\mbox{sup}} \frac{\mathcal{B}_h(\mathbf{v}_h, \varphi_h)}{\|\mathbf{v}_h\|_{DG,e}} + \mathcal{D}_h(\varphi_h,\varphi_h)^{\frac12} \geq \mathbb{B} \|\varphi_h\| \qquad \forall \varphi_h \in Q_h^m;
		\end{equation}
		\item[(ii) ] denoting by $|\cdot|_{dG}$ the broken $H(\mathrm{div})$-seminorm defined such that 
		$$
		|\mathbf{v}_h|_{dG}^2=\|\divh{\mathbf{v}_h}\|^2
		+\sum_{F\in\mathcal{F}} \underset{\kappa \in \{\kappa^+,\kappa^-\}}{\mbox{max}} \left(\ell^2 h_{\kappa}^{-1}\right)\|\jump{\mathbf{v}_h}_n\|_F^2,
		$$
		one has
		\begin{equation}
			\label{eq:RobQSTPE_bnd_Bh}
			\mathcal{B}_h(\mathbf{v}_h, \varphi_h) \lesssim
			|\mathbf{v}_h|_{dG} \|\varphi_h\| \lesssim
			\mu_m^{-\frac12}\|\mathbf{v}_h\|_{dG,e} \|\varphi_h\|
			\qquad\forall\mathbf{v}_h \in \mathbf{V}_h^{\ell}, \ \forall \varphi_h \in Q_h^m;
		\end{equation}
		\item[(iii) ] denoting by $|\cdot|_{dG,\infty}$ the broken seminorm defined such that 
		$$
		|\mathbf{v}_h|_{dG,\infty} = \|\divh{\mathbf{v}_h}\|_{L^\infty(\Omega)}
		+\max_{F\in\mathcal{F}} \underset{\kappa \in \{\kappa^+,\kappa^-\}}{\max} \left(\ell^2 h_{\kappa}^{-1}\right) \|\jump{\mathbf{v}_h}_n\|_{L^\infty(F)},
		$$
		one has
		\begin{equation}
			\label{eq:RobQSTPE_quasiskew_Ch}
			\mathcal{C}_h^{\text{stab}}(T_h, T_h) \gtrsim
			-|\boldsymbol{\eta}|_{dG,\infty} \|T_h\|^2
			\qquad\ \forall T_h \in V_h^\ell;
		\end{equation}
		\item[(iv) ] assuming $\boldsymbol{\eta}\in \mathbf{L}^\infty(\Omega)$, the bilinear form defined in \eqref{eq:RobQSTPE_linearized_C} satisfies
		\begin{equation}
			\label{eq:RobQSTPE_bnd_Ch}
			\mathcal{C}_h^{\text{stab}}(T_h, S_h) \lesssim \theta_m^{-\frac12} \|\boldsymbol{\eta}\|_{\mathbf{L}^\infty(\Omega)}
			\|T_h\|_{dG,T} \|S_h\|
			\qquad T_h, S_h \in V_h^{\ell}.
		\end{equation}
	\end{enumerate}
\end{lemma}
\begin{proof}
	\textit{(i)} For the proof of condition \eqref{eq:RobQSTPE_gen_inf_sup} we refer to \cite[Proposition 3.1]{antonietti2020_stokesDG}.
	
	\medskip\noindent
	\textit{(ii)} Estimate \eqref{eq:RobQSTPE_bnd_Bh} follows from the Cauchy--Schwarz inequality and 
	the trace inverse inequality stated in \eqref{eq:RobQSTPE_trace_inverse_ineq}, namely
	\begin{equation}\label{eq:RobQSTPE_bndB_det}
		\begin{aligned}
			\mathcal{B}_h(\mathbf{v}_h, \varphi_h) 
			&\le \|\divh{\mathbf{v}_h}\|\|\varphi_h\| + \sum_{\kappa\in\mathcal{T}_h} \|\jump{\mathbf{v}_h}_n\|_{\partial\kappa} \ \|\varphi_h\|_{\partial\kappa} \\
			&\lesssim \|\divh{\mathbf{v}_h}\|\|\varphi_h\| 
			+ \left(\sum_{\kappa\in\mathcal{T}_h} \ell^2 h_{\kappa}^{-1} \|\jump{\mathbf{v}_h}_n\|_{\partial\kappa}^2\right)^{\frac12} \ \left(\sum_{\kappa\in\mathcal{T}_h} \ell^{-2} h_{\kappa} \|\varphi_h\|_{\partial\kappa}^2\right)^\frac12 \\
			&\lesssim \|\divh{\mathbf{v}_h}\|\|\varphi_h\|
			+ \left(\sum_{F\in\mathcal{F}} \ \sum_{\kappa\in \{\kappa^+,\kappa^-\}} \ell^2 h_{\kappa}^{-1} \|\jump{\mathbf{v}_h}_n\|_F^2\right)^\frac12 \|\varphi_h\| \\
			&\lesssim \left(\|\divh{\mathbf{v}_h}\|^2 + \sum_{F\in\mathcal{F}} \underset{\kappa \in \{\kappa^+,\kappa^-\}}{\mbox{max}} \left(\ell^2 h_{\kappa}^{-1}\right)\|\jump{\mathbf{v}_h}_n\|_F^2\right)^\frac12 \|\varphi_h\| = |\mathbf{v}_h|_{dG} \|\varphi_h\|.
		\end{aligned}
	\end{equation}
	The second inequality in \eqref{eq:RobQSTPE_bnd_Bh} directly follows by observing that $\|\divh{\mathbf{v}_h}\| \le \|  \boldsymbol{\epsilon}_h(\mathbf{v})\|$ and,  for all $F\in\mathcal{F}$, $\|\jump{\mathbf{v}_h}_n\|_F \le \|\jump{\mathbf{v}_h}\|_F$.
	
	\medskip\noindent
	\textit{(iii)} We proceed with the proof of \eqref{eq:RobQSTPE_quasiskew_Ch}. Owing to Proposition \ref{prop:RobQSTPE_trasporto_uw}, we can write 
	$$
	\begin{aligned}
		\mathcal{C}^{\text{stab}}_h (T_h,T_h) & = 
		\frac{\mathcal{B}_h(T_h^2, \boldsymbol{\eta})}{2} + \sum_{F \in \mathcal{F}_I} \int_F \frac{\left| \avg{\boldsymbol{\eta}} \cdot \mathbf{n} \right|}{2} \jump{T_h}^2 
		+ \sum_{F\in\mathcal{F}_B} \int_F (\boldsymbol{\eta} \cdot \mathbf{n})^{\ominus} \, T_h^2 \\
		& \ge \frac12\left(
		\big(-\divh{\boldsymbol{\eta}}, T_h^2 \big) + 
		\sum_{F\in\mathcal{F}}\int_F \jump{\boldsymbol{\eta}}_n \avg{T_h^2}\right).
	\end{aligned}
	$$
	Thus, applying the H\"older inequality, recalling the definition of the $|\cdot|_{dG,\infty}$ seminorm, and using again the discrete trace inverse inequality, we have
	$$
	\begin{aligned}
		\mathcal{C}^{\text{stab}}_h (T_h,T_h) & \ge 
		\frac{-\|\divh{\boldsymbol{\eta}}\|_{L^\infty(\Omega)}}2 \|T_h\|^2 -\frac12\left(\max_{F\in\mathcal{F}} 
		\underset{\kappa \in \{\kappa^+,\kappa^-\}}{\max} \frac{\ell^2}{h_{\kappa}} \|\jump{\mathbf{v}_h}_n\|_{L^\infty(F)}
		\hspace{-0.5mm}\right)\hspace{-0.5mm}
		\sum_{F\in\mathcal{F}} \int_F \min_{\kappa=\kappa^\pm} \frac{h_{\kappa}}{\ell^2} \avg{T_h^2} \\
		& \ge \frac{-|\boldsymbol{\eta}|_{dG,\infty}}2
		\left(\|T_h\|^2 +\sum_{F\in\mathcal{F}} \sum_{\kappa \in \{\kappa^+,\kappa^-\}} \int_F \frac{h_{\kappa}}{\ell^2} \ T_h^2\right) \\
		& \ge \frac{-|\boldsymbol{\eta}|_{dG,\infty}}2
		\left(\|T_h\|^2 +\sum_{\kappa\in\mathcal{T}_h} \frac{h_{\kappa}}{\ell^2} \|T_h\|^2_{\partial\kappa}\right)
		\ge -(C_{\rm tr} +1)|\boldsymbol{\eta}|_{dG,\infty} \|T_h\|^2.
	\end{aligned}
	$$
	
	\medskip\noindent
	\textit{(iv)}
	The boundedness of $\mathcal{C}_h^{\text{stab}}$ in \eqref{eq:RobQSTPE_bnd_Ch} is obtained in a similar way. Applying the H\"older and triangle inequality we obtain
	$$
	\begin{aligned}
		\mathcal{C}^{\text{stab}}_h (T_h,S_h) &\le 
		\|\boldsymbol{\eta}\|_{\mathbf{L}^\infty(\Omega)}  \|\nabla_h T_h\| \|S_h\|  
		+ \|\boldsymbol{\eta}\|_{\mathbf{L}^\infty(\Omega)}  
		\left(\sum_{F \in \mathcal{F}} \| \jump{T_h}\|_F  \Big(\|\avg{S_h}\|_F+\frac12\|\jump{S_h}\|_F\Big)\right)  \\
		&\le \|\boldsymbol{\eta}\|_{\mathbf{L}^\infty(\Omega)}  \theta_m^{-\frac12}\|\sqrt{\boldsymbol{\Theta}}\nabla_h T_h\| \|S_h\|  
		+ \|\boldsymbol{\eta}\|_{\mathbf{L}^\infty(\Omega)}  
		\left(\sum_{F \in \mathcal{F}}\ \sum_{\kappa\in \{\kappa^+,\kappa^-\}} \| \jump{T_h}\|_F  \|S_h\|_F \right).
	\end{aligned}
	$$
	Then, proceeding as in \eqref{eq:RobQSTPE_bndB_det} and recalling the definition of the stabilization function $\sigma$ in \eqref{eq:RobQSTPE_stabilization_func}, it is inferred that
	$$
	\mathcal{C}^{\text{stab}}_h (T_h,S_h) \lesssim 
	\theta_m^{-\frac12} \|\boldsymbol{\eta}\|_{\mathbf{L}^\infty(\Omega)}\|S_h\|   \left(\|\sqrt{\boldsymbol{\Theta}}\nabla_h T_h\|^2 + 
	\sum_{F \in \mathcal{F}}\|\sigma^\frac12 \jump{T_h}\|_F ^2
	\right)^\frac12 \lesssim \theta_m^{-\frac12} \|\boldsymbol{\eta}\|_{\mathbf{L}^\infty(\Omega)}
	\|T_h\|_{dG,T} \ \|S_h\|.
	$$
\end{proof}
\begin{remark}
	The boundedness of the transport velocity in the $|\boldsymbol{\eta}|_{dG,\infty}<\infty$ holds whenever $\boldsymbol{\eta}\in\mathbf{V}_h^\ell$ is the approximation an $H(div)$-regular vector fields such that its divergence belongs to $L^\infty(\Omega)$. In the case of slightly compressible fluids, the divergence of the velocity field is usually close to zero.
\end{remark}
The next Lemma is instrumental for the derivation of an a priori stability estimate, that is robust with respect to limit cases of the model parameters.
\begin{lemma}
	\label{lem:RobQSTPE_Mh}
	Let Assumptions~\ref{assumption:RobQSTPE_model_problem} and ~\ref{ass:RobQSTPE_mesh_Th1} be satisfied and assume that $\ell+1 \geq m$ and $\alpha_4$ in \eqref{eq:RobQSTPE_stabilization_func} is large enough. Additionally, assume that at least two of the following four requirements are verified:
	$$
	\text(i)\; b_0 \ge b_m > 0 ,\quad
	\text(ii)\; a_0 -b_0 \ge a_m >0 ,\quad
	\text(iii)\; c_0 -b_0 \ge c_m > 0  ,\quad
	\text(iv)\; \lambda <\lambda_{M} <\infty.
	$$
	Then, there exist strictly positive constants $a_1,b_1$, and $c_1$ such that
	\begin{equation} \label{eq:RobQSTPE_posM}
		a_1 \|T_h\|^2 + b_1 \|\varphi_h\|^2 + c_1 \|p_h\|^2 
		\lesssim \mathcal{M}\big((p_h,T_h,\varphi_h),(p_h,T_h,\varphi_h)\big) + \mathcal{D}_h(\varphi_h,\varphi_h) + ||\mathbf{u}_h||_{DG,e}^2 + \|\mathbf{f}\|^2.
	\end{equation}
\end{lemma}
\begin{proof}
	First, we derive a bound for the total pressure variable $\varphi$ using the inf-sup property of $\mathcal{B}_h$.
	Taking $(\mathbf{v_h}, 0, 0, 0)$ as test function in the linearized discrete problem \eqref{eq:RobQSTPE_discrete_weak_form_Dtime_2}, one has
	$$
	\mathcal{B}_h(\varphi_h, \mathbf{v}_h) = 
	\mathcal{A}_h^e(\mathbf{u}_h, \mathbf{v}_h) - (\mathbf{f}, \mathbf{v}_h).
	$$
	Plugging the previous identity into \eqref{eq:RobQSTPE_gen_inf_sup} and using Lemma \eqref{lem:RobQSTPE_boundcoerc_bil_forms} followed by a discrete Poincaré--Korn inequality \cite{Botti2020_korn, Brenner2004}, we infer
	\begin{equation}
		\label{eq:RobQSTPE_bndphi_robust}
		\mathbb{B}^2\|\varphi_h\|^2 \lesssim 
		\mathcal{D}_h(\varphi_h, \varphi_h) + \left(
		\underset{\mathbf{0} \neq \mathbf{v}_h \in \mathbf{V}^{\ell}_h}{\mbox{sup}} \frac{\mathcal{A}_h^e(\mathbf{u}_h, \mathbf{v}_h) - (\mathbf{f}, \mathbf{v}_h)}{||\mathbf{v}_h||_{DG,e}} \right)^2 \lesssim
		\mathcal{D}_h(\varphi_h, \varphi_h) + ||\mathbf{u}_h||_{DG,e}^2 + \mu_m^{-1}||\mathbf{f}||^2.
	\end{equation}
	Now, we observe that we can express both the temperature and the pore pressure discrete fields as a linear combination of the terms in $\mathcal{M}\big((p_h,T_h,\varphi_h),(p_h,T_h,\varphi_h)\big) + \mathbb{B}^2\|\varphi_h\|$, namely
	$$
	\begin{aligned}
		T_h &= \gamma_1\big[(a_0-b_0)^\frac12 T_h\big] +\gamma_2 \big[b_0^\frac12 (p_h - T_h)\big] +\gamma_3 \big[(c_0-b_0)^\frac12 p_h\big] + \gamma_4
		\big[\lambda^{-\frac12}(\varphi_h+\alpha p_h+\beta T_h)\big]+\gamma_5 \big[\mathbb{B}\varphi_h\big],
		\\
		p_h &= \delta_1\big[(a_0-b_0)^\frac12 T_h\big] + \delta_2\big[b_0^\frac12 (p_h - T_h)\big] +\delta_3 \big[(c_0-b_0)^\frac12 p_h\big] +\delta_4 
		\big[\lambda^{-\frac12}(\varphi_h+\alpha p_h+\beta T_h)\big]+\delta_5\big[\mathbb{B}\varphi_h\big],
	\end{aligned}
	$$
	Therefore, the bound in \eqref{eq:RobQSTPE_posM} follows from the triangle inequality and \eqref{eq:RobQSTPE_bndphi_robust}.
\end{proof}

We are now ready to establish the stability estimates for the linearized thermo-hydro-mechanics problem. 
In particular, we aim at deriving a priori bounds which are independent of the thermal conductivity $\boldsymbol{\Theta}$ and hydraulic mobility $mathbf{K}$. Instead, for obtaining stability bounds independent of the dilatation parameter $\lambda$ and the coefficients $a_0, b_0$, and $c_0$ (but depending on $\boldsymbol{\Theta}$ and $\mathbf{K}$), we can follow the approach of \cite[Section 4]{AntoniettiBonetti2022}. Indeed, the next Theorem can be seen as complementary to \cite[Theorem 4.5]{AntoniettiBonetti2022} to infer the robustness of the scheme with respect to all possible physical regimes, including quasi-incompressible media, vanishing storage coefficients, and low-permeable media.
\begin{theorem} \label{thm:RobQSTPE_stab_est}
	Let the assumptions of Lemmata \ref{lem:RobQSTPE_boundcoerc_bil_forms}, \ref{lem:RobQSTPE_Bh}, and \ref{lem:RobQSTPE_Mh} be satisfied. Let the transport velocity $\boldsymbol{\eta}\in\mathbf{V}^\ell_h$ be such that
	\begin{equation}\label{eq:RobQSTPE_stab_condition1}
		|\boldsymbol{\eta} |_{dG,\infty} \lesssim a_1,
	\end{equation}
	with $a_1>0$ defined in Lemma \ref{lem:RobQSTPE_Mh} and hidden constant independent of $\mathbf{K}$ and $\boldsymbol{\Theta}$.
	Then, the solution $(\mathbf{u}_h,p_h,T_h,\varphi_h) \in \mathbf{V}_h^{\ell} \times V_h^{\ell} \times V_h^{\ell} \times Q_h^m$ to problem \eqref{eq:RobQSTPE_discrete_weak_form_Dtime_2} satisfies the a-priori bound
	\begin{equation}\label{eq:RobQSTPE_a-priori}
		a_1 \|T_h\|^2 + b_1 \|\varphi_h\|^2 + c_1 \|p_h\|^2 + \|\mathbf{u}_h\|_{dG,e}^2 + \|T_h \|_{dG,T}^2 + \|p_h \|_{dG,p}^2
		\lesssim \| H\|^2 + \| g\|^2 + \| \mathbf{f}\|^2,
	\end{equation}
	where the hidden constant is independent of the conductivity tensors $\boldsymbol{\Theta}$, $\mathbf{K}$, and the mesh size $h$.
\end{theorem}
\begin{remark}
	We point out that \eqref{eq:RobQSTPE_stab_condition1} is not too restrictive in the case of the numerical solution of quasi-static thermo-hydro-mechanical problems, which is our main focus. As mentioned in Section \ref{sec:RobQSTPE_model_problem}, the discrete problem \eqref{eq:RobQSTPE_discrete_weak_form_Dtime_2} corresponds to one step of an implicit time advancing scheme in which $\mathbf{K} = \delta t \tilde{\mathbf{K}}$. Thus, when the transport velocity is set as $\boldsymbol{\eta} = -c_f \mathbf{K}\nabla_h(p_h^{k-1})$, condition \eqref{eq:RobQSTPE_stab_condition1} becomes
	$$
	\delta t \lesssim a_1 \big|-c_f \tilde{\mathbf{K}}\nabla_h(p_h^{m-1})\big|_{dG,\infty}^{-1},
	$$
	meaning that it is sufficient to select the time-step $\delta t$ small enough.
\end{remark}
\begin{proof}
	Taking $(\mathbf{v}_h, q_h,S_h,\psi_h) =(\mathbf{u}_h,p_h,T_h,\varphi_h)$ in \eqref{eq:RobQSTPE_discrete_weak_form_Dtime_2} and using Lemma \ref{lem:RobQSTPE_boundcoerc_bil_forms} and Lemma \ref{lem:RobQSTPE_Mh}, we obtain
	\begin{equation}
		\label{eq:RobQSTPE_energy_balance}
		\begin{aligned}
			a_1 &\|T_h\|^2 + b_1 \|\varphi_h\|^2 + c_1 \|p_h\|^2 + \|\mathbf{u}_h\|_{dG,e}^2 + \|T_h \|_{dG,T}^2 + \|p_h \|_{dG,p}^2  
			\\
			&\lesssim\mathcal{M}((p_h, T_h, \varphi_h),(p_h, T_h, \varphi_h))+  \mathcal{D}_h(\varphi_h,\varphi_h) + \|\mathbf{u}_h\|_{dG,e}^2 + \|T_h \|_{dG,T}^2 + \|p_h \|_{dG,p}^2 + \|\mathbf{f}\|^2
			\\
			&\lesssim\mathcal{M}((p_h, T_h, \varphi_h),(p_h, T_h, \varphi_h))+  \mathcal{D}_h(\varphi_h,\varphi_h) + \mathcal{A}_h^{T}(T_h,T_h) + \mathcal{A}_h^{p}(p_h,p_h) + \mathcal{A}_h^{e}(\mathbf{u}_h,\mathbf{u}_h) + \|\mathbf{f}\|^2
			\\
			& = (H,T_h) + (g,p_h) + (\mathbf{f},\mathbf{u}_h) + \|\mathbf{f}\|^2 - \mathcal{C}^{\text{stab}}_h(T_h,T_h).
		\end{aligned}
	\end{equation}
	We now proceed to bound the right-hand side of the previous inequality. Using the Cauchy--Schwarz and Young inequalities and applying the discrete Poincaré--Korn inequality \cite{Botti2020_korn} to bound the $L^2$-norm of the displacement, it is inferred that
	$$
	(H,T_h) + (g,p_h) + (\mathbf{f},\mathbf{u}_h) \le 
	\frac{a_1}{4\varepsilon}\| T_h\|^2 + \frac\varepsilon{a_1} \| H\|^2 + 
	\frac{c_1}{2\varepsilon}\| p_h\|^2 + \frac\varepsilon{2c_1} \| g\|^2
	+ \frac1{2\varepsilon}\|\mathbf{u}_h\|_{dG,e}^2 + \frac{C_{\rm K}\varepsilon}{2\mu_m} \| \mathbf{f}\|^2,
	$$
	with $C_{\rm K}>0$ independent of the material properties and the discretization parameters.
	Moreover, owing to point \textit{(iii)} in Lemma \ref{lem:RobQSTPE_Bh} and recalling assumption \eqref{eq:RobQSTPE_stab_condition1}, we obtain
	$$
	-\mathcal{C}^{\text{stab}}_h(T_h,T_h) \lesssim 
	|\boldsymbol{\eta}|_{dG,\infty} \|T_h\|^2
	\lesssim \frac{a_1}{4}\|T_h\|^2.
	$$
	Plugging the two previous bounds in \eqref{eq:RobQSTPE_energy_balance} and fixing $\varepsilon>0$ equal to the product of the two hidden constants that realizes the inequalities in \eqref{eq:RobQSTPE_energy_balance}, we are led to the conclusion.
\end{proof}

\subsection{Convergence of the fixed point algorithm}
\label{sec:RobQSTPE_conv_fixedpoint}

In order to prove the convergence of the linearization procedure, we show that the difference of the approximate solutions at two successive iterations defines a contracting sequence. 

Let $(\mathbf{u}_h^{k+1},p_h^{k+1},T_h^{k+1},\varphi_h^{k+1})$ and $(\mathbf{u}_h^{k},p_h^{k},T_h^{k},\varphi_h^{k})$ be the solutions to \eqref{eq:RobQSTPE_discrete_weak_form_Dtime_2} at the $(k+1)^{\mathrm{th}}$ and $k^{\mathrm{th}}$ iterations, respectively. The transport velocity at the $(k+1)^{\mathrm{th}}$ step is given by  $\boldsymbol{\eta}= -c_f\mathbf{K}\nabla_h(p_h^{k})$.
For all $k\ge 1$, we define
$$
\mathbf{e}_{\mathbf{u}}^k = \mathbf{u}_h^{k+1}-\mathbf{u}_h^k, \qquad
e_p^k = p_h^{k+1} - p_h^{k}, \qquad
e_T^k = T_h^{k+1} - T_h^{k}, \qquad
e_\varphi^k = \varphi_h^{k+1} - \varphi_h^{k}.
$$
Then, it can be observed that the differences $(\mathbf{e}_{\mathbf{u}}^k, e_p^k, e_T^k, e_\varphi^k)$ solve the problem 
\begin{equation}
	\label{eq:RobQSTPE_errors_equations}
	\begin{aligned}
		& \mathcal{M}((e_p^k, e_T^k, e_\varphi^k),(q_h, S_h, \psi_h)) + \mathcal{A}_h^{T}(e_T^k,S_h) + \widetilde{\mathcal{C}}^{\text{stab}}_h(e_T^k,p_h^k,S_h)
		+ \mathcal{A}_h^{p}(e_p^k,q_h) + \mathcal{A}_h^{e}(\mathbf{e}_{\mathbf{u}}^k,\mathbf{v}_h) \\
		&\quad - \mathcal{B}_h(e_\varphi^k,\mathbf{v}_h) +  \mathcal{B}_h(\psi_h,\mathbf{e}_{\mathbf{u}}^k)
		+  \mathcal{D}_h(e_\varphi^k,\psi_h)
		= -\widetilde{\mathcal{C}}^{\text{stab}}_h(T_h^{k},e_p^{k-1},S_h),
	\end{aligned}
\end{equation}
with trilinear form $\widetilde{\mathcal{C}}^{\text{stab}}_h$ defined in \eqref{eq:RobQSTPE_bilinear_forms_discr}. In the following theorem we investigate the conditions under which the fixed-point iterative method converges. To do so, we need to introduce the $\|\cdot\|_{dG,\infty}$-norm of functions belonging to $V_h^\ell$. Similarly to $|\cdot|_{dG,\infty}$, the definition of $\|\cdot\|_{dG,\infty}$ is given by
$$
\|S_h\|_{dG,\infty} = \|\nabla_h{S_h}\|_{\mathbf{L}^\infty(\Omega)}
+\max_{F\in\mathcal{F}} \underset{\kappa \in \{\kappa^+,\kappa^-\}}{\max} 
\left(\ell^2 h_{\kappa}^{-1}\right) \|\jump{S_h}\|_{\mathbf{L}^\infty(F)},
\qquad\forall\ S_h\in V_h^\ell.
$$
\begin{theorem}
	Let the assumptions of Theorem \ref{thm:RobQSTPE_stab_est} hold. Additionally, assume that 
	\begin{equation}\label{eq:RobQSTPE_stab_condition2}
		\|T_h^k\|_{dG,\infty} \lesssim \sqrt{a_1} c_{fM}^{-1} \qquad \forall k\ge 1,
	\end{equation}
	with $a_1>0$ defined in Lemma \ref{lem:RobQSTPE_Mh} and hidden constant independent of $\mathbf{K}$ and $\boldsymbol{\Theta}$. Then, the linearization strategy defined in Section \ref{sec:RobQSTPE_linearization} converges, namely
	$$
	\mathbf{V}_h^\ell\times V_h^\ell\times V_h^\ell\times Q_h^m \ni
	(\mathbf{e}_{\mathbf{u}}^k,e_p^k,e_T^k,e_\varphi^k)\to\mathbf{0}\;\text{ as }\;k\to\infty.
	$$
\end{theorem}
\begin{proof}
	First, we observe that using the inf-sup condition \eqref{eq:RobQSTPE_gen_inf_sup} and taking $(\mathbf{v}_h, q_h, S_h, \psi_h) = (\mathbf{e}_{\mathbf{u}}^k, 0, 0, 0)$ in \eqref{eq:RobQSTPE_errors_equations} we can bound the $L^2$-norm of $e_\varphi^k$ as 
	\begin{equation}\label{eq:RobQSTPE_proof_fixedpoint1}
		\mathbb{B}^2\|e_\varphi^k\|^2 \lesssim 
		\mathcal{D}_h(e_\varphi^k, e_\varphi^k) + \left(
		\underset{\mathbf{0} \neq \mathbf{v}_h \in \mathbf{V}^{\ell}_h}{\mbox{sup}} \frac{\mathcal{A}_h^e(\mathbf{e}_{\mathbf{u}}^k, \mathbf{v}_h)}{||\mathbf{v}_h||_{DG,e}} \right)^2 \lesssim
		\mathcal{D}_h(e_\varphi^k, e_\varphi^k) + \|\mathbf{e}_{\mathbf{u}}^k\|_{DG,e}^2.
	\end{equation}
	Then, using the inequality in \eqref{eq:RobQSTPE_quasiskew_Ch} under assumption \eqref{eq:RobQSTPE_stab_condition1}, we obtain
	$$
	-\frac{a_1}2 \|e_T^k\|^2 \le
	-(C_{\rm tr} +1) \big|c_f\mathbf{K}\nabla_h(p_h^{k})\big|_{dG,\infty} \|e_T^k\|^2 \le    \widetilde{\mathcal{C}}^{\text{stab}}_h(e_T^k,p_h^k,e_T^k). 
	$$
	Using the previous estimate and Lemma \ref{lem:RobQSTPE_boundcoerc_bil_forms} and \ref{lem:RobQSTPE_Mh} together with \eqref{eq:RobQSTPE_proof_fixedpoint1}, we infer
	$$
	\begin{aligned}
		&\frac{a_1}2 \|e_T^k\|^2 + b_1 \|e_\varphi^k\|^2 + c_1 \|e_p^k\|^2 + \|\mathbf{e}_{\mathbf{u}}^k\|_{dG,e}^2 + \|e_T^k \|_{dG,T}^2 + \|e_p^k \|_{dG,p}^2 \le
		\\
		&a_1 \|e_T^k\|^2 + b_1 \|e_\varphi^k\|^2 + c_1 \|e_p^k\|^2 + \|\mathbf{e}_{\mathbf{u}}^k\|_{dG,e}^2 + \|e_T^k \|_{dG,T}^2 + \|e_p^k \|_{dG,p}^2 + \widetilde{\mathcal{C}}^{\text{stab}}_h(e_T^k,p_h^k,e_T^k)
		\lesssim
		\\
		& \mathcal{M}((e_p^k, e_T^k, e_\varphi^k),(e_p^k, e_T^k, e_\varphi^k)) +  \mathcal{D}_h(e_\varphi^k,e_\varphi^k) + \mathcal{A}_h^{e}(\mathbf{e}_{\mathbf{u}}^k,\mathbf{e}_{\mathbf{u}}^k) +\mathcal{A}_h^{T}(e_T^k,e_T^k) 
		+ \mathcal{A}_h^{p}(e_p^k,e_p^k) +\widetilde{\mathcal{C}}^{\text{stab}}_h(e_T^k,p_h^k,e_T^k).
	\end{aligned}
	$$
	Therefore, taking $(\mathbf{v}_h, q_h, S_h, \psi_h) = (\mathbf{e}_{\mathbf{u}}^k, e_p^k, e_T^k, e_\varphi^k)$ in \eqref{eq:RobQSTPE_errors_equations} and recalling the definition of the trilinear form $\widetilde{\mathcal{C}}^{\text{stab}}_h$, we have
	\begin{equation}\label{eq:RobQSTPE_conv_term0}
		\frac{a_1}2 \|e_T^k\|^2 + b_1 \|e_\varphi^k\|^2 + c_1 \|e_p^k\|^2 + \|\mathbf{e}_{\mathbf{u}}^k\|_{dG,e}^2 + \|e_T^k \|_{dG,T}^2 + \|e_p^k \|_{dG,p}^2 \lesssim
		-\widetilde{\mathcal{C}}^{\text{stab}}_h(T_h^{k},e_p^{k-1},e_T^k)
		= \mathcal{I}_1 + \mathcal{I}_2,
	\end{equation}
	with $\mathcal{I}_1$ and $\mathcal{I}_2$ defined as
	$$
	\begin{aligned}
		\mathcal{I}_1 &= \big( c_f \big(\mathbf{K} \ \nabla_h e_p^{k-1}\big) \cdot \nabla_h T_h^{k}, e_T^k \big), 
		\\ 
		\mathcal{I}_2 &= 
		\sum_{F \in \mathcal{F}_I} \int_{F}\big(\avg{\smash{-c_f \mathbf{K} \nabla_h e_p^{k-1}}} \cdot \jump{\smash{T_h^{k}}}  \avg{\smash{e_T^k}}
		- \frac{|\avg{-c_f\mathbf{K}\nabla_h e_p^{k-1}} \cdot \mathbf{n}|}2 \jump{\smash{T_h^k}} \cdot {\jump{\smash{e_T^k}}} \big)
		\\
		&\quad\, +\sum_{F \in \mathcal{F}_B} \int_F (c_f \mathbf{K} \nabla_h e_p^{k-1} \cdot \mathbf{n})^{\ominus} \ T_h^{k}\ e_T^k.
	\end{aligned}
	$$
	Using the H\"older and Young inequalities, we get
	\begin{equation}\label{eq:RobQSTPE_conv_term1}
		\mathcal{I}_1 \le 
		c_{fM} \|\nabla_h T_h^{k}\|_{\mathbf{L}^\infty(\Omega)}
		\|e_p^{k-1} \|_{dG,p} \|e_T^k\|\le 
		\frac{a_1}{4\varepsilon}\|e_T^k\|^2 + 
		\frac{\varepsilon c_{fM}^2}{a_1} \|T_h^{k}\|_{dG,\infty}^2 \|e_p^{k-1}\|_{dG,p}^2.
	\end{equation}
	For what concern the term $\mathcal{I}_2$, we apply the H\"older, and discrete trace-inverse, and Young inequalities to obtain
	\begin{equation}\label{eq:RobQSTPE_conv_term2}
		\begin{aligned}
			\mathcal{I}_2 &\le
			\sum_{F \in \mathcal{F}} \int_{F}  
			|\avg{\smash{-c_f\mathbf{K}\nabla_h e_p^{k-1}}}\cdot\mathbf{n}|\; |\jump{\smash{T_h^k}}|\;
			\frac{2\left|\avg{\smash{e_T^k}}\right| + \left|\jump{\smash{e_T^k}}\right|}2
			\\
			&\le \sum_{F\in\mathcal{F}}\sum_{\kappa\in\{\kappa^+,\kappa^-\}} 
			\|c_f  \mathbf{K} \nabla_h e_p^{k-1}\|_{F} \
			\|\jump{\smash{T_h^k}}\|_{\mathbf{L}^\infty(F)} \ \|e_T^k\|_F
			\\
			&\le \left(\max_{F\in\mathcal{F}} 
			\underset{\kappa \in \{\kappa^+,\kappa^-\}}{\max} \frac{\ell^2}{h_{\kappa}} \|\jump{\smash{T_h^k}}\|_{\mathbf{L}^\infty(F)}
			\hspace{-0.5mm}\right)\hspace{-0.5mm}
			\sum_{\kappa\in\mathcal{T}_h}
			\frac{h_{\kappa}^{1/2}}{\ell} \|c_f  \mathbf{K} \nabla_h e_p^{k-1}\|_{\partial\kappa} \
			\frac{h_{\kappa}^{1/2}}{\ell} \|e_T^k\|_{\partial\kappa}
			\\
			&\le \left(\max_{F\in\mathcal{F}} 
			\underset{\kappa \in \{\kappa^+,\kappa^-\}}{\max} \frac{\ell^2}{h_{\kappa}} \|\jump{\smash{T_h^k}}\|_{\mathbf{L}^\infty(F)}
			\hspace{-0.5mm}\right) C_{\rm tr}^2
			\left(\sum_{\kappa\in\mathcal{T}_h} \|c_f\mathbf{K}\nabla_h e_p^{k-1}\|_{\kappa}^2 \right)^{\frac12}
			\left(\sum_{\kappa\in\mathcal{T}_h} \|e_T^k\|_{\kappa}^2 \right)^{\frac12}
			\\
			&\le C_{\rm tr}^2 c_{fM} \|T_h^{k}\|_{dG,\infty}
			\|e_p^{k-1} \|_{dG,p} \|e_T^k\|
			\le 
			\frac{a_1}{4\varepsilon}\|e_T^k\|^2 + 
			\frac{\varepsilon C_{\rm tr}^4 c_{fM}^2}{a_1} \|T_h^{k}\|_{dG,\infty}^2 
			\|e_p^{k-1}\|_{dG,p}^2.
		\end{aligned}
	\end{equation}
	Hence, plugging \eqref{eq:RobQSTPE_conv_term1} and \eqref{eq:RobQSTPE_conv_term2} into \eqref{eq:RobQSTPE_conv_term0} and fixing $\varepsilon$ equal to the hidden constant in \eqref{eq:RobQSTPE_conv_term0}, yields
	$$
	\|e_p^k \|_{dG,p}^2 \lesssim (C_{\rm tr}^4+1) c_{fM}^2 a_1^{-1} \|T_h^{k}\|_{dG,\infty}^2 \|e_p^{k-1}\|_{dG,p}^2.
	$$
	As a result of condition \eqref{eq:RobQSTPE_stab_condition2}, the previous inequality shows that the map $(\mathbf{e}_{\mathbf{u}}^{k-1}, e_p^{k-1}, e_T^{k-1}, e_\varphi^{k-1})\to(\mathbf{e}_{\mathbf{u}}^k, e_p^k, e_T^k, e_\varphi^k)$ is a contraction. Thus, the conclusion follows by applying Banach fixed-point theorem.
\end{proof}

\section{Numerical results}
\label{sec:RobQSTPE_numerical_result}
The aim of this section is to assess the performance of the scheme in terms of accuracy and robustness. 

In \cite[Section 6.2]{AntoniettiBonetti2022} it has been shown that the fixed-point algorithm iterating on the temperature gradient may not converge for advection-dominated regimes, i.e. when hydraulic mobility is much greater than thermal conductivity. In this section we analyze in detail the new linearization we propose according to the stabilization correction we apply in case of advection-dominated regimes, with a special focus on the case of vanishing thermal conductivity. First, we look at the convergence properties of the method. Then, we inspect the numerical robustness with respect to the model's coefficients; to this aim, we focus on the case of very small $\boldsymbol{\Theta}$, but being careful of not losing robustness with respect to the other sensible parameters of the model, namely $a_0$, $b_0$, $c_0$, and $\mathbf{K}$. We compare different linearization schemes, that are denoted in the following way: consider the solution of the $m^{th}$-iteration of the fixed-point algorithm, we compare the following constructions of the linearized trilinear form  $\mathcal{C}^*_h(T_h, p_h, S_h)$: 
\begin{equation}
	\begin{aligned}
		\mathcal{C}^{\text{old}}_h (p^m,S) = \ & \big( -c_f \left(\mathbf{K} \ \nabla_h p^{m}\right) \cdot \nabla_h T^{m-1}, S \big), \\
		\mathcal{C}^{\text{vol}}_h (T^m,S) = \ & \big( -c_f \left(\mathbf{K} \ \nabla_h p^{m-1}\right) \cdot \nabla_h T^m, S \big), \\
		\mathcal{C}_h (T^m,S) = \ & \big( -c_f \left(\mathbf{K} \ \nabla_h p^{m-1}\right) \cdot \nabla_h T^m, S \big) - \sum_{F \in \mathcal{F}_I} \int_{F} \left( \wavg{ - c_f \ \mathbf{K} \nabla_h p^{m-1} } \cdot \jump{T^m} \right) \wavg{S}, \\
		\mathcal{C}^{\text{stab}}_h (T^m,S) = \ & \big( -c_f \left(\mathbf{K} \ \nabla_h p^{m-1}\right) \cdot \nabla_h T^m, S \big) - \sum_{F \in \mathcal{F}_I} \int_{F} \left( \wavg{ - c_f \ \mathbf{K} \nabla_h p^{m-1} } \cdot \jump{T^m} \right) \wavg{S} \\
		+ & \sum_{F \in \mathcal{F}} \int_F \ \frac{\left| \wavg{ - c_f \ \mathbf{K} \nabla_h p^{m-1}} \cdot \mathbf{n} \right|}{2} \jump{T^m} \mkern-2.5mu \cdot \mkern-2.5mu \jump{S} - \sum_{F \in \mathcal{F}_B} \int_F \ \frac{ (- c_f \ \mathbf{K} \nabla_h p^{m-1}) \cdot \mathbf{n}}{2} \, T^m \, S.
	\end{aligned}
\end{equation}
For the details of the iterative algorithm exploiting $\mathcal{C}^{\text{old}}_h$, see \cite{AntoniettiBonetti2022, Brun2019}.

In all the numerical tests, we consider a PolyDG-WSIP spatial discretization. The sequence of two-dimensional polygonal Voronoi meshes has been generated through the \texttt{Polymesher} algorithm \cite{Talischi2012}. Last, all the penalty coefficients $\alpha_i$, $i=1,\dots,4$ in \eqref{eq:RobQSTPE_stabilization_func} are set equal to $10$.
Thanks to the introduction of the total pressure--total pressure stabilization, we can use equal order approximations for the four unknowns of the problem. Thus, in every test we set $\ell = m$ and, for the sake of simplicity, we make use only of the symbol $\ell$ to denote the polynomial degree.

\subsection{Convergence Test}
\label{sec:RobQSTPE_conv_test}
We start the analysis by assessing the performance of the method in terms of accuracy. We consider problem \eqref{eq:RobQSTPE_QS_TPE_system} in the domain $\Omega = (0,1)^2$ with the following manufactured analytical solution:
\begin{equation}
	\begin{aligned}
		\mathbf{u}(x,y) & \ = \left( \begin{aligned}
			& x^2 \cos\left(\frac{\pi x}{2}\right) \sin(\pi x) \\
			& x^2 \cos\left(\frac{\pi x}{2}\right) \sin(\pi x)
		\end{aligned} \right), \ \
		p(x,y) = x^2 \sin(\pi x) \sin(\pi y), \ \
		T(x,y) = -y^2 \sin(\pi x) \sin(\pi y),
	\end{aligned}
\end{equation}
through which we infer the boundary conditions and forcing terms. The model coefficients are reported in Table~\ref{tab:RobQSTPE_TPE_params_convtest}. 
\begin{table}[ht]
	\centering 
	\footnotesize
	\begin{tabular}{l | c c l | c c l | c}
		$a_0 \ [\si[per-mode = symbol]{\giga\pascal \per \kelvin\squared}]$ & 0.02 & & $\alpha \ [-]$ & 1 & & $\mu, \lambda \ [\si{\giga\pascal}]$ & 1, 5  \\
		$b_0 \ [\si{\per \kelvin}]$ & 0.01 & & $\beta \ [\si{\giga\pascal \per \kelvin}]$ & 0.8 & & $\mathbf{K} \ [\si{\dm \squared \per \giga\pascal \per \hour}]$ & 0.2 \\
		$c_0 \ [\si{\per \giga\pascal}]$ & 0.03 & & $c_f \ [\si{\pascal \per \kelvin\squared}]$ & 1 & & $\boldsymbol{\Theta} \ [\si{\dm \squared \giga\pascal \per \kelvin\squared \per \hour}]$ & 0.05
	\end{tabular}
	\caption{Test case of Section~\ref{sec:RobQSTPE_conv_test}: problem's parameters for the convergence analysis}
	\label{tab:RobQSTPE_TPE_params_convtest}
\end{table}

The convergence of the PolyDG scheme is tested both with respect to the mesh size $h$ and to the polynomial degree $\ell$. 
We consider three different constructions of the bilinear form $\mathcal{C}^*_h$. Indeed, we do not assess the performance of the stabilized version as we are not in an advection-dominated regime. 
For the $h$-convergence a sequence of polygonal Voronoi meshes is considered and we test different polynomial degrees $\ell = 2, 3, 4$ (one for each linearization). For what concerns the convergence with respect to $\ell$ and for a fixed mesh size, we fix a computational mesh of $100$ elements and vary the polynomial degree $\ell = 1,2,\dots,8$.
In Figure~\ref{fig:RobQSTPE_ConvH_Brun} and Figure~\ref{fig:RobQSTPE_ConvP_Brun} we show the computed errors in the $L^2$- and $dG$-norms defined as in \eqref{eq:RobQSTPE_DG_norms} versus $h$ and $\ell$, respectively, with the use of the form $\mathcal{C}^{\mathrm{old}}_h$ (cf. \cite{Brun2019, AntoniettiBonetti2022}). In Figure~\ref{fig:RobQSTPE_ConvH_DarcyVol}, Figure~\ref{fig:RobQSTPE_ConvP_DarcyVol} the same results are reported for the form $\mathcal{C}^{\mathrm{vol}}_h$. In both cases, we observe that the results match the predicted convergence rates in the framework of PolyDG spatial discretizations \cite[Theorem 5.3]{AntoniettiBonetti2022}. 

\input{Fig/ConvH.tikz}
\input{Fig/ConvP.tikz}

For what concerns the convergence vs $h$ we observe that, by using $\ell = 2, 3$, the $dG$-errors for all the three unknowns decrease as $h^2$ and $h^3$, respectively. Moreover, for what concerns the $L^2$-errors, we achieve $h^{\ell+1}$ convergence. Looking at the convergence with respect to $\ell$ we see that in the two cases, both for the $L^2$- and the $dG$-errors, we observe an exponential decrease of the error.

Finally, we test the accuracy of the method with the whole linearization $\mathcal{C}_h$ (without any stabilization, as we are not in the advection-dominated regime) and we report the results in Figure~\ref{fig:RobQSTPE_ConvH_Darcy}, Figure~\ref{fig:RobQSTPE_ConvP_Darcy}. Even in this case, we observe that the results are coherent with respect to the previous linearizations and to the estimates presented in \cite[Section 5]{AntoniettiBonetti2022}. Thus, we prove that the modifications of the trilinear form explained in Section~\ref{sec:RobQSTPE_DG_problem} do not affect the convergence properties of the method. 

We conclude this section by looking at the fixed-point iteration count for the convergence of our scheme. The tolerance is set as $\num[exponent-product=\ensuremath{\cdot}, print-unity-mantissa=false]{1e-10}$ and the stopping criterion is the norm of two successive iterations. 

In Figure~\ref{fig:RobQSTPE_Conv_iterations} we observe that the fixed-point iteration counts are very low and, in general, decrease both with respect to $N$ and $\ell$. Only for the case $\ell = 8$ we observe a small increment for $\mathcal{C}_h^{\text{vol}}$ and $\mathcal{C}_h$. We think that this may be due to the conditioning of the linear system.
For the most of the cases we can also observe that the fixed-point iteration counts required considering the linearizations presented in this article is lower with respect to the fixed-point iteration count when using $\mathcal{C}_h^{\text{old}}$.

\begin{center}
\begin{figure}[H]
\centering
\begin{subfigure}[b]{0.49\textwidth}
\begin{tikzpicture}
\begin{axis}[%
width=0.65\textwidth,
height=0.5\textwidth,
at={(0\textwidth,0\textwidth)},
scale only axis,
xmode=log,
xmin=4,
xmax=75,
xminorticks=true,
xlabel={$1/h$},
ymax=11,
yminorticks=true,
ylabel={\#Iterations},
legend style={draw=none,fill=none,legend cell align=left},
legend pos=north east
]
\addplot [color=myred,solid,line width=1.5pt, mark=diamond*,mark options={color=myred}] table[row sep=crcr]{
5.5214   10 \\
9.7539   9 \\
17.5835  8 \\
30.4349  7 \\
55.1835 6 \\
};
\addlegendentry{$\mathcal{C}_{h}^{\text{old}}$}
\addplot [color=myblue,solid,line width=1.5pt, mark=diamond*,mark options={color=myblue}] table[row sep=crcr]{
5.5214   4 \\
9.7539   3 \\
17.5835  3 \\
30.4349  3 \\
55.1835 2 \\
};
\addlegendentry{$\mathcal{C}_{h}^{\text{vol}}$}
\addplot [color=mygreen,solid,line width=1.5pt, mark=diamond*,mark options={color=mygreen}] table[row sep=crcr]{
5.5214    3 \\
9.7539    3 \\
17.5835   2 \\
30.4349   2 \\
55.1835  2 \\
};
\addlegendentry{$\mathcal{C}_{h}$}
\end{axis}
\end{tikzpicture}
\end{subfigure}
\begin{subfigure}[b]{0.49\textwidth}
\begin{tikzpicture}
\begin{axis}[%
width=0.65\textwidth,
height=0.5\textwidth,
at={(0\textwidth,0\textwidth)},
scale only axis,
xmin=0,
xmax=8.5,
xminorticks=true,
xlabel={$\ell$},
ymax=15,
yminorticks=true,
ylabel={\#Iterations},
legend style={draw=none,fill=none,legend cell align=left},
legend pos=north east
]
\addplot [color=myred,solid,line width=1.5pt, mark=diamond*,mark options={color=myred}] table[row sep=crcr]{
1   14 \\
2   10 \\
3   7 \\
4   6 \\
5   2 \\
6   2 \\
7   2 \\
8   2 \\
};
\addlegendentry{$\mathcal{C}_{h}^{\text{old}}$}
\addplot [color=myblue,solid,line width=1.5pt, mark=diamond*,mark options={color=myblue}] table[row sep=crcr]{
1   5 \\
2   5 \\
3   4 \\
4   3 \\
5   3 \\
6   2 \\
7   2 \\
8   7 \\
};
\addlegendentry{$\mathcal{C}_{h}^{\text{vol}}$}
\addplot [color=mygreen,solid,line width=1.5pt, mark=diamond*,mark options={color=mygreen}] table[row sep=crcr]{
1   5 \\
2   5 \\
3   4 \\
4   3 \\
5   3 \\
6   2 \\
7   2 \\
8   3 \\
};
\addlegendentry{$\mathcal{C}_{h}$}
\end{axis}
\end{tikzpicture}
\end{subfigure}
\caption{Test case of Section~\ref{sec:RobQSTPE_conv_test} (convergence test): number of iterations of the fixed point algorithm for the convergence with respect to $h$ (left, \textit{semilogx} scale) and  with respect to $\ell$ (right).}
\label{fig:RobQSTPE_Conv_iterations}
\end{figure}
\end{center}

\subsection{Robustness Tests}
\label{sec:RobQSTPE_rob_test}
In this section, we address the main point of this article: ensuring numerical robustness of the scheme with respect to the numerical parameters. In particular, we are interested in testing the robustness with respect to low values of the permeability $\mathbf{K}$ and thermal conductivity $\boldsymbol{\Theta}$. Moreover, in the presented tests, we stress also the coefficients related to the mass terms ($a_0$, $b_0$, and $c_0$) considering them to be $a_0 = b_0 = c_0$ and, eventually, equal to $0$. The values of the other model parameters are as in Table~\ref{tab:RobQSTPE_TPE_params_convtest}.

In \cite{AntoniettiBonetti2022} the robustness with respect to model parameters has been already addressed. Therein, it is observed that the numerical scheme in which $\mathcal{C}_h^{\text{old}}$ is used exibits good performance when $\mathbf{K} \ll \boldsymbol{\Theta}$ and when $\mathbf{K}, \boldsymbol{\Theta} \ll 1$. However, it is possible to observe instabilities when $\boldsymbol{\Theta} \ll \mathbf{K}$. In this situation, we can see \eqref{eq:RobQSTPE_energy_cons} as an advection dominated problem and following this observation we can introduce suitable stabilization terms for making the scheme able to cop with this degenerate case.


We compare the four linearizations $\mathcal{C}_h^{\text{old}}$, $\mathcal{C}_h^{\text{vol}}$, $\mathcal{C}_h$, $\mathcal{C}_h^{\text{stab}}$ in terms of performance when the thermal conductivity degenerates. Moreover, we perform some tests in which we consider $\mathbf{K} \ll \boldsymbol{\Theta}$ and $\mathbf{K}, \boldsymbol{\Theta} \ll 1$, in order to verify that although we change the linearization to make the method more robust with respect to $\boldsymbol{\Theta}$, we do not lose robustness with respect to the other physical parameters (cf. \cite{AntoniettiBonetti2022}).

\begin{remark}
	We observe that, compared with the values of the physical parameters presented in \cite{AntoniettiBonetti2022}, in this case we expect the fixed-point algorithm either to converge slowly or diverge. This is because, in the stationary problem, we do not have the scaling of the parameters with respect to the time step, which was taken much smaller than 1 (namely, we considered $\delta t = \num[exponent-product=\ensuremath{\cdot}, print-unity-mantissa=true]{5e-5}$). In practice, we expect that for $\mathcal{C}_h^{\text{old}}$ the algorithm will fail for higher values of $\boldsymbol{\Theta}$ than we can observe in \cite{AntoniettiBonetti2022}.
\end{remark}

\subsubsection{Robustness with respect to degenerating thermal conductivity}
\label{sec:RobQSTPE_rob_test_theta}
In this section we focus on the robustness of the method with respect to the thermal conductivity. To this aim, we consider the following choice of the parameters: through the sequence of the simulations we keep constant $a_0 = b_0 = c_0 = 0$ and $\mathbf{K} = 1 \mathbf{I}$; instead, we consider the following sequence for the values of $\boldsymbol{\Theta} = \{1 \mathbf{I}, \num[exponent-product=\ensuremath{\cdot}, print-unity-mantissa=false]{1e-2} \mathbf{I}, \num[exponent-product=\ensuremath{\cdot}, print-unity-mantissa=false]{1e-4} \mathbf{I}, \num[exponent-product=\ensuremath{\cdot}, print-unity-mantissa=false]{1e-6} \mathbf{I}, \num[exponent-product=\ensuremath{\cdot}, print-unity-mantissa=false]{1e-8} \mathbf{I}, \num[exponent-product=\ensuremath{\cdot}, print-unity-mantissa=false]{1e-10} \mathbf{I} \}$.

We test the three linearizations proposed in this article and the linearization of \cite{Brun2019, AntoniettiBonetti2022}. For the presented sets of physical parameters we perform three simulations considering $N = [1000, 310, 100]$ and $\ell = [2, 3, 4]$. Then, we compare their performance in terms of errors ($L^2$- and $dG$-norms) and of fixed-point iteration counts. 

The results concerning the case ($\ell = 3, N = 310)$ are presented in Figure~\ref{fig:RobQSTPE_RobThetaErrors_P3} and Table~\ref{tab:RobQSTPE_RobThetaIterations} (second stripe). First, we observe that the linearization $\mathcal{C}_h^{\text{old}}$ leads to a fixed-point algorithm that is not convergent. Indeed, the fixed-point iteration counts and the errors blow up for all the values of $\boldsymbol{\Theta}$ lower than $1$. The three linearizations presented in this article work better. The algorithm using $\mathcal{C}_h^{\text{vol}}$ converges for the first four values of $\boldsymbol{\Theta}$, the algorithm that exploits $\mathcal{C}_h$ for the first five values and the stabilized version converges for all the tested values of $\boldsymbol{\Theta}$. By looking at the results for the temperature field $T_h$, we can observe that the $L^2$-errors for all the three methods increase as $\boldsymbol{\Theta}$ decreases. For what concerns the stabilized version, we see that after $\num[exponent-product=\ensuremath{\cdot}, print-unity-mantissa=false]{1e-6}$ settles around the value of \num[exponent-product=\ensuremath{\cdot}]{7e-2}. The same behavior can be observed also for the displacement and pressure fields. It is of particular interest to notice that the $L^2$- and $dG$-errors of the scheme exploiting $\mathcal{C}_h^{\text{stab}}$ remain almost constant as $\boldsymbol{\Theta}$ decreases for $\mathbf{u}_h$ and $p_h$. Similar results are obtained for the cases $(\ell = 2, N = 1000)$ and $(\ell = 4, N = 100)$ (cf. \ref{sec:RobQSTPE_rob_test_appendix_further_results}).

By looking at Table~\ref{tab:RobQSTPE_RobThetaIterations}, it is interesting to notice that -- for the three methods that make use of $\mathcal{C}_h^{\text{vol}}$, $\mathcal{C}_h$, and $\mathcal{C}_h^{\text{stab}}$ -- the higher the polynomial degree of approximation, the more robust the scheme. In fact, we can see that the threshold value for which the method does not converge decreases as the polynomial degree increases, and in addition, the iteration number remains comparable (often lower) than in previous cases. We can notice that -- in the case of the stabilized bilinear form $\mathcal{C}_h^{\text{stab}}$ -- most of the quantities analyzed remain constant and hover around a value below a certain threshold. This fact is also a qualitative guarantee of limited error and number of iterations even for degenerate values of thermal conductivity.

The numerical values of the results presented in the supplementary material of the article.

\begin{table}[H]
	\centering 
	\footnotesize
	\begin{tabular}{c | c | c | c | c | c | c | c}
		\textbf{Test} & \textbf{Linearization} & $\theta = 1$ & $\theta = \num[exponent-product=\ensuremath{\cdot}, print-unity-mantissa=false]{1e-2}$ & $\theta = \num[exponent-product=\ensuremath{\cdot}, print-unity-mantissa=false]{1e-4}$ & $\theta =  \num[exponent-product=\ensuremath{\cdot}, print-unity-mantissa=false]{1e-6}$ & $\theta = \num[exponent-product=\ensuremath{\cdot}, print-unity-mantissa=false]{1e-8}$ & $\theta = \num[exponent-product=\ensuremath{\cdot}, print-unity-mantissa=false]{1e-10}$ \T\B \\
		\hline
		\multirow{4}{*}{\vspace{-0.3cm} \rotatebox[origin=b]{90}{\specialcell{$\ell = 2$ \\ $N = 1000$}}}
		& $\mathcal{C}_h^{\text{old}}$ & 7 & \cellcolor{red!30}1000 & \cellcolor{red!30}1000 & \cellcolor{red!30}1000 & \cellcolor{red!30}1000 & \cellcolor{red!30}1000 \T\B \\
		& $\mathcal{C}_h^{\text{vol}}$ & 4 & 10 & 8 & \cellcolor{red!30}1000 & \cellcolor{red!30}1000 & \cellcolor{red!30}1000 \T\B \\
		& $\mathcal{C}_h$ & 4 & 10 & 11 & 9 & \cellcolor{red!30}1000 & \cellcolor{red!30}1000 \T\B \\
		& $\mathcal{C}_h^{\text{stab}}$ & 4 & 10 & 8 & 17 & 20 & 10 \T\B \\
		\hline
		\multirow{4}{*}{\vspace{-0.3cm} \rotatebox[origin=b]{90}{\specialcell{$\ell = 3$ \\ $N = 310$}}}
		& $\mathcal{C}_h^{\text{old}}$ & 5 & \cellcolor{red!30}1000 & \cellcolor{red!30}1000 & \cellcolor{red!30}1000 & \cellcolor{red!30}1000 & \cellcolor{red!30}1000 \T\B \\
		& $\mathcal{C}_h^{\text{vol}}$ & 3 & 8 & 5 & 11 & \cellcolor{red!30}1000 & \cellcolor{red!30}1000 \T\B \\
		& $\mathcal{C}_h$ & 3 & 8 & 9 & 7 & 48 & \cellcolor{red!30}1000 \T\B \\
		& $\mathcal{C}_h^{\text{stab}}$ & 3 & 8 & 7 & 7 & 7 & 8 \T\B \\
		\hline
		\multirow{4}{*}{\vspace{-0.3cm} \rotatebox[origin=b]{90}{\specialcell{$\ell = 4$ \\ $N = 100$}}}
		& $\mathcal{C}_h^{\text{old}}$ & 5 & \cellcolor{red!30}1000 & \cellcolor{red!30}1000 & \cellcolor{red!30}1000 & \cellcolor{red!30}1000 & \cellcolor{red!30}1000 \T\B \\
		& $\mathcal{C}_h^{\text{vol}}$ & 43 & 19 & 5 & 17 & \cellcolor{red!30}1000 & \cellcolor{red!30}1000 \T\B \\
		& $\mathcal{C}_h$ & 3 & 44 & 5 & 6 & 6 & 6 \T\B \\
		& $\mathcal{C}_h^{\text{stab}}$ & 3 & 33 & 5 & 7 & 7 & 7 \T\B \\
	\end{tabular}
	\caption{Test case of Section~\ref{sec:RobQSTPE_rob_test_theta} (robustness w.r.t. degenerating thermal conductivity): fixed-point iteration counts for the convergence of the algorithm versus $\theta$. The cells highlighted in red are the tests for which the maximum number of iterations is reached.}
	\label{tab:RobQSTPE_RobThetaIterations}
\end{table}

\begin{remark}
	We observed that, without introducing the \textit{inflow} stabilization for the boundary faces, the linearized form $\mathcal{C}_h^{\text{uw}}$ leads to performance that are similar to the ones of $\mathcal{C}_h$. Indeed, $\mathcal{C}_h^{\text{stab}}$ turns out to be way more robust.  
\end{remark}
\begin{remark}
	For the sake of completeness, we have reported also the value of the $dG$-errors, defined as in \eqref{eq:RobQSTPE_DG_norms}, for all the three cases. However, we observe that, as the thermal conductivity enters in the definition of the norm $\| \cdot \|_{dG,T}$ (cf. \eqref{eq:RobQSTPE_DG_norms}), their values are altered by the fact that we are considering very low values for $\boldsymbol{\Theta}$. It is still interesting to notice that the $dG$-errors for $\mathcal{C}_h^{\text{old}}$ blows up, while the $dG$-errors related to  $\mathcal{C}_h^{\text{vol}}$,  $\mathcal{C}_h$, and $\mathcal{C}_h^{\text{stab}}$ remain bounded around quite low values.
\end{remark}
\input{Fig/RobTheta_p3.tikz}

\subsubsection{Robustness with respect to degenerating permeability}
\label{sec:RobQSTPE_rob_test_kappa}
The aim of this section is to inspect the robustness of the scheme with respect to degenerating values of $\mathbf{K}$. As the method proposed in \cite{AntoniettiBonetti2022} was robust with respect to the hydraulic mobility, we directly test it considering $\mathbf{K} = \num[exponent-product=\ensuremath{\cdot}, print-unity-mantissa=false]{1e-10}  \mathbf{I}$ to verify that, by changing the linearization bilinear form to ensure robustness with respect to $\boldsymbol{\Theta}$, we do not lose robustness with respect to $\mathbf{K}$. The other physical parameters we consider are: $a_0 = b_0 = c_0 = 0$ and $\boldsymbol{\Theta} = 1 \mathbf{I}$. We consider a sequence of Voronoi meshes whose number of elements is $N = [100, 310, 1000, 3100, 10000]$ and polynomial degree of approximation equal to $\ell = 2$.

In Figure~\ref{fig:RobQSTPE_RobKappaErrors_P2}, Table~\ref{tab:RobQSTPE_RobKappaIterations_P2} we report the results for the errors and the number of iterations, respectively. First, we observe that for the three linearizations proposed in this article, the fixed-point algorithm converge in a few iterations and that the number of iterations is comparable for the three cases. For what concerns $\mathcal{C}_h^{\text{old}}$, we see that globally it converges. Even if it reaches the maximum number of iterations for $N = 3100$, we observe that the computed error for that refinement is very close to the one computed using the other three methods. Concerning the errors, we see that we have convergence both in $L^2$- and $dG$-norms. In the $L^2$-case we have order of convergence equal to $2$ (i.e. we lose the $\ell+1$ accuracy rate), while in the $dG$-case we have order of accuracy below $2$. 

\begin{figure}[H]
\centering
\begin{subfigure}[b]{0.49\textwidth}
\begin{tikzpicture}
\begin{axis}[%
width=0.65\textwidth,
height=0.5\textwidth,
at={(0\textwidth,0\textwidth)},
scale only axis,
xmode=log,
xmin=4,
xmax=75,
xminorticks=true,
xlabel={$1/h$},
ymode=log,
ymin=1e-3,
ymax=0.5,
yminorticks=true,
ylabel={$L^2$-Errors ($p_h$)},
legend style={draw=none,fill=none,legend cell align=left},
legend pos=south west
]
\addplot [color=myyellow,solid,line width=1.5pt, mark=diamond*,mark options={color=myyellow}]
  table[row sep=crcr]{
5.5214   0.18439 \\
9.7539   0.056148 \\
17.5835  0.021448 \\
30.4349  0.0079298 \\
55.1835  0.0025529 \\
};
\addlegendentry{$\mathcal{C}^{\text{old}}_h$}

\addplot [color=myblue,solid,line width=1.5pt,mark=square*,mark options={color=myblue}]
  table[row sep=crcr]{
5.5214   0.18439 \\
9.7539   0.056148 \\
17.5835  0.021448 \\
30.4349  0.0079298 \\
55.1835  0.0025529 \\
};
\addlegendentry{$\mathcal{C}^{\text{vol}}_h$}

\addplot [color=mygreen,solid,line width=1.5pt,mark=triangle*,mark options={color=mygreen}]
  table[row sep=crcr]{
5.5214   0.18439 \\
9.7539   0.056148 \\
17.5835   0.021448 \\
30.4349   0.0079298 \\
55.1835   0.0025529 \\
};
\addlegendentry{$\mathcal{C}_h$}

\addplot [color=myred,solid,line width=1.5pt,mark=triangle*,mark options={color=myred}]
  table[row sep=crcr]{
5.5214   0.18439 \\
9.7539   0.056148 \\
17.5835   0.021448 \\
30.4349   0.0079298 \\
55.1835   0.0025529 \\
};
\addlegendentry{$\mathcal{C}^{\text{stab}}_h$}

\addplot [color=black,solid,line width=0.5pt]
  table[row sep=crcr]{
    30.4349      0.0132 \\
    55.1835      0.0132 \\
    55.1835      0.004 \\
    30.4349      0.0132 \\  
};
\node[right, align=left, text=black, font=\footnotesize]
at (axis cs:55.1835, 0.0086) {2}; 

\end{axis}
\end{tikzpicture}
\end{subfigure}
\begin{subfigure}[b]{0.49\textwidth}
\begin{tikzpicture}
\begin{axis}[%
width=0.65\textwidth,
height=0.5\textwidth,
at={(0\textwidth,0\textwidth)},
scale only axis,
xmode=log,
xmin=4,
xmax=75,
xminorticks=true,
xlabel={$1/h$},
ymode=log,
ymin=1e-5,
ymax=3e-4,
yminorticks=true,
ylabel={$dG$-Errors ($p_h$)},
legend style={draw=none,fill=none,legend cell align=left},
legend pos=south west
]
\addplot [color=myyellow,solid,line width=1.5pt, mark=diamond*,mark options={color=myyellow}]
  table[row sep=crcr]{
5.5214   0.00014555 \\
9.7539   7.3643e-05 \\
17.5835  4.418e-05 \\
30.4349  2.7949e-05 \\
55.1835  1.5159e-05 \\
};
\addlegendentry{$\mathcal{C}^{\text{old}}_h$}

\addplot [color=myblue,solid,line width=1.5pt,mark=square*,mark options={color=myblue}]
  table[row sep=crcr]{
5.5214   0.00014555 \\
9.7539   7.3643e-05 \\
17.5835  4.418e-05 \\
30.4349  2.7949e-05 \\
55.1835  1.5159e-05 \\
};
\addlegendentry{$\mathcal{C}^{\text{vol}}_h$}

\addplot [color=mygreen,solid,line width=1.5pt,mark=triangle*,mark options={color=mygreen}]
  table[row sep=crcr]{
5.5214   0.00014555 \\
9.7539   7.3643e-05 \\
17.5835  4.418e-05 \\
30.4349  2.7949e-05 \\
55.1835  1.5159e-05 \\
};
\addlegendentry{$\mathcal{C}_h$}

\addplot [color=myred,solid,line width=1.5pt,mark=triangle*,mark options={color=myred}]
  table[row sep=crcr]{
5.5214   0.00014555 \\
9.7539   7.3643e-05 \\
17.5835  4.418e-05 \\
30.4349  2.7949e-05 \\
55.1835  1.5159e-05 \\
};
\addlegendentry{$\mathcal{C}^{\text{stab}}_h$}

\end{axis}
\end{tikzpicture}
\end{subfigure}
\caption{Test case of Section~\ref{sec:RobQSTPE_rob_test_kappa} (robustness w.r.t. degenerating permeability): computed errors for the pressure $p_h$ in $L^2$-norm (left) and $dG$-norm (right) versus $1/h$ (\textit{log-log} scale). The different colors represent the different linearization schemes. The polynomial degree of approximation is $\ell = 2$.}
\label{fig:RobQSTPE_RobKappaErrors_P2}
\end{figure}

\begin{table}[H]
	\centering 
	\footnotesize
	\begin{tabular}{ c | c  | c | c | c | c }
		\textbf{Linearization} & $h = 0.1811$ & $h = 0.1025$ & $h = 0.0569$ & $h = 0.0329$ & $h = 0.0181$ \T\B \\
		\hline \T\B
		$\mathcal{C}_h^{\text{old}}$ & 3 & \cellcolor{red!30}1000 & 2 & 2 & 2 \T\B \\
		$\mathcal{C}_h^{\text{vol}}$ & 3 & 3 & 2 & 2 & 2 \T\B \\
		$\mathcal{C}_h$ & 3 & 2 & 3 & 3 & 2 \T\B \\
		$\mathcal{C}_h^{\text{stab}}$ & 3 & 3 & 3 & 2 & 2 \T\\
	\end{tabular}
	\caption{Test case of Section~\ref{sec:RobQSTPE_rob_test_kappa} (robustness w.r.t. degenerating permeability): fixed-point iteration counts for the convergence of the algorithm versus $h$. The polynomial degree of approximation is $\ell = 2$. The cells highlighted in red are the test cases for which the maximum number of iterations is reached.}
	\label{tab:RobQSTPE_RobKappaIterations_P2}
\end{table}

We observe that the $\ell+1$ accuracy in the $L^2$-norm is not guaranteed by the theory, we are only sure the $L^2$-errors converge at least with rate $\ell$, when using $\ell$ as polynomial degree of approximation. In the $dG$-norm we do not observe the expected order because the values of the errors are altered by the fact that the permeability directly enters in the definition of the $dG$-norms, cf. \eqref{eq:RobQSTPE_DG_norms}.

The analysis for the displacement $\mathbf{u}_h$ and the temperature $T_h$ is reported in Appendix~\ref{sec:RobQSTPE_rob_test_appendix_further_results}. The numerical values of the results presented in the supplementary material of the article.

\subsubsection{Robustness with respect to degenerating thermal conductivity and permeability}
\label{sec:RobQSTPE_rob_test_thetakappa}

As for the previous test, the aim of this section is to verify that we have not lost robustness with respect to the scheme proposed in \cite{Brun2019, AntoniettiBonetti2022}. Here, we consider the case $\mathbf{K}, \boldsymbol{\Theta} \ll 1$. On a sequence of Voronoi meshes with $N = [100, 310, 1000, 3100, 10000]$ and using as polynomial degree of approximation $\ell = 2$, we use run a convergence test with the following values of the model's coefficient: $\mathbf{K} = \boldsymbol{\Theta} = \num[exponent-product=\ensuremath{\cdot}, print-unity-mantissa=false]{1e-10}  \mathbf{I}$ and $a_0 = b_0 = c_0 = 0.01$. As in the previous section we observe numerical robustness of the method with respect to small values of permeability and thermal conductivity. In Figure~\ref{fig:RobQSTPE_RobThetaKappaErrors_pressure_P2}, Figure~\ref{fig:RobQSTPE_RobThetaKappaErrors_temperature_P2} the errors for the pressure and temperature fields are reported, respectively. In both the two cases we observe $2$ as order of accuracy in $L^2$-norm and convergence in $dG$-norm. We observe in Table~\ref{tab:RobQSTPE_RobThetaKappaIterations_P2} that the number of iterations  necessary for the convergence of the fixed-point iteration scheme is slightly greater than the previous case, however it is still limited and quite low (maximum is $11$) considering that we are using very low  values for the two tensors. The analysis for the displacement $\mathbf{u}_h$ is reported in Appendix~\ref{sec:RobQSTPE_rob_test_appendix_further_results}. The computed numerical values of the results presented in the supplementary material of the article.

\begin{figure}[H]
\centering
\begin{subfigure}[b]{1\textwidth}
\begin{subfigure}[b]{0.49\textwidth}
\begin{tikzpicture}
\begin{axis}[%
width=0.65\textwidth,
height=0.5\textwidth,
at={(0\textwidth,0\textwidth)},
scale only axis,
xmode=log,
xmin=4,
xmax=75,
xminorticks=true,
xlabel={$1/h$},
ymode=log,
ymin=1e-3,
ymax=0.2,
yminorticks=true,
ylabel={$L^2$-Errors ($p_h$)},
legend style={draw=none,fill=none,legend cell align=left},
legend pos=south west
]
\addplot [color=myyellow,solid,line width=1.5pt, mark=diamond*,mark options={color=myyellow}]
  table[row sep=crcr]{
5.5214   0.09964 \\
9.7539   0.030585 \\
17.5835   0.011751 \\
30.4349   0.0043638 \\
55.1835   0.0014118 \\
};
\addlegendentry{$\mathcal{C}^{\text{old}}_h$}

\addplot [color=myblue,solid,line width=1.5pt,mark=square*,mark options={color=myblue}]
  table[row sep=crcr]{
5.5214   0.09964 \\
9.7539   0.030585 \\
17.5835   0.011751 \\
30.4349   0.0043638 \\
55.1835   0.0014118 \\
};
\addlegendentry{$\mathcal{C}^{\text{vol}}_h$}

\addplot [color=mygreen,solid,line width=1.5pt,mark=triangle*,mark options={color=mygreen}]
  table[row sep=crcr]{
5.5214   0.099649 \\
9.7539   0.030586 \\
17.5835   0.011752 \\
30.4349   0.0043638 \\
55.1835   0.0014118 \\
};
\addlegendentry{$\mathcal{C}_h$}

\addplot [color=myred,solid,line width=1.5pt,mark=triangle*,mark options={color=myred}]
  table[row sep=crcr]{
5.5214   0.099643 \\
9.7539   0.030582 \\
17.5835   0.011751 \\
30.4349   0.0043637 \\
55.1835   0.0014118 \\
};
\addlegendentry{$\mathcal{C}^{\text{stab}}_h$}

\addplot [color=black,solid,line width=0.5pt]
  table[row sep=crcr]{
    30.4349      0.0066 \\
    55.1835      0.0066 \\
    55.1835      0.002 \\
    30.4349      0.0066 \\  
};
\node[right, align=left, text=black, font=\footnotesize]
at (axis cs:55.1835, 0.0043) {2}; 

\end{axis}
\end{tikzpicture}
\end{subfigure}
\begin{subfigure}[b]{0.49\textwidth}
\begin{tikzpicture}
\begin{axis}[%
width=0.65\textwidth,
height=0.5\textwidth,
at={(0\textwidth,0\textwidth)},
scale only axis,
xmode=log,
xmin=4,
xmax=75,
xminorticks=true,
xlabel={$1/h$},
ymode=log,
ymin=5e-6,
ymax=1e-4,
yminorticks=true,
ylabel={$dG$-Errors ($p_h$)},
legend style={draw=none,fill=none,legend cell align=left},
legend pos=south west
]
\addplot [color=myyellow,solid,line width=1.5pt, mark=diamond*,mark options={color=myyellow}]
  table[row sep=crcr]{
5.5214   8.0956e-05 \\
9.7539   4.094e-05 \\
17.5835   2.4558e-05 \\
30.4349   1.5539e-05 \\
55.1835   8.4361e-06 \\
};
\addlegendentry{$\mathcal{C}^{\text{old}}_h$}

\addplot [color=myblue,solid,line width=1.5pt,mark=square*,mark options={color=myblue}]
  table[row sep=crcr]{
5.5214   8.0956e-05 \\
9.7539   4.094e-05 \\
17.5835  2.4558e-05 \\
30.4349  1.5539e-05 \\
55.1835  8.4361e-06 \\
};
\addlegendentry{$\mathcal{C}^{\text{vol}}_h$}

\addplot [color=mygreen,solid,line width=1.5pt,mark=triangle*,mark options={color=mygreen}]
  table[row sep=crcr]{
5.5214   8.0955e-05 \\
9.7539   4.094e-05 \\
17.5835  2.4558e-05 \\
30.4349  1.5539e-05 \\
55.1835  8.4361e-06 \\
};
\addlegendentry{$\mathcal{C}_h$}

\addplot [color=myred,solid,line width=1.5pt,mark=triangle*,mark options={color=myred}]
  table[row sep=crcr]{
5.5214   8.0956e-05 \\
9.7539   4.094e-05 \\
17.5835  2.4558e-05 \\
30.4349  1.5539e-05 \\
55.1835  8.4361e-06 \\
};
\addlegendentry{$\mathcal{C}^{\text{stab}}_h$}
\end{axis}
\end{tikzpicture}
\end{subfigure}
\caption{Robustness test vs $\mathbf{K} = k \mathbf{I}, \boldsymbol{\Theta} = \theta \mathbf{I}$: computed errors for the pressure $p_h$.}
\label{fig:RobQSTPE_RobThetaKappaErrors_pressure_P2}
\end{subfigure}

\begin{subfigure}[b]{1\textwidth}
\begin{subfigure}[b]{0.49\textwidth}
\begin{tikzpicture}
\begin{axis}[%
width=0.65\textwidth,
height=0.5\textwidth,
at={(0\textwidth,0\textwidth)},
scale only axis,
xmode=log,
xmin=4,
xmax=75,
xminorticks=true,
xlabel={$1/h$},
ymode=log,
ymin=1e-3,
ymax=0.2,
yminorticks=true,
ylabel={$L^2$-Errors ($T_h$)},
legend style={draw=none,fill=none,legend cell align=left},
legend pos=south west
]
\addplot [color=myyellow,solid,line width=1.5pt, mark=diamond*,mark options={color=myyellow}]
  table[row sep=crcr]{
5.5214   0.09964 \\
9.7539   0.030584 \\
17.5835  0.011751 \\
30.4349  0.0043632 \\
55.1835  0.0014113 \\
};
\addlegendentry{$\mathcal{C}^{\text{old}}_h$}

\addplot [color=myblue,solid,line width=1.5pt,mark=square*,mark options={color=myblue}]
  table[row sep=crcr]{
5.5214   0.09964 \\
9.7539   0.030584 \\
17.5835  0.011751 \\
30.4349  0.0043632 \\
55.1835  0.0014113 \\
};
\addlegendentry{$\mathcal{C}^{\text{vol}}_h$}

\addplot [color=mygreen,solid,line width=1.5pt,mark=triangle*,mark options={color=mygreen}]
  table[row sep=crcr]{
5.5214   0.099649 \\
9.7539   0.030585 \\
17.5835  0.011751 \\
30.4349  0.0043632 \\
55.1835  0.0014113 \\
};
\addlegendentry{$\mathcal{C}_h$}

\addplot [color=myred,solid,line width=1.5pt,mark=triangle*,mark options={color=myred}]
  table[row sep=crcr]{
5.5214   0.099642 \\
9.7539   0.030582 \\
17.5835  0.01175 \\
30.4349  0.0043631 \\
55.1835  0.0014113 \\
};
\addlegendentry{$\mathcal{C}^{\text{stab}}_h$}

\addplot [color=black,solid,line width=0.5pt]
  table[row sep=crcr]{
    30.4349      0.0066 \\
    55.1835      0.0066 \\
    55.1835      0.002 \\
    30.4349      0.0066 \\  
};
\node[right, align=left, text=black, font=\footnotesize]
at (axis cs:55.1835, 0.0043) {2}; 

\end{axis}
\end{tikzpicture}
\end{subfigure}
\begin{subfigure}[b]{0.49\textwidth}
\begin{tikzpicture}
\begin{axis}[%
width=0.65\textwidth,
height=0.5\textwidth,
at={(0\textwidth,0\textwidth)},
scale only axis,
xmode=log,
xmin=4,
xmax=75,
xminorticks=true,
xlabel={$1/h$},
ymode=log,
ymin=5e-6,
ymax=1e-4,
yminorticks=true,
ylabel={$dG$-Errors ($T_h$)},
legend style={draw=none,fill=none,legend cell align=left},
legend pos=south west
]
\addplot [color=myyellow,solid,line width=1.5pt, mark=diamond*,mark options={color=myyellow}]
  table[row sep=crcr]{
5.5214   8.0952e-05 \\
9.7539   4.0936e-05 \\
17.5835  2.4552e-05 \\
30.4349  1.5531e-05 \\
55.1835  8.425e-06 \\
};
\addlegendentry{$\mathcal{C}^{\text{old}}_h$}

\addplot [color=myblue,solid,line width=1.5pt,mark=square*,mark options={color=myblue}]
  table[row sep=crcr]{
5.5214   8.0952e-05 \\
9.7539   4.0936e-05 \\
17.5835  2.4552e-05 \\
30.4349  1.5531e-05 \\
55.1835  8.425e-06 \\
};
\addlegendentry{$\mathcal{C}^{\text{vol}}_h$}

\addplot [color=mygreen,solid,line width=1.5pt,mark=triangle*,mark options={color=mygreen}]
  table[row sep=crcr]{
5.5214   8.0952e-05 \\
9.7539   4.0936e-05 \\
17.5835  2.4552e-05 \\
30.4349  1.5531e-05 \\
55.1835  8.425e-06 \\
};
\addlegendentry{$\mathcal{C}_h$}

\addplot [color=myred,solid,line width=1.5pt,mark=triangle*,mark options={color=myred}]
  table[row sep=crcr]{
5.5214   8.0952e-05 \\
9.7539   4.0936e-05 \\
17.5835  2.4552e-05 \\
30.4349  1.5531e-05 \\
55.1835  8.425e-06 \\
};
\addlegendentry{$\mathcal{C}^{\text{stab}}_h$}
\end{axis}
\end{tikzpicture}
\end{subfigure}
\caption{Robustness test vs $\mathbf{K} = k \mathbf{I}, \boldsymbol{\Theta} = \theta \mathbf{I}$: computed errors for the temperature $T_h$.}
\label{fig:RobQSTPE_RobThetaKappaErrors_temperature_P2}
\end{subfigure}
\caption{Test case of Section~\ref{sec:RobQSTPE_rob_test_kappa} (robustness w.r.t. degenerating permeability and thermal conductivity): computed errors in $L^2$-norm (left column) and $dG$-norms (right column) versus $1/h$ (\textit{log-log} scale). The different colors represent the different linearization schemes. The polynomial degree of approximation is $\ell = 2$.}
\label{fig:RobQSTPE_RobThetaKappaErrors_P2}
\end{figure}
\begin{table}[H]
	\centering 
	\footnotesize
	\begin{tabular}{ c | c  | c | c | c | c }
		\textbf{Linearization} & $h = 0.1811$ & $h = 0.1025$ & $h = 0.0569$ & $h = 0.0329$ & $h = 0.0181$ \T\B \\
		\hline \T\B
		$\mathcal{C}_h^{\text{old}}$ & 4 & 8 & 4 & 4 & 3 \T\B \\
		$\mathcal{C}_h^{\text{vol}}$ & 6 & 11 & 4 & 4 & 3 \T\B \\
		$\mathcal{C}_h$ & 4 & 7 & 4 & 4 & 3 \T\B \\
		$\mathcal{C}_h^{\text{stab}}$ & 9 & 18 & 6 & 5 & 3 \T\\
	\end{tabular}
	\caption{Test case of Section~\ref{sec:RobQSTPE_rob_test_thetakappa} (robustness w.r.t. degenerating permeability and thermal conductivity): fixed-point iteration counts for the convergence the algorithm versus $h$. The polynomial degree of approximation is $\ell = 2$.}
	\label{tab:RobQSTPE_RobThetaKappaIterations_P2}
\end{table}

\section{Conclusions}
In this work we have presented a four-field PolyDG-WSIP formulation for the non-linear fully-coupled thermo-hydro-mechanical problem. The stability estimate and the convergence of the fixed-point iteration scheme are presented. Numerical simulations are performed to assess the theoretical bounds and test the robustness properties of the method. The results confirm that the PolyDG-WSIP scheme proposed here is appealing for real problems' simulations. 

Further developments of the present work are possible. In particular, we mention the extension to other non-linear models, such the Darcy-Forchheimer flow equation, the inclusion of the non-linear convective term and of  the stabilization techniques presented in this paper in the fully-dynamic problem, and the generalization of the non-linear term including the time-derivative of the solid displacement. Furthermore, for what concerns the solution strategies, one possible development is the implementation of effective splitting schemes for reducing the large computational cost required by a non-linear fully-coupled problem. This last development is of particular interest for the extension of the numerical implementation to the three-dimensional case.

\section*{Funding}
S.B and P.F.A. has been partially funded by the research grant PRIN2020 n. 20204LN5N5 funded by the Italian Ministry of Universities and Research (MUR). P.F.A. has been partially funded by the Italian Research Center on High Performance Computing, Big Data and Quantum Computing (ICSC), under the NextGenerationEU project. S.B., M.B., and P.F.A. are members of INdAM-GNCS. The work of M.B. has been partially supported by the INdAM-GNCS project CUP E55F22000270001. The present research is part of the activities of the project Dipartimento di Eccellenza 2023-2027, Dipartimento di Matematica, Politecnico di Milano.

\bibliography{bibliography}

\begin{thebibliography}{34}
\providecommand{\natexlab}[1]{#1}
\providecommand{\url}[1]{\texttt{#1}}
\expandafter\ifx\csname urlstyle\endcsname\relax
  \providecommand{\doi}[1]{doi: #1}\else
  \providecommand{\doi}{doi: \begingroup \urlstyle{rm}\Url}\fi

\bibitem[Antonietti et~al.(2013)Antonietti, Giani, and Houston]{Antonietti2013}
P.~F. Antonietti, S.~Giani, and P.~Houston.
\newblock {$hp$-version composite Discontinuous Galerkin methods for elliptic
  problems on complicated domains}.
\newblock \emph{SIAM Journal on Scientific Computing}, 35\penalty0
  (3):\penalty0 A1417--A1439, 2013.

\bibitem[Antonietti et~al.(2019)Antonietti, Facciolà, Russo, and
  Verani]{Antonietti2019}
P.~F. Antonietti, C.~Facciolà, A.~Russo, and M.~Verani.
\newblock {Discontinuous Galerkin approximation of flows in fractured porous
  media on polytopic grids}.
\newblock \emph{SIAM Journal on Scientific Computing}, 41\penalty0
  (1):\penalty0 A109--A138, 2019.

\bibitem[Antonietti et~al.(2021{\natexlab{a}})Antonietti, Botti, Mazzieri, and
  Nati~Poltri]{Antonietti.Botti.ea:21}
P.~F. Antonietti, M.~Botti, I.~Mazzieri, and S.~Nati~Poltri.
\newblock A high-order discontinuous {G}alerkin method for the
  poro-elasto-acoustic problem on polygonal and polyhedral grids.
\newblock \emph{SIAM Journal on Scientific Computing}, 44\penalty0
  (1):\penalty0 B1--B28, 2021{\natexlab{a}}.

\bibitem[Antonietti et~al.(2021{\natexlab{b}})Antonietti, Mascotto, Verani, and
  Zonca]{antonietti2020_stokesDG}
P.~F. Antonietti, L.~Mascotto, M.~Verani, and S.~Zonca.
\newblock {Stability analysis of polytopic Discontinuous Galerkin
  approximations of the Stokes problem with applications to Fluid-Structure
  Interaction Problems}.
\newblock \emph{Journal of Scientific Computing}, 90\penalty0 (1):\penalty0 23,
  2021{\natexlab{b}}.

\bibitem[Antonietti et~al.(2023)Antonietti, Bonetti, and
  Botti]{AntoniettiBonetti2022}
P.~F. Antonietti, S.~Bonetti, and M.~Botti.
\newblock Discontinuous {G}alerkin approximation of the fully coupled
  thermo-poroelastic problem.
\newblock \emph{SIAM Journal on Scientific Computing}, 45\penalty0
  (2):\penalty0 A621--A645, 2023.

\bibitem[Arnold(1982)]{Arnold1982}
D.~N. Arnold.
\newblock {An Interior Penalty Finite Element method with discontinuous
  elements}.
\newblock \emph{SIAM Journal on Numerical Analysis}, 19\penalty0 (4):\penalty0
  742--760, 1982.

\bibitem[Arnold et~al.(2002)Arnold, Brezzi, Cockburn, and Marini]{Arnold2002}
D.~N. Arnold, F.~Brezzi, B.~Cockburn, and L.~D. Marini.
\newblock {Unified analysis of Discontinuous Galerkin methods for elliptic
  problems}.
\newblock \emph{SIAM Journal on Numerical Analysis}, 39\penalty0 (5):\penalty0
  1749--1779, 2002.

\bibitem[Ayuso and Marini(2009)]{Ayuso2009}
B.~Ayuso and L.~D. Marini.
\newblock Discontinuous {G}alerkin {M}ethods for advection-diffusion-reaction
  problems.
\newblock \emph{SIAM Journal on Numerical Analysis}, 47\penalty0 (2):\penalty0
  1391--1420, 2009.

\bibitem[Bassi et~al.(2012)Bassi, Botti, Colombo, {Di Pietro}, and
  Tesini]{Bassi2012}
F.~Bassi, L.~Botti, A.~Colombo, D.~{Di Pietro}, and P.~Tesini.
\newblock {On the flexibility of agglomeration based physical space
  Discontinuous Galerkin discretizations}.
\newblock \emph{Journal of Computational Physics}, 231\penalty0 (1):\penalty0
  45--65, 2012.

\bibitem[Bernardi et~al.(2018)Bernardi, Dib, Girault, Hecht, Murat, and
  Sayah]{Bernardi.Dib:18}
C.~Bernardi, S.~Dib, V.~Girault, F.~Hecht, F.~Murat, and T.~Sayah.
\newblock {Finite element methods for Darcy’s problem coupled with the heat
  equation}.
\newblock \emph{Numerische Mathematik}, 139\penalty0 (1):\penalty0 315--348,
  2018.

\bibitem[Botti et~al.(2021)Botti, Botti, and {Di Pietro}]{Botti2021}
L.~Botti, M.~Botti, and D.~A. {Di Pietro}.
\newblock {An abstract analysis framework for monolithic discretisations of
  poroelasticity with application to Hybrid High-Order methods}.
\newblock \emph{Computers \& Mathematics with Applications}, 91:\penalty0
  150--175, 2021.
\newblock Robust and Reliable Finite Element Methods in Poromechanics.

\bibitem[Botti et~al.(2020{\natexlab{a}})Botti, {Di Pietro}, {Le Ma\^itre}, and
  Sochala]{Botti.Di-Pietro.ea:19}
M.~Botti, D.~A. {Di Pietro}, O.~{Le Ma\^itre}, and P.~Sochala.
\newblock Numerical approximation of poroelasticity with random coefficients
  using {Polynomial Chaos} and {H}ybrid {H}igh-{O}rder methods.
\newblock \emph{Computer Methods in Applied Mechanics and Engineering}, 361,
  2020{\natexlab{a}}.

\bibitem[Botti et~al.(2020{\natexlab{b}})Botti, Pietro, and
  Sochala]{Botti2020_korn}
M.~Botti, D.~A.~D. Pietro, and P.~Sochala.
\newblock {A Hybrid High-Order discretization method for nonlinear
  poroelasticity}.
\newblock \emph{Computational Methods in Applied Mathematics}, 20\penalty0
  (2):\penalty0 227--249, 2020{\natexlab{b}}.

\bibitem[Brenner(2004)]{Brenner2004}
S.~C. Brenner.
\newblock {Korn's inequalities for piecewise $H^1$ vector fields}.
\newblock \emph{Mathematics of Computation}, pages 1067--1087, 2004.

\bibitem[Brezzi et~al.(2004)Brezzi, Marini, and S\"{u}li]{Brezzi2004}
F.~Brezzi, L.~D. Marini, and E.~S\"{u}li.
\newblock Discontinuous {G}alerkin methods for first-order hyperbolic problems.
\newblock \emph{Mathematical Models and Methods in Applied Sciences},
  14\penalty0 (12):\penalty0 1893--1903, 2004.

\bibitem[Brun et~al.(2018)Brun, Berre, Nordbotten, and Radu]{Brun2018}
M.~K. Brun, I.~Berre, J.~M. Nordbotten, and F.~A. Radu.
\newblock {Upscaling of the coupling of hydromechanical and thermal processes
  in a quasi-static poroelastic medium}.
\newblock \emph{Transport in Porous Media}, 124\penalty0 (1):\penalty0
  137--158, 2018.

\bibitem[Brun et~al.(2019)Brun, Ahmed, Nordbotten, and Radu]{Brun2019}
M.~K. Brun, E.~Ahmed, J.~M. Nordbotten, and F.~A. Radu.
\newblock {Well-posedness of the fully coupled quasi-static thermo-poroelastic
  equations with nonlinear convective transport}.
\newblock \emph{Journal of Mathematical Analysis and Applications},
  471\penalty0 (1):\penalty0 239--266, 2019.

\bibitem[Brun et~al.(2020)Brun, Ahmed, Berre, Nordbotten, and Radu]{Brun2020}
M.~K. Brun, E.~Ahmed, I.~Berre, J.~M. Nordbotten, and F.~A. Radu.
\newblock {Monolithic and splitting solution schemes for fully coupled
  quasi-static thermo-poroelasticity with nonlinear convective transport}.
\newblock \emph{Computers \& Mathematics with Applications}, 80\penalty0
  (8):\penalty0 1964--1984, 2020.

\bibitem[Cangiani et~al.(2014)Cangiani, Georgoulis, and Houston]{Cangiani2014}
A.~Cangiani, E.~H. Georgoulis, and P.~Houston.
\newblock {$hp$-Version discontinuous Galerkin methods on polygonal and
  polyhedral meshes}.
\newblock \emph{Mathematical Models and Methods in Applied Sciences},
  24\penalty0 (10):\penalty0 2009--2041, 2014.

\bibitem[Cangiani et~al.(2017{\natexlab{a}})Cangiani, Dong, and
  Georgoulis]{Cangiani2017}
A.~Cangiani, Z.~Dong, and E.~H. Georgoulis.
\newblock {$hp$-Version Space-Time Discontinuous Galerkin methods for parabolic
  problems on prismatic meshes}.
\newblock \emph{SIAM Journal on Scientific Computing}, 39\penalty0
  (4):\penalty0 A1251--A1279, 2017{\natexlab{a}}.

\bibitem[Cangiani et~al.(2017{\natexlab{b}})Cangiani, Dong, Georgoulis, and
  Houston]{CangianiDong:17}
A.~Cangiani, Z.~Dong, E.~H. Georgoulis, and P.~Houston.
\newblock \emph{{$hp$-version Discontinuous Galerkin methods on polytopic
  meshes}}.
\newblock SpringerBriefs in Mathematics. Springer International Publishing,
  2017{\natexlab{b}}.

\bibitem[Coussy(2003)]{Coussy2003}
O.~Coussy.
\newblock \emph{{Thermoporoelasticity}}, chapter~4, pages 71--112.
\newblock John Wiley \& Sons, Ltd, 2003.
\newblock ISBN 9780470092712.

\bibitem[Di~Pietro and Ern(2012)]{DiPietro2012}
D.~Di~Pietro and A.~Ern.
\newblock \emph{{Mathematical aspects of Discontinuous Galerkin methods}}.
\newblock Springer, Berlin, Heidelberg, 2012.

\bibitem[Ern and Guermond(2021)]{Ern2021}
A.~Ern and J.~Guermond.
\newblock \emph{{Finite Elements II - Galerkin approximation, elliptic and
  mixed PDEs}}.
\newblock Springer Cham, 2021.
\newblock ISBN 978-3-030-56923-5.

\bibitem[Ern et~al.(2009)Ern, Stephansen, and Zunino]{Ern2009}
A.~Ern, A.~F. Stephansen, and P.~Zunino.
\newblock A discontinuous {G}alerkin method with weighted averages for
  advection–diffusion equations with locally small and anisotropic
  diffusivity.
\newblock \emph{IMA Journal of Numerical Analysis}, 29\penalty0 (2):\penalty0
  235--256, 2009.

\bibitem[Heinrich and Nicaise(2003)]{Heinrich2003}
B.~Heinrich and S.~Nicaise.
\newblock The {N}itsche mortar finite‐element method for transmission
  problems with singularities.
\newblock \emph{IMA Journal of Numerical Analysis}, 23\penalty0 (2):\penalty0
  331--358, 2003.

\bibitem[Heinrich and Pietsch(2002)]{Heinrich2002}
B.~Heinrich and K.~Pietsch.
\newblock Nitsche type mortaring for some elliptic problem with corner
  singularities.
\newblock \emph{Computing}, 68\penalty0 (3):\penalty0 217--238, 2002.

\bibitem[Heinrich and P{\"o}nitz(2005)]{Heinrich2005}
B.~Heinrich and K.~P{\"o}nitz.
\newblock Nitsche type mortaring for singularly perturbed reaction-diffusion
  problems.
\newblock \emph{Computing}, 75\penalty0 (4):\penalty0 257--279, 2005.

\bibitem[Houston et~al.(2002)Houston, Schwab, and Süli]{Houston2002}
P.~Houston, C.~Schwab, and E.~Süli.
\newblock Discontinuous $hp$-{F}inite {E}lement {M}ethods for
  advection-diffusion-reaction problems.
\newblock \emph{SIAM Journal on Numerical Analysis}, 39\penalty0 (6):\penalty0
  2133--2163, 2002.

\bibitem[Oyarzúa and Ruiz-Baier(2016)]{Oyarzua2016}
R.~Oyarzúa and R.~Ruiz-Baier.
\newblock {Locking-Free Finite Element methods for poroelasticity}.
\newblock \emph{SIAM Journal on Numerical Analysis}, 54\penalty0 (5):\penalty0
  2951--2973, 2016.

\bibitem[Stenberg(1998)]{Stenberg1998}
R.~Stenberg.
\newblock Mortaring by a method of {J.A. Nitsche}.
\newblock \emph{Computational mechanics}, 1998.

\bibitem[S{\"u}li et~al.(2000)S{\"u}li, Schwab, and Houston]{Suli2000}
E.~S{\"u}li, C.~Schwab, and P.~Houston.
\newblock $hp$-{DGFEM} for partial differential equations with nonnegative
  characteristic form.
\newblock In B.~Cockburn, G.~E. Karniadakis, and C.-W. Shu, editors,
  \emph{Discontinuous Galerkin Methods}, pages 221--230, Berlin, Heidelberg,
  2000. Springer Berlin Heidelberg.

\bibitem[Talischi et~al.(2012)Talischi, Paulino, Pereira, and
  Menezes]{Talischi2012}
C.~Talischi, G.~H. Paulino, A.~Pereira, and I.~F.~M. Menezes.
\newblock {PolyMesher: a general-purpose mesh generator for polygonal elements
  written in Matlab}.
\newblock \emph{Structural and Multidisciplinary Optimization}, 45\penalty0
  (3):\penalty0 309--328, 2012.

\bibitem[Wheeler(1978)]{Wheeler1978}
M.~F. Wheeler.
\newblock {An Elliptic Collocation-Finite Element method with Interior
  Penalties}.
\newblock \emph{SIAM Journal on Numerical Analysis}, 15\penalty0 (1):\penalty0
  152--161, 1978.

\end{thebibliography}

\appendix
\section{Superconvergence of the PolyDG-WSIP scheme for the linear THM problem}
\label{sec:RobQSTPE_Superconvergence}

The aim of this section is to prove that the PolyDG-WSIP scheme proposed for the THM problem is not only optimal convergence, but it also shows some superconvergence properties in the linear case. To this aim, we consider the following exact solutions:
\begin{equation}
	\begin{aligned}
		\mathbf{u}(x,y) & \ = \nu_u \ \left( \begin{aligned}
			& x^2 \cos\left(\frac{\pi x}{2}\right) \sin(\pi x) \\
			& x^2 \cos\left(\frac{\pi x}{2}\right) \sin(\pi x)
		\end{aligned} \right), \\
		p(x,y) & \ = \nu_p \ x^2 \sin(\pi x) \sin(\pi y), \\
		T(x,y) & \ = -\nu_T \ y^2 \sin(\pi x) \sin(\pi y),
	\end{aligned}
\end{equation}
through which we infer the boundary conditions and forcing terms. The parameters $\nu_u$, $\nu_p$, $\nu_T$ control the magnitude of the displacement, pressure, and temperature, respectively. The model coefficients are reported in Table~\ref{tab:RobQSTPE_TPE_params_convtest} (chosen as in the convergence test).

We observe that the optimal convergence property (even for the non-linear problem) has been already proven in Section~\ref{sec:RobQSTPE_conv_test} by setting $\nu_u = \nu_p = \nu_T = 1$.

In order to observe the better robustness of the scheme with respect to large pressure and temperatures, we propose to different tests. In the first, we set $\nu_u = 0.1$, $\nu_p = \nu_T = \num[exponent-product=\ensuremath{\cdot}, print-unity-mantissa=false]{1e+4}$, we consider the same sequence of meshes as in Section~\ref{sec:RobQSTPE_conv_test} and we test different polynomial degrees, i.e. $\ell = 2, 3, 4$. For the second test, we fix $\nu_u = 0.1$, the mesh, and the polynomial degree of approximation; then we vary the values of $\nu_p = \nu_T =  \left[ 1, \num[exponent-product=\ensuremath{\cdot}, print-unity-mantissa=false]{1e+1}, \num[exponent-product=\ensuremath{\cdot}, print-unity-mantissa=false]{1e+2}, \num[exponent-product=\ensuremath{\cdot}, print-unity-mantissa=false]{1e+3}, \num[exponent-product=\ensuremath{\cdot}, print-unity-mantissa=false]{1e+4}, \num[exponent-product=\ensuremath{\cdot}, print-unity-mantissa=false]{1e+5}, \num[exponent-product=\ensuremath{\cdot}, print-unity-mantissa=false]{1e+6} \right]$, concurrently. For the second test we consider the following couples of discretization parameters: $(N = 10000, \ \ell = 2)$, $(N = 3100, \ \ell = 3)$, $(N = 1000, \ \ell = 4)$. 

By looking at Figure~\ref{fig:RobQSTPE_SuperconvAppendix_vsH_L2} and Figure~\ref{fig:RobQSTPE_SuperconvAppendix_vsH_dG} we observe that, for all the tested polynomial degree, we have the superconvergence phenomenon for the displacement. Indeed, we observe that using polynomial degree of approximation equal to $\ell$, then the error of the displacement in $dG$-norm converges with order $\ell+1$ (we remark that the expected order is $\ell$ in this case and this rate is observed for the temperature and the pressure). Moreover, for what concerns the error in $L^2$-norm we observe $(\ell+1)+1$ convergence rate. In Figure~\ref{fig:RobQSTPE_SuperconvAppendix_vsH_L2} (right) we observe a slight plateau in the last refinement, this is due to the fact that we are reaching very low values for the error and numerical errors appear too.

In Figure~\ref{fig:RobQSTPE_SuperconvAppendix_vsNu_L2} and Figure~\ref{fig:RobQSTPE_SuperconvAppendix_vsNu_dG} we observe the behaviour of the errors with respect to increasing values of $\nu_p, \nu_T$. We see that in all the tested cases and for both the $L^2$- and $dG$-errors are way lower than the errors of pressure and temperature (even for values of $\nu_p, \nu_T$ not too big). It is interesting to notice that, for $\nu_p, \nu_T = [1, 10, 100]$ the errors of the displacement remain almost constant, while the errors of the pressure and temperature start growing.

\input{Fig/SuperconvergenceH.tikz}
\begin{center}
\begin{figure}[H]
\begin{subfigure}[b]{0.33\textwidth}
\begin{tikzpicture}
\begin{axis}[%
width=0.65\textwidth,
height=0.5\textwidth,
at={(0\textwidth,0\textwidth)},
scale only axis,
xmode=log,
xmin=0.99,
xmax=1000001,
xminorticks=true,
xlabel={$\nu_p, \nu_T$},
ymode=log,
ymin=1e-8,
ymax=0.3,
yminorticks=true,
ylabel={$L^2$-Errors},
legend style={draw=none,fill=none,legend cell align=left},
legend pos=north west
]
\addplot [color=myred,solid,line width=1.5pt, mark=diamond*,mark options={color=myred}]
  table[row sep=crcr]{
1  2.4117e-08   \\
10  2.63658e-08   \\
100  1.126826e-07   \\
1000  1.1042621e-06   \\
10000  1.10432653e-05   \\
100000  0.0001104356695   \\
1000000   0.0011043602045   \\
};
\addlegendentry{\footnotesize{$\mathbf{u}_h(\mathbf{x},t)$}}

\addplot [color=myblue,solid,line width=1.5pt,mark=square*,mark options={color=myblue}]
  table[row sep=crcr]{
1  1.9018e-07   \\
10  1.90175e-06   \\
100  1.901748e-05   \\
1000  0.00019017476   \\
10000   0.00190174755   \\
100000    0.01901747538   \\
1000000     0.19017475578   \\
};
\addlegendentry{\footnotesize{$p_h(\mathbf{x},t)$}}

\addplot [color=mygreen,solid,line width=1.5pt,mark=triangle*,mark options={color=mygreen}]
  table[row sep=crcr]{
1  1.9126e-07   \\
10  1.91256e-06   \\
100  1.912564e-05   \\
1000  0.00019125636   \\
10000   0.00191256357   \\
100000    0.01912563589   \\
1000000     0.19125635955   \\
};
\addlegendentry{\footnotesize{$T_h(\mathbf{x},t)$}}

\end{axis}
\end{tikzpicture}
\end{subfigure}
\begin{subfigure}[b]{0.33\textwidth}
\begin{tikzpicture}
\begin{axis}[%
width=0.65\textwidth,
height=0.5\textwidth,
at={(0\textwidth,0\textwidth)},
scale only axis,
xmode=log,
xmin=0.99,
xmax=1000001,
xminorticks=true,
xlabel={$\nu_p, \nu_T$},
ymode=log,
ymin=9e-10,
ymax=0.02,
yminorticks=true,
ylabel={$L^2$-Errors},
legend style={draw=none,fill=none,legend cell align=left},
legend pos=north west
]
\addplot [color=myred,solid,line width=1.5pt, mark=diamond*,mark options={color=myred}]
  table[row sep=crcr]{
1  1.0704e-09   \\
10   1.0781e-09   \\
100  1.646648e-09   \\
1000  1.2511399e-08   \\
10000  1.24610706e-07   \\
100000  1.246016887e-06   \\
1000000   1.246013997e-05   \\
};
\addlegendentry{$\mathbf{u}_h(\mathbf{x},t)$}

\addplot [color=myblue,solid,line width=1.5pt,mark=square*,mark options={color=myblue}]
  table[row sep=crcr]{
1  1.0707e-08   \\
10  1.07066e-07   \\
100  1.070661e-06   \\
1000  1.0706606e-05   \\
10000  0.000107066063   \\
100000   0.001070660636   \\
1000000    0.010706606375   \\
};
\addlegendentry{$p_h(\mathbf{x},t)$}

\addplot [color=mygreen,solid,line width=1.5pt,mark=triangle*,mark options={color=mygreen}]
  table[row sep=crcr]{
1   1.155e-08   \\
10  1.15503e-07   \\
100  1.155032e-06   \\
1000  1.1550321e-05   \\
10000  0.000115503208   \\
100000   0.001155032084   \\
1000000    0.011550320843   \\
};
\addlegendentry{$T_h(\mathbf{x},t)$}

\end{axis}
\end{tikzpicture}
\end{subfigure}
\begin{subfigure}[b]{0.33\textwidth}
\begin{tikzpicture}
\begin{axis}[%
width=0.65\textwidth,
height=0.5\textwidth,
at={(0\textwidth,0\textwidth)},
scale only axis,
xmode=log,
xmin=0.99,
xmax=1000001,
xminorticks=true,
xlabel={$\nu_p, \nu_T$},
ymode=log,
ymin=1e-10,
ymax=0.002,
yminorticks=true,
ylabel={$L^2$-Errors},
legend style={draw=none,fill=none,legend cell align=left},
legend pos=north west
]
\addplot [color=myred,solid,line width=1.5pt, mark=diamond*,mark options={color=myred}]
  table[row sep=crcr]{
1   2.292e-10   \\
10   2.3142e-10   \\
100  4.490305e-10   \\
1000  3.9329955e-09   \\
10000  3.93286097e-08   \\
100000  3.933447902e-07   \\
1000000  3.9335112038e-06   \\
};
\addlegendentry{$\mathbf{u}_h(\mathbf{x},t)$}

\addplot [color=myblue,solid,line width=1.5pt,mark=square*,mark options={color=myblue}]
  table[row sep=crcr]{
1  1.9689e-09   \\
10   1.9689e-08   \\
100  1.968898e-07   \\
1000  1.9688976e-06   \\
10000  1.96889757e-05   \\
100000  0.0001968897578   \\
1000000   0.0019688975746   \\
};
\addlegendentry{$p_h(\mathbf{x},t)$}

\addplot [color=mygreen,solid,line width=1.5pt,mark=triangle*,mark options={color=mygreen}]
  table[row sep=crcr]{
1  1.7755e-09   \\
10  1.77553e-08   \\
100  1.775532e-07   \\
1000  1.7755315e-06   \\
10000  1.77553151e-05   \\
100000  0.0001775531535   \\
1000000   0.0017755315303   \\
};
\addlegendentry{$T_h(\mathbf{x},t)$}

\end{axis}
\end{tikzpicture}
\end{subfigure}
\caption{Superconvergence test: computed errors in $L^2$-norm versus $\nu_p, \nu_T$ (\textit{log-log} scale) using as polynomial degrees of approximation and number of elements $(\ell = 2, N = 10000)$ (left), $(\ell = 3, N = 3100)$ (center), and $(\ell = 4, N = 1000)$ (right).}
\label{fig:RobQSTPE_SuperconvAppendix_vsNu_L2}
\end{figure}
\end{center}

\begin{center}
\begin{figure}[H]
\begin{subfigure}[b]{0.33\textwidth}
\begin{tikzpicture}
\begin{axis}[%
width=0.65\textwidth,
height=0.5\textwidth,
at={(0\textwidth,0\textwidth)},
scale only axis,
xmode=log,
xmin=0.99,
xmax=1000001,
xminorticks=true,
xlabel={$\nu_p, \nu_T$},
ymode=log,
ymin=5e-5,
ymax=400,
yminorticks=true,
ylabel={$dG$-Errors},
legend style={draw=none,fill=none,legend cell align=left},
legend pos=north west
]
\addplot [color=myred,solid,line width=1.5pt, mark=diamond*,mark options={color=myred}]
  table[row sep=crcr]{
1  7.0178e-05   \\
10  7.01709e-05   \\
100  7.100584e-05   \\
1000  0.00013770447   \\
10000   0.00119681323   \\
100000    0.01195765878   \\
1000000     0.11958445883   \\
};
\addlegendentry{\footnotesize{$\mathbf{u}_h(\mathbf{x},t)$}}

\addplot [color=myblue,solid,line width=1.5pt,mark=square*,mark options={color=myblue}]
  table[row sep=crcr]{
1  0.00033979   \\
10   0.00339794   \\
100    0.03397935   \\
1000      0.3397935   \\
10000      3.39793505   \\
100000      33.97935049   \\
1000000      339.79350487   \\
};
\addlegendentry{\footnotesize{$p_h(\mathbf{x},t)$}}

\addplot [color=mygreen,solid,line width=1.5pt,mark=triangle*,mark options={color=mygreen}]
  table[row sep=crcr]{
1  0.00033801   \\
10   0.00338014   \\
100    0.03380144   \\
1000     0.33801437   \\
10000      3.38014374   \\
100000      33.80143735   \\
1000000      338.01437354   \\
};
\addlegendentry{\footnotesize{$T_h(\mathbf{x},t)$}}

\end{axis}
\end{tikzpicture}
\end{subfigure}
\begin{subfigure}[b]{0.33\textwidth}
\begin{tikzpicture}
\begin{axis}[%
width=0.65\textwidth,
height=0.5\textwidth,
at={(0\textwidth,0\textwidth)},
scale only axis,
xmode=log,
xmin=0.99,
xmax=1000001,
xminorticks=true,
xlabel={$\nu_p, \nu_T$},
ymode=log,
ymin=1e-6,
ymax=20,
yminorticks=true,
ylabel={$dG$-Errors},
legend style={draw=none,fill=none,legend cell align=left},
legend pos=north west
]
\addplot [color=myred,solid,line width=1.5pt, mark=diamond*,mark options={color=myred}]
  table[row sep=crcr]{
1  2.3551e-06   \\
10  2.35641e-06   \\
100  2.470975e-06   \\
1000   7.783547e-06   \\
10000   7.4162761e-05   \\
100000  0.0007411954496   \\
1000000   0.0074118555575   \\
};
\addlegendentry{\footnotesize{$\mathbf{u}_h(\mathbf{x},t)$}}

\addplot [color=myblue,solid,line width=1.5pt,mark=square*,mark options={color=myblue}]
  table[row sep=crcr]{
1  1.7834e-05   \\
10   0.00017834   \\
100   0.001783397   \\
1000    0.017833974   \\
10000     0.178339742   \\
100000       1.78339742   \\
1000000      17.833974201   \\
};
\addlegendentry{\footnotesize{$p_h(\mathbf{x},t)$}}

\addplot [color=mygreen,solid,line width=1.5pt,mark=triangle*,mark options={color=mygreen}]
  table[row sep=crcr]{
1  1.9092e-05   \\
10  0.000190918   \\
100   0.001909183   \\
1000    0.019091832   \\
10000     0.190918319   \\
100000      1.909183187   \\
1000000      19.091831874   \\
};
\addlegendentry{\footnotesize{$T_h(\mathbf{x},t)$}}

\end{axis}
\end{tikzpicture}
\end{subfigure}
\begin{subfigure}[b]{0.33\textwidth}
\begin{tikzpicture}
\begin{axis}[%
width=0.65\textwidth,
height=0.5\textwidth,
at={(0\textwidth,0\textwidth)},
scale only axis,
xmode=log,
xmin=0.99,
xmax=1000001,
xminorticks=true,
xlabel={$\nu_p, \nu_T$},
ymode=log,
ymin=1e-7,
ymax=2,
yminorticks=true,
ylabel={$dG$-Errors},
legend style={draw=none,fill=none,legend cell align=left},
legend pos=north west
]
\addplot [color=myred,solid,line width=1.5pt, mark=diamond*,mark options={color=myred}]
  table[row sep=crcr]{
1  3.0694e-07   \\
10  3.07014e-07   \\
100  3.314064e-07   \\
1000  1.3327708e-06   \\
10000  1.30197156e-05   \\
100000  0.0001302079254   \\
1000000   0.0013021222855   \\
};
\addlegendentry{\footnotesize{$\mathbf{u}_h(\mathbf{x},t)$}}

\addplot [color=myblue,solid,line width=1.5pt,mark=square*,mark options={color=myblue}]
  table[row sep=crcr]{
1  1.5388e-06   \\
10  1.53879e-05   \\
100  0.0001538792   \\
1000   0.0015387924   \\
10000    0.0153879242   \\
100000     0.1538792418   \\
1000000      1.5387924245   \\
};
\addlegendentry{\footnotesize{$p_h(\mathbf{x},t)$}}

\addplot [color=mygreen,solid,line width=1.5pt,mark=triangle*,mark options={color=mygreen}]
  table[row sep=crcr]{
1  1.5858e-06   \\
10  1.58577e-05   \\
100  0.0001585768   \\
1000   0.0015857683   \\
10000    0.0158576836   \\
100000     0.1585768363   \\
1000000      1.5857683614   \\
};
\addlegendentry{\footnotesize{$T_h(\mathbf{x},t)$}}

\end{axis}
\end{tikzpicture}
\end{subfigure}
\caption{Superconvergence test: computed errors in $dG$-norm versus $\nu_p, \nu_T$ (\textit{log-log} scale) using as polynomial degrees of approximation and number of elements $(\ell = 2, N = 10000)$ (left), $(\ell = 3, N = 3100)$ (center), and $(\ell = 4, N = 1000)$ (right).}
\label{fig:RobQSTPE_SuperconvAppendix_vsNu_dG}
\end{figure}
\end{center}

\section{Robustness tests: further results}
\label{sec:RobQSTPE_rob_test_appendix_further_results}

In this supplementary material Section we report the results concerning the robustness tests that are not presented in Section 5.2.

\subsection{Robustness with respect to degenerating thermal conductivity}

In this section we present the results of the tests presented in Section 5.2.1 for the configurations $(\ell = 2, N = 1000)$ and $(\ell = 4, N = 100)$. We recall that we consider the following values for the model's coefficients: we wide $\boldsymbol{\Theta} = \{1 \mathbf{I}, \num[exponent-product=\ensuremath{\cdot}, print-unity-mantissa=false]{1e-2} \mathbf{I}, \num[exponent-product=\ensuremath{\cdot}, print-unity-mantissa=false]{1e-4} \mathbf{I}, \num[exponent-product=\ensuremath{\cdot}, print-unity-mantissa=false]{1e-6} \mathbf{I}, \num[exponent-product=\ensuremath{\cdot}, print-unity-mantissa=false]{1e-8} \mathbf{I}, \num[exponent-product=\ensuremath{\cdot}, print-unity-mantissa=false]{1e-10} \mathbf{I} \}$ and keep constant $a_0 = b_0 = c_0 = 0$, $\mathbf{K} = 1 \mathbf{I}$. The other physical parameters are taken as in Section 5.1. We observe that the behavior is similar to the configuration $(\ell = 3, N = 310)$ presented in the article. We point out just two observations about these results. First, we notice that for specific values of the physical and the discretization parameters (cf. Figure~\ref{fig:RobQSTPE_RobThetaErrors_P2}), the stabilized scheme leads to a larger error than the unstabilized case. This is expected when stabilization is not necessary. Second, we observe that increasing the polynomial degree ($\ell = 4$, cf. Figure~\ref{fig:RobQSTPE_RobThetaErrors_P4}), we get equivalent results for $\mathcal{C}_h$ and $\mathcal{C}_h^{\text{stab}}$ in terms of displacement and pressure field (moreover, the algorithm converges in both the cases, cf. Table 3). This may be an additional hint that increasing the polynomial degree helps us in solving correctly the non-linearity (for the linearizations proposed in this article).

\input{Fig/RobTheta_p2_Appendix.tikz}
\input{Fig/RobTheta_p4_Appendix.tikz}

\subsection{Robustness with respect to degenerating permeability}
In this section we present the results of the tests presented in Section 5.2.2 for the displacement $\mathbf{u}_h$ and the temperature $T_h$ in the case of degenerating permeability. We recall that we consider the following values for the model's coefficients: $\mathbf{K} = \num[exponent-product=\ensuremath{\cdot}, print-unity-mantissa=false]{1e-10} \mathbf{I}$, $\boldsymbol{\Theta} = 1 \mathbf{I}$, and $a_0 = b_0 = c_0 = 0$. The other physical parameters are taken as in Section 5.1. We take $\ell = 2$ as polynomial degree of approximation and a sequence of Voronoi meshes with $N = [100, 310, 1000, 3100, 10000]$ elements.

We observe that, as expected, the results are very similar (essentially equivalent) for the four linearizations. By using polynomial degree of approximation $\ell = 2$, we see that the $dG$-errors decay as $h^{\ell} = h^2$ as the very low value of the permeability coefficient does not enter in the definition of the norms for the displacement and the temperature. Morevoer, it is interesting to notice that for these fields we observe $\ell+1$ order of accuracy for the error in $L^2$-norm.

\begin{figure}[H]
\centering
\begin{subfigure}[b]{1\textwidth}
\begin{subfigure}[b]{0.49\textwidth}
\begin{tikzpicture}
\begin{axis}[%
width=0.65\textwidth,
height=0.5\textwidth,
at={(0\textwidth,0\textwidth)},
scale only axis,
xmode=log,
xmin=4,
xmax=75,
xminorticks=true,
xlabel={$1/h$},
ymode=log,
ymin=1e-7,
ymax=5e-4,
yminorticks=true,
ylabel={$L^2$-Errors ($\mathbf{u}_h$)},
legend style={draw=none,fill=none,legend cell align=left},
legend pos=south west
]
\addplot [color=myyellow,solid,line width=1.5pt, mark=diamond*,mark options={color=myyellow}]
  table[row sep=crcr]{
5.5214   0.00031584 \\
9.7539   5.2991e-05 \\
17.5835   8.4934e-06 \\
30.4349   1.3342e-06 \\
55.1835   2.2999e-07 \\
};
\addlegendentry{$\mathcal{C}^{\text{old}}_h$}

\addplot [color=myblue,solid,line width=1.5pt,mark=square*,mark options={color=myblue}]
  table[row sep=crcr]{
5.5214   0.00031584 \\
9.7539   5.2991e-05 \\
17.5835   8.4934e-06 \\
30.4349   1.3342e-06 \\
55.1835   2.2999e-07 \\
};
\addlegendentry{$\mathcal{C}^{\text{vol}}_h$}

\addplot [color=mygreen,solid,line width=1.5pt,mark=triangle*,mark options={color=mygreen}]
  table[row sep=crcr]{
5.5214   0.00031584 \\
9.7539   5.2991e-05 \\
17.5835   8.4934e-06 \\
30.4349   1.3342e-06 \\
55.1835   2.2999e-07 \\
};
\addlegendentry{$\mathcal{C}_h$}

\addplot [color=myred,solid,line width=1.5pt,mark=triangle*,mark options={color=myred}]
  table[row sep=crcr]{
5.5214   0.00031584 \\
9.7539   5.2991e-05 \\
17.5835   8.4934e-06 \\
30.4349   1.3342e-06 \\
55.1835   2.2999e-07 \\  
};
\addlegendentry{$\mathcal{C}^{\text{stab}}_h$}

\addplot [color=black,solid,line width=0.5pt]
  table[row sep=crcr]{
    30.4349      2.9805e-06 \\
    55.1835      2.9805e-06 \\
    55.1835      5e-7 \\
    30.4349      2.9805e-06 \\  
};
\node[right, align=left, text=black, font=\footnotesize]
at (axis cs:55.1835, 1.7402e-06) {3}; 

\end{axis}
\end{tikzpicture}
\end{subfigure}
\begin{subfigure}[b]{0.49\textwidth}
\begin{tikzpicture}
\begin{axis}[%
width=0.65\textwidth,
height=0.5\textwidth,
at={(0\textwidth,0\textwidth)},
scale only axis,
xmode=log,
xmin=4,
xmax=75,
xminorticks=true,
xlabel={$1/h$},
ymode=log,
ymin=3e-4,
ymax=0.1,
yminorticks=true,
ylabel={$dG$-Errors ($\mathbf{u}_h$)},
legend style={draw=none,fill=none,legend cell align=left},
legend pos=south west
]
\addplot [color=myyellow,solid,line width=1.5pt, mark=diamond*,mark options={color=myyellow}]
  table[row sep=crcr]{
5.5214   0.067706 \\
9.7539   0.020397 \\
17.5835   0.0064744 \\
30.4349   0.0020609 \\
55.1835   0.00063986 \\
};
\addlegendentry{$\mathcal{C}^{\text{old}}_h$}

\addplot [color=myblue,solid,line width=1.5pt,mark=square*,mark options={color=myblue}]
  table[row sep=crcr]{
5.5214   0.067706 \\
9.7539   0.020397 \\
17.5835   0.0064744 \\
30.4349   0.0020609 \\
55.1835   0.00063986 \\
};
\addlegendentry{$\mathcal{C}^{\text{vol}}_h$}

\addplot [color=mygreen,solid,line width=1.5pt,mark=triangle*,mark options={color=mygreen}]
  table[row sep=crcr]{
5.5214   0.067706 \\
9.7539   0.020397 \\
17.5835   0.0064744 \\
30.4349   0.0020609 \\
55.1835   0.00063986 \\
};
\addlegendentry{$\mathcal{C}_h$}

\addplot [color=myred,solid,line width=1.5pt,mark=triangle*,mark options={color=myred}]
  table[row sep=crcr]{
5.5214   0.067706 \\
9.7539   0.020397 \\
17.5835   0.0064744 \\
30.4349   0.0020609 \\
55.1835   0.00063986 \\
};
\addlegendentry{$\mathcal{C}^{\text{stab}}_h$}

\addplot [color=black,solid,line width=0.5pt]
  table[row sep=crcr]{
    30.4349      0.003 \\
    55.1835      0.003 \\
    55.1835      9e-4 \\
    30.4349      0.003 \\  
};
\node[right, align=left, text=black, font=\footnotesize]
at (axis cs:55.1835, 0.0019) {2};

\end{axis}
\end{tikzpicture}
\end{subfigure}
\caption{Robustness test vs $\mathbf{K} = k \mathbf{I}$: computed errors for the displacement $\mathbf{u}_h$.}
\label{fig:RobQSTPE_RobKappaErrors_u_P2}
\end{subfigure}
\vspace{0.6cm}

\begin{subfigure}[b]{1\textwidth}
\centering
\begin{subfigure}[b]{0.49\textwidth}
\begin{tikzpicture}
\begin{axis}[%
width=0.65\textwidth,
height=0.5\textwidth,
at={(0\textwidth,0\textwidth)},
scale only axis,
xmode=log,
xmin=4,
xmax=75,
xminorticks=true,
xlabel={$1/h$},
ymode=log,
ymin=1e-7,
ymax=6e-4,
yminorticks=true,
ylabel={$L^2$-Errors ($T_h$)},
legend style={draw=none,fill=none,legend cell align=left},
legend pos=south west
]
\addplot [color=myyellow,solid,line width=1.5pt, mark=diamond*,mark options={color=myyellow}]
  table[row sep=crcr]{
5.5214   0.00039281 \\
9.7539   4.7998e-05 \\
17.5835   7.3684e-06 \\
30.4349   1.1819e-06 \\
55.1835   1.9126e-07 \\
};
\addlegendentry{$\mathcal{C}^{\text{old}}_h$}

\addplot [color=myblue,solid,line width=1.5pt,mark=square*,mark options={color=myblue}]
  table[row sep=crcr]{
5.5214   0.00039281 \\
9.7539   4.7998e-05 \\
17.5835   7.3684e-06 \\
30.4349   1.1819e-06 \\
55.1835   1.9126e-07 \\
};
\addlegendentry{$\mathcal{C}^{\text{vol}}_h$}

\addplot [color=mygreen,solid,line width=1.5pt,mark=triangle*,mark options={color=mygreen}]
  table[row sep=crcr]{
5.5214   0.00039281 \\
9.7539   4.7998e-05 \\
17.5835   7.3684e-06 \\
30.4349   1.1819e-06 \\
55.1835   1.9126e-07 \\
};
\addlegendentry{$\mathcal{C}_h$}

\addplot [color=myred,solid,line width=1.5pt,mark=triangle*,mark options={color=myred}]
  table[row sep=crcr]{
5.5214   0.00039281 \\
9.7539   4.7998e-05 \\
17.5835   7.3684e-06 \\
30.4349   1.1819e-06 \\
55.1835   1.9126e-07 \\
};
\addlegendentry{$\mathcal{C}^{\text{stab}}_h$}

\addplot [color=black,solid,line width=0.5pt]
  table[row sep=crcr]{
    30.4349      1.7883e-06 \\
    55.1835      1.7883e-06 \\
    55.1835      3e-7 \\
    30.4349      1.7883e-06 \\  
};
\node[right, align=left, text=black, font=\footnotesize]
at (axis cs:55.1835, 1.0441e-06) {3}; 

\end{axis}
\end{tikzpicture}
\end{subfigure}
\begin{subfigure}[b]{0.49\textwidth}
\begin{tikzpicture}
\begin{axis}[%
width=0.65\textwidth,
height=0.5\textwidth,
at={(0\textwidth,0\textwidth)},
scale only axis,
xmode=log,
xmin=4,
xmax=75,
xminorticks=true,
xlabel={$1/h$},
ymode=log,
ymin=2e-4,
ymax=0.05,
yminorticks=true,
ylabel={$dG$-Errors ($T_h$)},
legend style={draw=none,fill=none,legend cell align=left},
legend pos=south west
]
\addplot [color=myyellow,solid,line width=1.5pt, mark=diamond*,mark options={color=myyellow}]
  table[row sep=crcr]{
5.5214   0.037435 \\
9.7539   0.011988 \\
17.5835   0.0034051 \\
30.4349   0.0011237 \\
55.1835   0.00033801 \\
};
\addlegendentry{$\mathcal{C}^{\text{old}}_h$}

\addplot [color=myblue,solid,line width=1.5pt,mark=square*,mark options={color=myblue}]
  table[row sep=crcr]{
5.5214   0.037435 \\
9.7539   0.011988 \\
17.5835   0.0034051 \\
30.4349   0.0011237 \\
55.1835   0.00033801 \\
};
\addlegendentry{$\mathcal{C}^{\text{vol}}_h$}

\addplot [color=mygreen,solid,line width=1.5pt,mark=triangle*,mark options={color=mygreen}]
  table[row sep=crcr]{
5.5214   0.037435 \\
9.7539   0.011988 \\
17.5835   0.0034051 \\
30.4349   0.0011237 \\
55.1835   0.00033801 \\
};
\addlegendentry{$\mathcal{C}_h$}

\addplot [color=myred,solid,line width=1.5pt,mark=triangle*,mark options={color=myred}]
  table[row sep=crcr]{
5.5214   0.037435 \\
9.7539   0.011988 \\
17.5835   0.0034051 \\
30.4349   0.0011237 \\
55.1835   0.00033801 \\
};
\addlegendentry{$\mathcal{C}^{\text{stab}}_h$}

\addplot [color=black,solid,line width=0.5pt]
  table[row sep=crcr]{
    30.4349      0.0016 \\
    55.1835      0.0016 \\
    55.1835      5e-4 \\
    30.4349      0.0016 \\  
};
\node[right, align=left, text=black, font=\footnotesize]
at (axis cs:55.1835, 0.0011) {2}; 

\end{axis}
\end{tikzpicture}
\end{subfigure}
\caption{Robustness test vs $\mathbf{K} = k \mathbf{I}$: computed errors for the temperature $T_h$.}
\label{fig:RobQSTPE_RobKappaErrors_T_P2}
\end{subfigure}
\caption{Test case of Section 5.2.2 (robustness w.r.t. degenerating permeability): computed errors in $L^2$-norm (left column) and $dG$-norms (right column) versus $1/h$ (\textit{log-log} scale). The different colors represent the different linearization schemes. The polynomial degree of approximation is $\ell = 2$.}
\label{fig:RobQSTPE_RobKappaErrors_P2_Appendix}
\end{figure}

\subsection{Robustness with respect to degenerating thermal conductivity and permeability}
In this section we present the results of the tests presented in Section 5.2.3 for the displacement $\mathbf{u}_h$ in the case of degenerating permeability and thermal conductivity. We recall that we consider the following values for the model's coefficients: $\mathbf{K} = \num[exponent-product=\ensuremath{\cdot}, print-unity-mantissa=false]{1e-10}  \mathbf{I}$, $\boldsymbol{\Theta} = \num[exponent-product=\ensuremath{\cdot}, print-unity-mantissa=false]{1e-10}  \mathbf{I}$, and $a_0 = b_0 = c_0 = 0.01$. The other physical parameters are taken as in Section 5.1. We take $\ell = 2$ as polynomial degree of approximation and a sequence of Voronoi meshes with $N = [100, 310, 1000, 3100, 10000]$ elements.

Similarly to the previous Section, we observe that by using $\ell$ as polynomial degree of approximation for the four variables, we get $h^{\ell}$ and $h^{\ell+1}$ error decay for the $dG$- and $L^2$-errors of the displacement, respectively. 

\begin{figure}[H]
\centering
\begin{subfigure}[b]{0.49\textwidth}
\begin{tikzpicture}
\begin{axis}[%
width=0.65\textwidth,
height=0.5\textwidth,
at={(0\textwidth,0\textwidth)},
scale only axis,
xmode=log,
xmin=4,
xmax=75,
xminorticks=true,
xlabel={$1/h$},
ymode=log,
ymin=1e-7,
ymax=5e-4,
yminorticks=true,
ylabel={$L^2$-Errors},
legend style={draw=none,fill=none,legend cell align=left},
legend pos=south west
]
\addplot [color=myyellow,solid,line width=1.5pt, mark=diamond*,mark options={color=myyellow}]
  table[row sep=crcr]{
5.5214   0.00031584 \\
9.7539   5.2991e-05 \\
17.5835  8.4934e-06 \\
30.4349  1.3342e-06 \\
55.1835  2.2999e-07 \\
};
\addlegendentry{$\mathcal{C}^{\text{old}}_h$}

\addplot [color=myblue,solid,line width=1.5pt,mark=square*,mark options={color=myblue}]
  table[row sep=crcr]{
5.5214   0.00031584 \\
9.7539   5.2991e-05 \\
17.5835  8.4934e-06 \\
30.4349  1.3342e-06 \\
55.1835  2.2999e-07 \\
};
\addlegendentry{$\mathcal{C}^{\text{vol}}_h$}

\addplot [color=mygreen,solid,line width=1.5pt,mark=triangle*,mark options={color=mygreen}]
  table[row sep=crcr]{
5.5214   0.00031584 \\
9.7539   5.2991e-05 \\
17.5835  8.4934e-06 \\
30.4349  1.3342e-06 \\
55.1835  2.2999e-07 \\
};
\addlegendentry{$\mathcal{C}_h$}

\addplot [color=myred,solid,line width=1.5pt,mark=triangle*,mark options={color=myred}]
  table[row sep=crcr]{
5.5214   0.00031584 \\
9.7539   5.2991e-05 \\
17.5835  8.4934e-06 \\
30.4349  1.3342e-06 \\
55.1835  2.2999e-07 \\
};
\addlegendentry{$\mathcal{C}^{\text{stab}}_h$}

\addplot [color=black,solid,line width=0.5pt]
  table[row sep=crcr]{
    30.4349      2.3844e-06 \\
    55.1835      2.3844e-06 \\
    55.1835      4e-7 \\
    30.4349      2.3844e-06 \\  
};
\node[right, align=left, text=black, font=\footnotesize]
at (axis cs:55.1835, 1.3922e-06) {3}; 

\end{axis}
\end{tikzpicture}
\end{subfigure}
\begin{subfigure}[b]{0.49\textwidth}
\begin{tikzpicture}
\begin{axis}[%
width=0.65\textwidth,
height=0.5\textwidth,
at={(0\textwidth,0\textwidth)},
scale only axis,
xmode=log,
xmin=4,
xmax=75,
xminorticks=true,
xlabel={$1/h$},
ymode=log,
ymin=4e-4,
ymax=0.09,
yminorticks=true,
ylabel={$dG$-Errors},
legend style={draw=none,fill=none,legend cell align=left},
legend pos=south west
]
\addplot [color=myyellow,solid,line width=1.5pt, mark=diamond*,mark options={color=myyellow}]
  table[row sep=crcr]{
5.5214   0.067706 \\
9.7539   0.020397 \\
17.5835  0.0064744 \\
30.4349  0.0020609 \\
55.1835  0.00063986 \\
};
\addlegendentry{$\mathcal{C}^{\text{old}}_h$}

\addplot [color=myblue,solid,line width=1.5pt,mark=square*,mark options={color=myblue}]
  table[row sep=crcr]{
5.5214   0.067706 \\
9.7539   0.020397 \\
17.5835  0.0064744 \\
30.4349  0.0020609 \\
55.1835  0.00063986 \\
};
\addlegendentry{$\mathcal{C}^{\text{vol}}_h$}

\addplot [color=mygreen,solid,line width=1.5pt,mark=triangle*,mark options={color=mygreen}]
  table[row sep=crcr]{
5.5214   0.067706 \\
9.7539   0.020397 \\
17.5835  0.0064744 \\
30.4349  0.0020609 \\
55.1835  0.00063986 \\
};
\addlegendentry{$\mathcal{C}_h$}

\addplot [color=myred,solid,line width=1.5pt,mark=triangle*,mark options={color=myred}]
  table[row sep=crcr]{
5.5214   0.067706 \\
9.7539   0.020397 \\
17.5835  0.0064744 \\
30.4349  0.0020609 \\
55.1835  0.00063986 \\
};
\addlegendentry{$\mathcal{C}^{\text{stab}}_h$}

\addplot [color=black,solid,line width=0.5pt]
  table[row sep=crcr]{
    30.4349      0.0026 \\
    55.1835      0.0026 \\
    55.1835      8e-4 \\
    30.4349      0.0026 \\  
};
\node[right, align=left, text=black, font=\footnotesize]
at (axis cs:55.1835, 0.0017) {2}; 

\end{axis}
\end{tikzpicture}
\end{subfigure}
\caption{Test case of Section 5.2.3 (robustness w.r.t. degenerating permeability and thermal conductivity): computed errors for the displacement $\mathbf{u}_h$ in $L^2$-norm (left) and $dG$-norm (right) versus $1/h$ (\textit{log-log} scale). The different colors represent the different linearization schemes. The polynomial degree of approximation is $\ell = 2$.}
\label{fig:RobQSTPE_RobThetaKappaErrors_u_P2}
\end{figure}
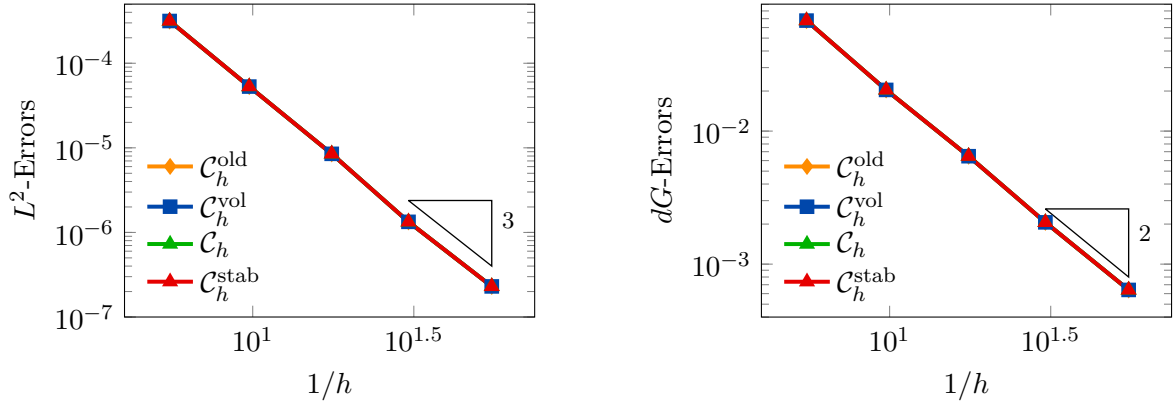

\section{Robustness tests: error tables}
\label{sec:RobQSTPE_rob_test_appendix_table_errors}

In this suppplementary material, we report the tables of the numerical values of the errors for the test presented in Section 5.2. We refer the reader to Section 5.2, Section~\ref{sec:RobQSTPE_rob_test_appendix_further_results} for the comments on the results.

\begin{table}[!h]
	\centering 
	\footnotesize
	\begin{tabular}{c | l | c | c | c | c | c | c}
		\textbf{Lin.} & \textbf{Error} & $\theta = 1$ & $\theta = \num[exponent-product=\ensuremath{\cdot}, print-unity-mantissa=false]{1e-2}$ & $\theta = \num[exponent-product=\ensuremath{\cdot}, print-unity-mantissa=false]{1e-4}$ & $\theta =  \num[exponent-product=\ensuremath{\cdot}, print-unity-mantissa=false]{1e-6}$ & $\theta = \num[exponent-product=\ensuremath{\cdot}, print-unity-mantissa=false]{1e-8}$ & $\theta = \num[exponent-product=\ensuremath{\cdot}, print-unity-mantissa=false]{1e-10}$ \T\B \\
		\hline
		\multirow{6}{*}{{$\mathcal{C}_h^{\text{old}}$}}
		& $||\mathbf{e}^{u}||_{L^2}$  & \cellcolor{green!30}\num[exponent-product=\ensuremath{\cdot}]{8.564e-06} & 1.587 & 7.460 & 1.741 & 2.208 & 1.877 \T\B \\
		& $||\mathbf{e}^{u}||_{dG}$  & \cellcolor{green!30}0.007 & 12.457 & \num[exponent-product=\ensuremath{\cdot}]{1.187e+3} & 70.483 & 112.891 & 96.352 \T\B \\
		& $||e^{p}||_{L^2}$  & \cellcolor{green!30}\num[exponent-product=\ensuremath{\cdot}]{7.097e-06} & 0.160 & 0.181 & 0.169 & 0.168 & 0.169 \T\B \\
		& $||e^{p}||_{dG}$  & 0.004 & 0.896 & 5.070 & 0.959 & 0.977 & 1.225 \T\B \\
		& $||e^{T}||_{L^2}$  & \cellcolor{green!30}\num[exponent-product=\ensuremath{\cdot}]{7.291e-06} & 199.905 & \num[exponent-product=\ensuremath{\cdot}]{1.449e+4} & \num[exponent-product=\ensuremath{\cdot}]{3.925e+3} & \num[exponent-product=\ensuremath{\cdot}]{2.357e+3} & \num[exponent-product=\ensuremath{\cdot}]{8.567e+5} \T\B \\
		& $||e^{T}||_{dG}$  & \cellcolor{green!30}0.003 & 169.448 & \num[exponent-product=\ensuremath{\cdot}]{9.173e+4} & \num[exponent-product=\ensuremath{\cdot}]{1.843e+3} & 272.656 & 825.833 \T\B \\
		\hline
		\multirow{6}{*}{{$\mathcal{C}_h^{\text{vol}}$}}
		& $||\mathbf{e}^{u}||_{L^2}$  & \cellcolor{green!30}\num[exponent-product=\ensuremath{\cdot}]{8.564e-06} & \cellcolor{green!30}\num[exponent-product=\ensuremath{\cdot}]{1.035e-05} & \cellcolor{green!30}\num[exponent-product=\ensuremath{\cdot}]{5.197e-05} & 0.005 & 0.033 & 0.074 \T\B \\
		& $||\mathbf{e}^{u}||_{dG}$  & \cellcolor{green!30}0.007 & \cellcolor{green!30}0.007 & \cellcolor{green!30}0.007 & 0.072 & 0.244 & 0.678 \T\B \\
		& $||e^{p}||_{L^2}$  & \cellcolor{green!30}\num[exponent-product=\ensuremath{\cdot}]{7.097e-06} & \num[exponent-product=\ensuremath{\cdot}]{7.068e-06} & \cellcolor{green!30}\num[exponent-product=\ensuremath{\cdot}]{1.043e-05} & \num[exponent-product=\ensuremath{\cdot}]{4.477e-4} & 0.001 & 0.003 \T\B \\
		& $||e^{p}||_{dG}$  & \cellcolor{green!30}0.003 & \cellcolor{green!30}0.004 & \cellcolor{green!30}0.004 & 0.005 & 0.011 & 0.026 \T\B \\
		& $||e^{T}||_{L^2}$  & \cellcolor{green!30}\num[exponent-product=\ensuremath{\cdot}]{7.291e-06} & \cellcolor{green!30}0.001 & \cellcolor{green!30}0.009 & 0.500 & 4.244 & 56.313 \T\B \\
		& $||e^{T}||_{dG}$  & \cellcolor{green!30}0.003 & \cellcolor{green!30}0.001 & \cellcolor{green!30}0.023 & 0.202 & 0.048 & 0.055 \T\B \\
		\hline
		\multirow{6}{*}{{$\mathcal{C}_h$}}
		& $||\mathbf{e}^{u}||_{L^2}$  & \cellcolor{green!30}\num[exponent-product=\ensuremath{\cdot}]{8.564e-06} & \num[exponent-product=\ensuremath{\cdot}]{1.053e-05} & \num[exponent-product=\ensuremath{\cdot}]{4.749e-4} & 0.001 & \num[exponent-product=\ensuremath{\cdot}]{8.175e-4} & \num[exponent-product=\ensuremath{\cdot}]{5.036e-4} \T\B \\
		& $||\mathbf{e}^{u}||_{dG}$  & \cellcolor{green!30}0.007 & \cellcolor{green!30}0.007 & 0.011 & \cellcolor{green!30}0.010 & 0.011 & \cellcolor{green!30}0.008 \T\B \\
		& $||e^{p}||_{L^2}$  & \cellcolor{green!30}\num[exponent-product=\ensuremath{\cdot}]{7.097e-06} & \cellcolor{green!30}\num[exponent-product=\ensuremath{\cdot}]{7.067e-06} & \num[exponent-product=\ensuremath{\cdot}]{3.718e-05} & \num[exponent-product=\ensuremath{\cdot}]{6.262e-05} & \num[exponent-product=\ensuremath{\cdot}]{6.422e-05} & \num[exponent-product=\ensuremath{\cdot}]{3.427e-05} \T\B \\
		& $||e^{p}||_{dG}$  & 0.004 & \cellcolor{green!30}0.004 & \cellcolor{green!30}0.004 & \cellcolor{green!30}0.004 & \cellcolor{green!30}0.004 & \cellcolor{green!30}0.004 \T\B \\
		& $||e^{T}||_{L^2}$  & \num[exponent-product=\ensuremath{\cdot}]{7.296e-06} & \cellcolor{green!30}0.001 & 0.405 & \cellcolor{green!30}0.095 & \cellcolor{green!30}0.163 & 8.795 \T\B \\
		& $||e^{T}||_{dG}$  & \cellcolor{green!30}0.003 & 0.002 & 0.262 & \cellcolor{green!30}0.008 & \cellcolor{green!30}0.002 & 0.008 \T\B \\
		\hline
		\multirow{6}{*}{{$\mathcal{C}_h^{\text{stab}}$}}
		& $||\mathbf{e}^{u}||_{L^2}$  & \cellcolor{green!30}\num[exponent-product=\ensuremath{\cdot}]{8.564e-06} & \num[exponent-product=\ensuremath{\cdot}]{1.057e-05} & \num[exponent-product=\ensuremath{\cdot}]{2.090e-4} & \cellcolor{green!30}\num[exponent-product=\ensuremath{\cdot}]{1.717e-4} & \cellcolor{green!30}\num[exponent-product=\ensuremath{\cdot}]{2.736e-05} & \cellcolor{green!30}\num[exponent-product=\ensuremath{\cdot}]{2.290e-05} \T\B \\
		& $||\mathbf{e}^{u}||_{dG}$  & \cellcolor{green!30}0.007 & \cellcolor{green!30}0.007 & 0.008 & 0.011 & \cellcolor{green!30}0.010 & 0.010 \T\B \\
		& $||e^{p}||_{L^2}$  & \cellcolor{green!30}\num[exponent-product=\ensuremath{\cdot}]{7.097e-06} & \cellcolor{green!30}\num[exponent-product=\ensuremath{\cdot}]{7.067e-06} & \num[exponent-product=\ensuremath{\cdot}]{1.784e-05} & \num[exponent-product=\ensuremath{\cdot}]{1.533e-05} & \cellcolor{green!30}\num[exponent-product=\ensuremath{\cdot}]{7.408e-06} & \cellcolor{green!30}\num[exponent-product=\ensuremath{\cdot}]{7.137e-06} \T\B \\
		& $||e^{p}||_{dG}$  & 0.004 & \cellcolor{green!30}0.004 & \cellcolor{green!30}0.004 & \cellcolor{green!30}0.004 & \cellcolor{green!30}0.004 & \cellcolor{green!30}0.004 \T\B \\
		& $||e^{T}||_{L^2}$  & \num[exponent-product=\ensuremath{\cdot}]{7.297e-06} & \cellcolor{green!30}0.001 & 0.179 & 0.612 & 0.759 & \cellcolor{green!30}0.777 \T\B \\
		& $||e^{T}||_{dG}$  & \cellcolor{green!30}0.003 & 0.002 & 0.086 & 0.034 & 0.005 & \cellcolor{green!30}\num[exponent-product=\ensuremath{\cdot}]{5.752e-4} \T\B \\
	\end{tabular}
	\caption{Robustness test vs $\boldsymbol{\Theta} = \theta \mathbf{I}$ Section 5.2.1: $L^2$- and $dG$-errors of $\mathbf{u}_h$, $p_h$, and $T_h$ versus $\theta$. The results for the four different choices of the linearized form are reported. The polynomial degree of approximation and the number of elements are taken as $\ell = 2$ and $N=1000$. In green we highlight the lowest error (for each field and each norm) between the four linearization schemes.}
	\label{tab:RobQSTPE_RobThetaErrors_p2}
\end{table}

\begin{table}
	\centering 
	\footnotesize
	\begin{tabular}{c | l | c | c | c | c | c | c}
		\textbf{Lin.} & \textbf{Error} & $\theta = 1$ & $\theta = \num[exponent-product=\ensuremath{\cdot}, print-unity-mantissa=false]{1e-2}$ & $\theta = \num[exponent-product=\ensuremath{\cdot}, print-unity-mantissa=false]{1e-4}$ & $\theta =  \num[exponent-product=\ensuremath{\cdot}, print-unity-mantissa=false]{1e-6}$ & $\theta = \num[exponent-product=\ensuremath{\cdot}, print-unity-mantissa=false]{1e-8}$ & $\theta = \num[exponent-product=\ensuremath{\cdot}, print-unity-mantissa=false]{1e-10}$ \T\B \\
		\hline
		\multirow{6}{*}{{$\mathcal{C}_h^{\text{old}}$}}
		& $||\mathbf{e}^{u}||_{L^2}$  & \cellcolor{green!30}\num[exponent-product=\ensuremath{\cdot}]{9.793e-07} & 1.588 & 2.174 & 21.127 & 3.656 & 4.391 \T\B \\
		& $||\mathbf{e}^{u}||_{dG}$  & \cellcolor{green!30}\num[exponent-product=\ensuremath{\cdot}]{5.537e-4} & 14.515 & 186.758 & \num[exponent-product=\ensuremath{\cdot}]{1.824e+3} & 207.322 & 142.031 \T\B \\
		& $||e^{p}||_{L^2}$  & \cellcolor{green!30}\cellcolor{green!30}\num[exponent-product=\ensuremath{\cdot}]{1.416e-06} & 0.155 & 0.170 & 0.171 & 0.168 & 0.169 \T\B \\
		& $||e^{p}||_{dG}$  & \cellcolor{green!30}\num[exponent-product=\ensuremath{\cdot}]{6.166e-4} & 0.965 & 1.271 & 3.144 & 1.014 & 1.033 \T\B \\
		& $||e^{T}||_{L^2}$  & \cellcolor{green!30}\num[exponent-product=\ensuremath{\cdot}]{1.414e-06} & 201.896 & \num[exponent-product=\ensuremath{\cdot}]{2.632e+3} & \num[exponent-product=\ensuremath{\cdot}]{2.198e+4} & \num[exponent-product=\ensuremath{\cdot}]{3.631e+4} & \num[exponent-product=\ensuremath{\cdot}]{4.073e+6} \T\B \\
		& $||e^{T}||_{dG}$  & \cellcolor{green!30}\num[exponent-product=\ensuremath{\cdot}]{6.006e-4} & 210.367 & \num[exponent-product=\ensuremath{\cdot}]{7.485e+3} & \num[exponent-product=\ensuremath{\cdot}]{1.650e+4} & 543.160 & \num[exponent-product=\ensuremath{\cdot}]{4.402e+3} \T\B \\
		\hline
		\multirow{6}{*}{{$\mathcal{C}_h^{\text{vol}}$}}
		& $||\mathbf{e}^{u}||_{L^2}$  & \cellcolor{green!30}\num[exponent-product=\ensuremath{\cdot}]{9.793e-07} & \num[exponent-product=\ensuremath{\cdot}]{9.797e-07} & \cellcolor{green!30}\num[exponent-product=\ensuremath{\cdot}]{1.415e-06} & \num[exponent-product=\ensuremath{\cdot}]{3.791e-05} & 0.005 & 0.020 \T\B \\
		& $||\mathbf{e}^{u}||_{dG}$  & \cellcolor{green!30}\num[exponent-product=\ensuremath{\cdot}]{5.537e-4} & \cellcolor{green!30}\num[exponent-product=\ensuremath{\cdot}]{5.537e-4} & \cellcolor{green!30}\num[exponent-product=\ensuremath{\cdot}]{5.540e-4} & 0.002 & 0.084 & 0.332 \T\B \\
		& $||e^{p}||_{L^2}$  & \cellcolor{green!30}\num[exponent-product=\ensuremath{\cdot}]{1.416e-06} & \cellcolor{green!30}\num[exponent-product=\ensuremath{\cdot}]{1.416e-06} & \cellcolor{green!30}\num[exponent-product=\ensuremath{\cdot}]{1.418e-06} & \num[exponent-product=\ensuremath{\cdot}]{1.806e-06} & \num[exponent-product=\ensuremath{\cdot}]{3.176e-4} & 0.001 \T\B \\
		& $||e^{p}||_{dG}$  & \cellcolor{green!30}\num[exponent-product=\ensuremath{\cdot}]{6.166e-4} & \cellcolor{green!30}\num[exponent-product=\ensuremath{\cdot}]{6.166e-4} & \cellcolor{green!30}\num[exponent-product=\ensuremath{\cdot}]{6.166e-4} & \num[exponent-product=\ensuremath{\cdot}]{6.171e-4} & 0.003 & 0.010 \T\B \\
		& $||e^{T}||_{L^2}$  & \cellcolor{green!30}\num[exponent-product=\ensuremath{\cdot}]{1.414e-06} & \cellcolor{green!30}\num[exponent-product=\ensuremath{\cdot}]{4.816e-06} & \cellcolor{green!30}\num[exponent-product=\ensuremath{\cdot}]{1.164e-4} & 0.011 & 0.555 & 32.109 \T\B \\
		& $||e^{T}||_{dG}$  & \cellcolor{green!30}\cellcolor{green!30}\num[exponent-product=\ensuremath{\cdot}]{6.006e-4} & \num[exponent-product=\ensuremath{\cdot}]{6.057e-05} & \cellcolor{green!30}\num[exponent-product=\ensuremath{\cdot}]{2.578e-4} & 0.004 & 0.022 & 0.036 \T\B \\
		\hline
		\multirow{6}{*}{{$\mathcal{C}_h$}}
		& $||\mathbf{e}^{u}||_{L^2}$  & \cellcolor{green!30}\num[exponent-product=\ensuremath{\cdot}]{9.793e-07} & \cellcolor{green!30}\num[exponent-product=\ensuremath{\cdot}]{9.796e-07} & \num[exponent-product=\ensuremath{\cdot}]{1.699e-05} & \num[exponent-product=\ensuremath{\cdot}]{2.869e-05} & \num[exponent-product=\ensuremath{\cdot}]{5.618e-4} & \num[exponent-product=\ensuremath{\cdot}]{5.097e-4} \T\B \\
		& $||\mathbf{e}^{u}||_{dG}$  & \cellcolor{green!30}\num[exponent-product=\ensuremath{\cdot}]{5.537e-4} & \cellcolor{green!30}\num[exponent-product=\ensuremath{\cdot}]{5.537e-4} & \num[exponent-product=\ensuremath{\cdot}]{6.202e-4} & \num[exponent-product=\ensuremath{\cdot}]{5.790e-4} & 0.003 & 0.003 \T\B \\
		& $||e^{p}||_{L^2}$  & \cellcolor{green!30}\num[exponent-product=\ensuremath{\cdot}]{1.416e-06} & \cellcolor{green!30}\num[exponent-product=\ensuremath{\cdot}]{1.416e-06} & \num[exponent-product=\ensuremath{\cdot}]{1.895e-06} & \num[exponent-product=\ensuremath{\cdot}]{2.262e-06} & \num[exponent-product=\ensuremath{\cdot}]{3.473e-05} & \num[exponent-product=\ensuremath{\cdot}]{3.152e-05} \T\B \\
		& $||e^{p}||_{dG}$  & \cellcolor{green!30}\num[exponent-product=\ensuremath{\cdot}]{6.166e-4} & \cellcolor{green!30}\num[exponent-product=\ensuremath{\cdot}]{6.166e-4} & \cellcolor{green!30}\num[exponent-product=\ensuremath{\cdot}]{6.166e-4} & \num[exponent-product=\ensuremath{\cdot}]{6.168e-4} & \num[exponent-product=\ensuremath{\cdot}]{6.639e-4} & \num[exponent-product=\ensuremath{\cdot}]{6.558e-4} \T\B \\
		& $||e^{T}||_{L^2}$  & \cellcolor{green!30}\num[exponent-product=\ensuremath{\cdot}]{1.414e-06} & \cellcolor{green!30}\num[exponent-product=\ensuremath{\cdot}]{4.316e-06} & 0.014 & \cellcolor{green!30}0.001 & 0.225 & 19.316 \T\B \\
		& $||e^{T}||_{dG}$  & \cellcolor{green!30}\cellcolor{green!30}\num[exponent-product=\ensuremath{\cdot}]{6.006e-4} & \cellcolor{green!30}\num[exponent-product=\ensuremath{\cdot}]{6.047e-05} & 0.010 & \cellcolor{green!30}\num[exponent-product=\ensuremath{\cdot}]{9.346e-05} & 0.002 & 0.021 \T\B \\
		\hline
		\multirow{6}{*}{{$\mathcal{C}_h^{\text{stab}}$}}
		& $||\mathbf{e}^{u}||_{L^2}$  & \cellcolor{green!30}\num[exponent-product=\ensuremath{\cdot}]{9.793e-07} & \num[exponent-product=\ensuremath{\cdot}]{9.797e-07} & \num[exponent-product=\ensuremath{\cdot}]{4.606e-06} & \cellcolor{green!30}\cellcolor{green!30}\num[exponent-product=\ensuremath{\cdot}]{5.238e-06} & \cellcolor{green!30}\num[exponent-product=\ensuremath{\cdot}]{5.127e-06} & \cellcolor{green!30}\num[exponent-product=\ensuremath{\cdot}]{5.143e-06} \T\B \\
		& $||\mathbf{e}^{u}||_{dG}$  & \cellcolor{green!30}\num[exponent-product=\ensuremath{\cdot}]{5.537e-4} & \cellcolor{green!30}\num[exponent-product=\ensuremath{\cdot}]{5.537e-4} & \num[exponent-product=\ensuremath{\cdot}]{5.567e-4} & \cellcolor{green!30}\num[exponent-product=\ensuremath{\cdot}]{5.628e-4} & \cellcolor{green!30}\num[exponent-product=\ensuremath{\cdot}]{5.596e-4} & \cellcolor{green!30}\num[exponent-product=\ensuremath{\cdot}]{5.596e-4} \T\B \\
		& $||e^{p}||_{L^2}$  & \cellcolor{green!30}\num[exponent-product=\ensuremath{\cdot}]{1.416e-06} & \cellcolor{green!30}\num[exponent-product=\ensuremath{\cdot}]{1.416e-06} & \num[exponent-product=\ensuremath{\cdot}]{1.460e-06} & \cellcolor{green!30}\num[exponent-product=\ensuremath{\cdot}]{1.475e-06} & \cellcolor{green!30}\num[exponent-product=\ensuremath{\cdot}]{1.454e-06} & \cellcolor{green!30}\num[exponent-product=\ensuremath{\cdot}]{1.453e-06} \T\B \\
		& $||e^{p}||_{dG}$  & \cellcolor{green!30}\num[exponent-product=\ensuremath{\cdot}]{6.166e-4} & \cellcolor{green!30}\num[exponent-product=\ensuremath{\cdot}]{6.166e-4} & \cellcolor{green!30}\num[exponent-product=\ensuremath{\cdot}]{6.166e-4} & \cellcolor{green!30}\num[exponent-product=\ensuremath{\cdot}]{6.167e-4} & \cellcolor{green!30}\num[exponent-product=\ensuremath{\cdot}]{6.166e-4} & \cellcolor{green!30}\num[exponent-product=\ensuremath{\cdot}]{6.166e-4} \T\B \\
		& $||e^{T}||_{L^2}$  & \cellcolor{green!30}\num[exponent-product=\ensuremath{\cdot}]{1.414e-06} & \num[exponent-product=\ensuremath{\cdot}]{4.850e-06} & 0.002 & 0.007 & \cellcolor{green!30}0.007 & \cellcolor{green!30}0.007 \T\B \\
		& $||e^{T}||_{dG}$  & \cellcolor{green!30}\num[exponent-product=\ensuremath{\cdot}]{6.006e-4} & \num[exponent-product=\ensuremath{\cdot}]{6.061e-05} & 0.001 & \num[exponent-product=\ensuremath{\cdot}]{4.356e-4} & \cellcolor{green!30}\num[exponent-product=\ensuremath{\cdot}]{5.351e-05} & \cellcolor{green!30}\num[exponent-product=\ensuremath{\cdot}]{5.452e-06} \T\B \\
	\end{tabular}
	\caption{Robustness test vs $\boldsymbol{\Theta} = \theta \mathbf{I}$ Section 5.2.1: $L^2$- and $dG$-errors of $\mathbf{u}_h$, $p_h$, and $T_h$ versus $\theta$. The results for the four different choices of the linearized form are reported. The polynomial degree of approximation and the number of elements are taken as $\ell = 3$ and $N=310$. In green we highlight the lowest error (for each field and each norm) between the four linearization schemes.}
	\label{tab:RobQSTPE_RobThetaErrors_p3}
\end{table}

\begin{table}
	\centering 
	\footnotesize
	\begin{tabular}{c | l | c | c | c | c | c | c}
		\textbf{Lin.} & \textbf{Error} & $\theta = 1$ & $\theta = \num[exponent-product=\ensuremath{\cdot}, print-unity-mantissa=false]{1e-2}$ & $\theta = \num[exponent-product=\ensuremath{\cdot}, print-unity-mantissa=false]{1e-4}$ & $\theta =  \num[exponent-product=\ensuremath{\cdot}, print-unity-mantissa=false]{1e-6}$ & $\theta = \num[exponent-product=\ensuremath{\cdot}, print-unity-mantissa=false]{1e-8}$ & $\theta = \num[exponent-product=\ensuremath{\cdot}, print-unity-mantissa=false]{1e-10}$ \T\B \\
		\hline
		\multirow{6}{*}{{$\mathcal{C}_h^{\text{old}}$}}
		& $||\mathbf{e}^{u}||_{L^2}$  & \cellcolor{green!30}\num[exponent-product=\ensuremath{\cdot}]{6.543e-07} & 1.590 & 4.635 & 1.975 & 17.557 & 78.719 \T\B \\
		& $||\mathbf{e}^{u}||_{dG}$  & \cellcolor{green!30}\num[exponent-product=\ensuremath{\cdot}]{3.019e-4} & 12.228 & 414.469 & 80.932 & 732.996 & \num[exponent-product=\ensuremath{\cdot}]{4.755e+3} \T\B \\
		& $||e^{p}||_{L^2}$  & \cellcolor{green!30}\num[exponent-product=\ensuremath{\cdot}]{5.327e-07} & 0.131 & 0.168 & 0.169 & 0.168 & 0.171 \T\B \\
		& $||e^{p}||_{dG}$  & \cellcolor{green!30}\num[exponent-product=\ensuremath{\cdot}]{1.713e-4} & 0.788 & 2.205 & 0.929 & 1.323 & 5.616 \T\B \\
		& $||e^{T}||_{L^2}$  & \cellcolor{green!30}\num[exponent-product=\ensuremath{\cdot}]{6.261e-07} & 189.214 & \num[exponent-product=\ensuremath{\cdot}]{3.869e+3} & \num[exponent-product=\ensuremath{\cdot}]{2.759e+3} & \num[exponent-product=\ensuremath{\cdot}]{5.005e+4} & \num[exponent-product=\ensuremath{\cdot}]{5.016e+6} \T\B \\
		& $||e^{T}||_{dG}$  & \cellcolor{green!30}\num[exponent-product=\ensuremath{\cdot}]{1.773e-4} & 290.314 & \num[exponent-product=\ensuremath{\cdot}]{1.507e+4} & 789.085 & 980.373 & \num[exponent-product=\ensuremath{\cdot}]{5.485e+3} \T\B \\
		\hline
		\multirow{6}{*}{{$\mathcal{C}_h^{\text{vol}}$}}
		& $||\mathbf{e}^{u}||_{L^2}$  & \cellcolor{green!30}\num[exponent-product=\ensuremath{\cdot}]{6.543e-07} & \cellcolor{green!30}\num[exponent-product=\ensuremath{\cdot}]{1.322e-06} & \num[exponent-product=\ensuremath{\cdot}]{7.914e-07} & \num[exponent-product=\ensuremath{\cdot}]{4.849e-4} & \num[exponent-product=\ensuremath{\cdot}]{2.780e-4} & 0.002 \T\B \\
		& $||\mathbf{e}^{u}||_{dG}$  & \cellcolor{green!30}\num[exponent-product=\ensuremath{\cdot}]{3.019e-4} & \cellcolor{green!30}\num[exponent-product=\ensuremath{\cdot}]{3.021e-4} & \cellcolor{green!30}\num[exponent-product=\ensuremath{\cdot}]{3.019e-4} & 0.006 & 0.004 & 0.031 \T\B \\
		& $||e^{p}||_{L^2}$  & \cellcolor{green!30}\num[exponent-product=\ensuremath{\cdot}]{5.327e-07} & \cellcolor{green!30}\num[exponent-product=\ensuremath{\cdot}]{5.450e-07} & \num[exponent-product=\ensuremath{\cdot}]{5.348e-07} & \num[exponent-product=\ensuremath{\cdot}]{2.601e-05} & \num[exponent-product=\ensuremath{\cdot}]{1.893e-05} & \num[exponent-product=\ensuremath{\cdot}]{1.535e-4} \T\B \\
		& $||e^{p}||_{dG}$  & \cellcolor{green!30}\num[exponent-product=\ensuremath{\cdot}]{1.713e-4} & \cellcolor{green!30}\num[exponent-product=\ensuremath{\cdot}]{1.713e-4} & \cellcolor{green!30}\num[exponent-product=\ensuremath{\cdot}]{1.713e-4} & \num[exponent-product=\ensuremath{\cdot}]{2.897e-4} & \num[exponent-product=\ensuremath{\cdot}]{2.480e-4} & 0.001 \T\B \\
		& $||e^{T}||_{L^2}$  & \cellcolor{green!30}\num[exponent-product=\ensuremath{\cdot}]{6.261e-07} & \cellcolor{green!30}\num[exponent-product=\ensuremath{\cdot}]{2.102e-4} & \num[exponent-product=\ensuremath{\cdot}]{5.627e-05} & 0.038 & 0.025 & 2.311 \T\B \\
		& $||e^{T}||_{dG}$  & \cellcolor{green!30}\num[exponent-product=\ensuremath{\cdot}]{1.773e-4} & \cellcolor{green!30}\num[exponent-product=\ensuremath{\cdot}]{2.877e-4} & \num[exponent-product=\ensuremath{\cdot}]{5.011e-05} & 0.011 & \num[exponent-product=\ensuremath{\cdot}]{6.357e-4} & 0.003 \T\B \\
		\hline
		\multirow{6}{*}{{$\mathcal{C}_h$}}
		& $||\mathbf{e}^{u}||_{L^2}$  & \cellcolor{green!30}\num[exponent-product=\ensuremath{\cdot}]{6.543e-07} & \num[exponent-product=\ensuremath{\cdot}]{2.018e-06} & \cellcolor{green!30}\num[exponent-product=\ensuremath{\cdot}]{6.610e-07} & \cellcolor{green!30}\num[exponent-product=\ensuremath{\cdot}]{7.920e-07} & \num[exponent-product=\ensuremath{\cdot}]{1.081e-06} & \num[exponent-product=\ensuremath{\cdot}]{1.083e-06} \T\B \\
		& $||\mathbf{e}^{u}||_{dG}$  & \cellcolor{green!30}\num[exponent-product=\ensuremath{\cdot}]{3.019e-4} & \num[exponent-product=\ensuremath{\cdot}]{3.025e-4} & \cellcolor{green!30}\num[exponent-product=\ensuremath{\cdot}]{3.019e-4} & \cellcolor{green!30}\num[exponent-product=\ensuremath{\cdot}]{3.018e-4} & \cellcolor{green!30}\num[exponent-product=\ensuremath{\cdot}]{3.018e-4} & \cellcolor{green!30}\num[exponent-product=\ensuremath{\cdot}]{3.018e-4} \T\B \\
		& $||e^{p}||_{L^2}$  & \cellcolor{green!30}\num[exponent-product=\ensuremath{\cdot}]{5.327e-07} & \num[exponent-product=\ensuremath{\cdot}]{5.660e-07} & \cellcolor{green!30}\num[exponent-product=\ensuremath{\cdot}]{5.327e-07} & \cellcolor{green!30}\num[exponent-product=\ensuremath{\cdot}]{5.333e-07} & \cellcolor{green!30}\num[exponent-product=\ensuremath{\cdot}]{5.349e-07} & \cellcolor{green!30}\num[exponent-product=\ensuremath{\cdot}]{5.349e-07} \T\B \\
		& $||e^{p}||_{dG}$  & \cellcolor{green!30}\num[exponent-product=\ensuremath{\cdot}]{1.713e-4} & \cellcolor{green!30}\num[exponent-product=\ensuremath{\cdot}]{1.713e-4} & \cellcolor{green!30}\num[exponent-product=\ensuremath{\cdot}]{1.713e-4} & \cellcolor{green!30}\num[exponent-product=\ensuremath{\cdot}]{1.713e-4} & \cellcolor{green!30}\num[exponent-product=\ensuremath{\cdot}]{1.713e-4} & \cellcolor{green!30}\num[exponent-product=\ensuremath{\cdot}]{1.713e-4} \T\B \\
		& $||e^{T}||_{L^2}$  & \cellcolor{green!30}\num[exponent-product=\ensuremath{\cdot}]{6.261e-07} & \num[exponent-product=\ensuremath{\cdot}]{3.493e-4} & \cellcolor{green!30}\num[exponent-product=\ensuremath{\cdot}]{1.376e-05} & \num[exponent-product=\ensuremath{\cdot}]{2.218e-4} & 0.004 & 0.416 \T\B \\
		& $||e^{T}||_{dG}$  & \cellcolor{green!30}\num[exponent-product=\ensuremath{\cdot}]{1.773e-4} & \num[exponent-product=\ensuremath{\cdot}]{4.769e-4} & \cellcolor{green!30}\num[exponent-product=\ensuremath{\cdot}]{1.134e-05} & \num[exponent-product=\ensuremath{\cdot}]{2.043e-05} & \num[exponent-product=\ensuremath{\cdot}]{4.699e-05} & \num[exponent-product=\ensuremath{\cdot}]{4.515e-4} \T\B \\
		\hline
		\multirow{6}{*}{{$\mathcal{C}_h^{\text{stab}}$}}
		& $||\mathbf{e}^{u}||_{L^2}$  & \cellcolor{green!30}\num[exponent-product=\ensuremath{\cdot}]{6.543e-07} & \num[exponent-product=\ensuremath{\cdot}]{1.984e-06} & \num[exponent-product=\ensuremath{\cdot}]{6.657e-07} & \num[exponent-product=\ensuremath{\cdot}]{9.170e-07} & \cellcolor{green!30}\num[exponent-product=\ensuremath{\cdot}]{1.057e-06} & \cellcolor{green!30}\num[exponent-product=\ensuremath{\cdot}]{1.059e-06} \T\B \\
		& $||\mathbf{e}^{u}||_{dG}$  & \cellcolor{green!30}\num[exponent-product=\ensuremath{\cdot}]{3.019e-4} & \num[exponent-product=\ensuremath{\cdot}]{3.025e-4} & \cellcolor{green!30}\num[exponent-product=\ensuremath{\cdot}]{3.019e-4} & \num[exponent-product=\ensuremath{\cdot}]{3.021e-4} & \num[exponent-product=\ensuremath{\cdot}]{3.021e-4} & \num[exponent-product=\ensuremath{\cdot}]{3.021e-4} \T\B \\
		& $||e^{p}||_{L^2}$  & \cellcolor{green!30}\num[exponent-product=\ensuremath{\cdot}]{5.327e-07} & \num[exponent-product=\ensuremath{\cdot}]{5.650e-07} & \cellcolor{green!30}\num[exponent-product=\ensuremath{\cdot}]{5.327e-07} & \num[exponent-product=\ensuremath{\cdot}]{5.344e-07} & \num[exponent-product=\ensuremath{\cdot}]{5.353e-07} & \num[exponent-product=\ensuremath{\cdot}]{5.353e-07} \T\B \\
		& $||e^{p}||_{dG}$  & \cellcolor{green!30}\num[exponent-product=\ensuremath{\cdot}]{1.713e-4} & \cellcolor{green!30}\num[exponent-product=\ensuremath{\cdot}]{1.7131e-4} & \cellcolor{green!30}\num[exponent-product=\ensuremath{\cdot}]{1.713e-4} & \cellcolor{green!30}\num[exponent-product=\ensuremath{\cdot}]{1.713e-4} & \cellcolor{green!30}\num[exponent-product=\ensuremath{\cdot}]{1.713e-4} & \cellcolor{green!30}\num[exponent-product=\ensuremath{\cdot}]{1.713e-4} \T\B \\
		& $||e^{T}||_{L^2}$  & \cellcolor{green!30}\num[exponent-product=\ensuremath{\cdot}]{6.261e-07} & \num[exponent-product=\ensuremath{\cdot}]{3.431e-4} & \num[exponent-product=\ensuremath{\cdot}]{2.379e-05} & \cellcolor{green!30}\num[exponent-product=\ensuremath{\cdot}]{1.465e-4} & \cellcolor{green!30}\num[exponent-product=\ensuremath{\cdot}]{1.830e-4} & \cellcolor{green!30}\num[exponent-product=\ensuremath{\cdot}]{1.837e-4} \T\B \\
		& $||e^{T}||_{dG}$  & \cellcolor{green!30}\num[exponent-product=\ensuremath{\cdot}]{1.773e-4} & \num[exponent-product=\ensuremath{\cdot}]{4.681e-4} & \num[exponent-product=\ensuremath{\cdot}]{1.436e-05} & \cellcolor{green!30}\num[exponent-product=\ensuremath{\cdot}]{9.604e-06} & \cellcolor{green!30}\num[exponent-product=\ensuremath{\cdot}]{1.952e-06} & \cellcolor{green!30}\num[exponent-product=\ensuremath{\cdot}]{2.055e-07} \T\B \\
	\end{tabular}
	\caption{Robustness test vs $\boldsymbol{\Theta} = \theta \mathbf{I}$ Section 5.2.1: $L^2$- and $dG$-errors of $\mathbf{u}_h$, $p_h$, and $T_h$ versus $\theta$. The results for the four different choices of the linearized form are reported. The polynomial degree of approximation and the number of elements are taken as $\ell = 4$ and $N=100$. In green we highlight the lowest error (for each field and each norm) between the four linearization schemes.}
	\label{tab:RobQSTPE_RobThetaErrors_p4}
\end{table}

\begin{table}
	\centering 
	\footnotesize
	\begin{tabular}{c | l | c | c | c | c | c }
		\textbf{Lin.} & \textbf{Error} & $N=100$ & $N=310$ & $N=1000$ & $N=3100$ & $N=10000$ \T\B \\
		\hline
		\multirow{6}{*}{{$\mathcal{C}_h^{\text{old}}$}}
		& $||\mathbf{e}^{u}||_{L^2}$  & \num[exponent-product=\ensuremath{\cdot}]{3.158e-4} & \num[exponent-product=\ensuremath{\cdot}]{5.299e-05} & \num[exponent-product=\ensuremath{\cdot}]{8.493e-06} & \num[exponent-product=\ensuremath{\cdot}]{1.334e-06} & \num[exponent-product=\ensuremath{\cdot}]{2.300e-07} \T\B \\
		& $||\mathbf{e}^{u}||_{dG}$  & 0.068 & 0.020 & 0.006 & 0.002 & \num[exponent-product=\ensuremath{\cdot}]{6.399e-4} \T\B \\
		& $||e^{p}||_{L^2}$  & 0.184 & 0.056 & 0.021 & 0.008 & 0.003 \T\B \\
		& $||e^{p}||_{dG}$  & \num[exponent-product=\ensuremath{\cdot}]{1.456e-4} & \num[exponent-product=\ensuremath{\cdot}]{7.364e-05} & \num[exponent-product=\ensuremath{\cdot}]{4.418e-05} & \num[exponent-product=\ensuremath{\cdot}]{2.795e-05} & \num[exponent-product=\ensuremath{\cdot}]{1.516e-05} \T\B \\
		& $||e^{T}||_{L^2}$  & \num[exponent-product=\ensuremath{\cdot}]{3.928e-4} & \num[exponent-product=\ensuremath{\cdot}]{4.800e-05} & \num[exponent-product=\ensuremath{\cdot}]{7.368e-06} & \num[exponent-product=\ensuremath{\cdot}]{1.182e-06} & \num[exponent-product=\ensuremath{\cdot}]{1.913e-07} \T\B \\
		& $||e^{T}||_{dG}$  & 0.037 & 0.012 & 0.003 & 0.001 & \num[exponent-product=\ensuremath{\cdot}]{3.380e-4} \T\B \\
		\hline
		\multirow{6}{*}{{$\mathcal{C}_h^{\text{vol}}$}}
		& $||\mathbf{e}^{u}||_{L^2}$  & \num[exponent-product=\ensuremath{\cdot}]{3.158e-4} & \num[exponent-product=\ensuremath{\cdot}]{5.299e-05} & \num[exponent-product=\ensuremath{\cdot}]{8.493e-06} & \num[exponent-product=\ensuremath{\cdot}]{1.334e-06} & \num[exponent-product=\ensuremath{\cdot}]{2.300e-07} \T\B \\
		& $||\mathbf{e}^{u}||_{dG}$  & 0.068 & 0.020 & 0.006 & 0.002 & \num[exponent-product=\ensuremath{\cdot}]{6.399e-4} \T\B \\
		& $||e^{p}||_{L^2}$  & 0.184 & 0.056 & 0.021 & 0.008 & 0.003 \T\B \\
		& $||e^{p}||_{dG}$  & \num[exponent-product=\ensuremath{\cdot}]{1.456e-4} & \num[exponent-product=\ensuremath{\cdot}]{7.364e-05} & \num[exponent-product=\ensuremath{\cdot}]{4.418e-05} & \num[exponent-product=\ensuremath{\cdot}]{2.795e-05} & \num[exponent-product=\ensuremath{\cdot}]{1.516e-05} \T\B \\
		& $||e^{T}||_{L^2}$  & \num[exponent-product=\ensuremath{\cdot}]{3.928e-4} & \num[exponent-product=\ensuremath{\cdot}]{4.800e-05} & \num[exponent-product=\ensuremath{\cdot}]{7.368e-06} & \num[exponent-product=\ensuremath{\cdot}]{1.182e-06} & \num[exponent-product=\ensuremath{\cdot}]{1.913e-07} \T\B \\
		& $||e^{T}||_{dG}$  & 0.037 & 0.012 & 0.003 & 0.001 & \num[exponent-product=\ensuremath{\cdot}]{3.380e-4} \T\B \\
		\hline
		\multirow{6}{*}{{$\mathcal{C}_h$}}
		& $||\mathbf{e}^{u}||_{L^2}$  & \num[exponent-product=\ensuremath{\cdot}]{3.158e-4} & \num[exponent-product=\ensuremath{\cdot}]{5.299e-05} & \num[exponent-product=\ensuremath{\cdot}]{8.493e-06} & \num[exponent-product=\ensuremath{\cdot}]{1.334e-06} & \num[exponent-product=\ensuremath{\cdot}]{2.300e-07} \T\B \\
		& $||\mathbf{e}^{u}||_{dG}$  & 0.068 & 0.020 & 0.006 & 0.002 & \num[exponent-product=\ensuremath{\cdot}]{6.399e-4} \T\B \\
		& $||e^{p}||_{L^2}$  & 0.184 & 0.056 & 0.021 & 0.008 & 0.003 \T\B \\
		& $||e^{p}||_{dG}$  & \num[exponent-product=\ensuremath{\cdot}]{1.456e-4} & \num[exponent-product=\ensuremath{\cdot}]{7.364e-05} & \num[exponent-product=\ensuremath{\cdot}]{4.418e-05} & \num[exponent-product=\ensuremath{\cdot}]{2.795e-05} & \num[exponent-product=\ensuremath{\cdot}]{1.516e-05} \T\B \\
		& $||e^{T}||_{L^2}$  & \num[exponent-product=\ensuremath{\cdot}]{3.928e-4} & \num[exponent-product=\ensuremath{\cdot}]{4.800e-05} & \num[exponent-product=\ensuremath{\cdot}]{7.368e-06} & \num[exponent-product=\ensuremath{\cdot}]{1.182e-06} & \num[exponent-product=\ensuremath{\cdot}]{1.913e-07} \T\B \\
		& $||e^{T}||_{dG}$  & 0.037 & 0.012 & 0.003 & 0.001 & \num[exponent-product=\ensuremath{\cdot}]{3.380e-4} \T\B \\
		\hline
		\multirow{6}{*}{{$\mathcal{C}_h^{\text{stab}}$}}
		& $||\mathbf{e}^{u}||_{L^2}$  & \num[exponent-product=\ensuremath{\cdot}]{3.158e-4} & \num[exponent-product=\ensuremath{\cdot}]{5.299e-05} & \num[exponent-product=\ensuremath{\cdot}]{8.493e-06} & \num[exponent-product=\ensuremath{\cdot}]{1.334e-06} & \num[exponent-product=\ensuremath{\cdot}]{2.300e-07} \T\B \\
		& $||\mathbf{e}^{u}||_{dG}$  & 0.068 & 0.020 & 0.006 & 0.002 & \num[exponent-product=\ensuremath{\cdot}]{6.399e-4} \T\B \\
		& $||e^{p}||_{L^2}$  & 0.184 & 0.056 & 0.021 & 0.008 & 0.003 \T\B \\
		& $||e^{p}||_{dG}$  & \num[exponent-product=\ensuremath{\cdot}]{1.456e-4} & \num[exponent-product=\ensuremath{\cdot}]{7.364e-05} & \num[exponent-product=\ensuremath{\cdot}]{4.418e-05} & \num[exponent-product=\ensuremath{\cdot}]{2.795e-05} & \num[exponent-product=\ensuremath{\cdot}]{1.516e-05} \T\B \\
		& $||e^{T}||_{L^2}$  & \num[exponent-product=\ensuremath{\cdot}]{3.928e-4} & \num[exponent-product=\ensuremath{\cdot}]{4.800e-05} & \num[exponent-product=\ensuremath{\cdot}]{7.368e-06} & \num[exponent-product=\ensuremath{\cdot}]{1.182e-06} & \num[exponent-product=\ensuremath{\cdot}]{1.913e-07} \T\B \\
		& $||e^{T}||_{dG}$  & 0.037 & 0.012 & 0.003 & 0.001 & \num[exponent-product=\ensuremath{\cdot}]{3.380e-4} \T\B \\
	\end{tabular}
	\caption{Robustness test vs $\mathbf{K} = k \mathbf{I}$ Section 5.2.2: $L^2$- and $dG$-errors of $\mathbf{u}_h$, $p_h$, and $T_h$ versus $N$. The results for the four different choices of the linearized form are reported. The polynomial degree of approximation and the number of elements are taken as $\ell = 2$ and $N=1000$.}
	\label{tab:RobQSTPE_RobKappaErrors_p2}
\end{table}

\begin{table}
	\centering 
	\footnotesize
	\begin{tabular}{c | l | c | c | c | c | c }
		\textbf{Lin.} & \textbf{Error} & $N=100$ & $N=310$ & $N=1000$ & $N=3100$ & $N=10000$ \T\B \\
		\hline
		\multirow{6}{*}{{$\mathcal{C}_h^{\text{old}}$}}
		& $||\mathbf{e}^{u}||_{L^2}$  & \num[exponent-product=\ensuremath{\cdot}]{3.158e-4} & \num[exponent-product=\ensuremath{\cdot}]{5.299e-05} & \num[exponent-product=\ensuremath{\cdot}]{8.494e-06} & \num[exponent-product=\ensuremath{\cdot}]{1.334e-06} & \num[exponent-product=\ensuremath{\cdot}]{2.300e-07} \T\B \\
		& $||\mathbf{e}^{u}||_{dG}$  & 0.068 & 0.020 & 0.006 & 0.002 & \num[exponent-product=\ensuremath{\cdot}]{6.399e-4} \T\B \\
		& $||e^{p}||_{L^2}$  & 0.100 & 0.031 & 0.012 & 0.004 & 0.001 \T\B \\
		& $||e^{p}||_{dG}$  & \num[exponent-product=\ensuremath{\cdot}]{8.096e-05} & \num[exponent-product=\ensuremath{\cdot}]{4.094e-05} & \num[exponent-product=\ensuremath{\cdot}]{2.456e-05} & \num[exponent-product=\ensuremath{\cdot}]{1.554e-05} & \num[exponent-product=\ensuremath{\cdot}]{8.436e-06} \T\B \\
		& $||e^{T}||_{L^2}$  & 0.100 & 0.031 & 0.012 & 0.004 & 0.001 \T\B \\
		& $||e^{T}||_{dG}$  & \num[exponent-product=\ensuremath{\cdot}]{8.095e-05} & \num[exponent-product=\ensuremath{\cdot}]{4.094e-05} & \num[exponent-product=\ensuremath{\cdot}]{2.455e-05} & \num[exponent-product=\ensuremath{\cdot}]{1.553e-05} & \num[exponent-product=\ensuremath{\cdot}]{8.425e-06} \T\B \\
		\hline
		\multirow{6}{*}{{$\mathcal{C}_h^{\text{vol}}$}}
		& $||\mathbf{e}^{u}||_{L^2}$  & \num[exponent-product=\ensuremath{\cdot}]{3.158e-4} & \num[exponent-product=\ensuremath{\cdot}]{5.299e-05} & \num[exponent-product=\ensuremath{\cdot}]{8.494e-06} & \num[exponent-product=\ensuremath{\cdot}]{1.334e-06} & \num[exponent-product=\ensuremath{\cdot}]{2.300e-07} \T\B \\
		& $||\mathbf{e}^{u}||_{dG}$  & 0.068 & 0.020 & 0.006 & 0.002 & \num[exponent-product=\ensuremath{\cdot}]{6.399e-4} \T\B \\
		& $||e^{p}||_{L^2}$  & 0.100 & 0.031 & 0.012 & 0.004 & 0.001 \T\B \\
		& $||e^{p}||_{dG}$  & \num[exponent-product=\ensuremath{\cdot}]{8.096e-05} & \num[exponent-product=\ensuremath{\cdot}]{4.094e-05} & \num[exponent-product=\ensuremath{\cdot}]{2.456e-05} & \num[exponent-product=\ensuremath{\cdot}]{1.554e-05} & \num[exponent-product=\ensuremath{\cdot}]{8.436e-06} \T\B \\
		& $||e^{T}||_{L^2}$  & 0.100 & 0.031 & 0.012 & 0.004 & 0.001 \T\B \\
		& $||e^{T}||_{dG}$  & \num[exponent-product=\ensuremath{\cdot}]{8.095e-05} & \num[exponent-product=\ensuremath{\cdot}]{4.094e-05} & \num[exponent-product=\ensuremath{\cdot}]{2.455e-05} & \num[exponent-product=\ensuremath{\cdot}]{1.553e-05} & \num[exponent-product=\ensuremath{\cdot}]{8.425e-06} \T\B \\
		\hline
		\multirow{6}{*}{{$\mathcal{C}_h$}}
		& $||\mathbf{e}^{u}||_{L^2}$  & \num[exponent-product=\ensuremath{\cdot}]{3.158e-4} & \num[exponent-product=\ensuremath{\cdot}]{5.299e-05} & \num[exponent-product=\ensuremath{\cdot}]{8.494e-06} & \num[exponent-product=\ensuremath{\cdot}]{1.334e-06} & \num[exponent-product=\ensuremath{\cdot}]{2.300e-07} \T\B \\
		& $||\mathbf{e}^{u}||_{dG}$  & 0.068 & 0.020 & 0.006 & 0.002 & \num[exponent-product=\ensuremath{\cdot}]{6.399e-4} \T\B \\
		& $||e^{p}||_{L^2}$  & 0.100 & 0.031 & 0.012 & 0.004 & 0.001 \T\B \\
		& $||e^{p}||_{dG}$  & \num[exponent-product=\ensuremath{\cdot}]{8.096e-05} & \num[exponent-product=\ensuremath{\cdot}]{4.094e-05} & \num[exponent-product=\ensuremath{\cdot}]{2.456e-05} & \num[exponent-product=\ensuremath{\cdot}]{1.554e-05} & \num[exponent-product=\ensuremath{\cdot}]{8.436e-06} \T\B \\
		& $||e^{T}||_{L^2}$  & 0.100 & 0.031 & 0.012 & 0.004 & 0.001 \T\B \\
		& $||e^{T}||_{dG}$  & \num[exponent-product=\ensuremath{\cdot}]{8.095e-05} & \num[exponent-product=\ensuremath{\cdot}]{4.094e-05} & \num[exponent-product=\ensuremath{\cdot}]{2.455e-05} & \num[exponent-product=\ensuremath{\cdot}]{1.553e-05} & \num[exponent-product=\ensuremath{\cdot}]{8.425e-06} \T\B \\
		\hline
		\multirow{6}{*}{{$\mathcal{C}_h^{\text{stab}}$}}
		& $||\mathbf{e}^{u}||_{L^2}$  & \num[exponent-product=\ensuremath{\cdot}]{3.158e-4} & \num[exponent-product=\ensuremath{\cdot}]{5.299e-05} & \num[exponent-product=\ensuremath{\cdot}]{8.494e-06} & \num[exponent-product=\ensuremath{\cdot}]{1.334e-06} & \num[exponent-product=\ensuremath{\cdot}]{2.300e-07} \T\B \\
		& $||\mathbf{e}^{u}||_{dG}$  & 0.068 & 0.020 & 0.006 & 0.002 & \num[exponent-product=\ensuremath{\cdot}]{6.399e-4} \T\B \\
		& $||e^{p}||_{L^2}$  & 0.100 & 0.031 & 0.012 & 0.004 & 0.001 \T\B \\
		& $||e^{p}||_{dG}$  & \num[exponent-product=\ensuremath{\cdot}]{8.096e-05} & \num[exponent-product=\ensuremath{\cdot}]{4.094e-05} & \num[exponent-product=\ensuremath{\cdot}]{2.456e-05} & \num[exponent-product=\ensuremath{\cdot}]{1.554e-05} & \num[exponent-product=\ensuremath{\cdot}]{8.436e-06} \T\B \\
		& $||e^{T}||_{L^2}$  & 0.100 & 0.031 & 0.012 & 0.004 & 0.001 \T\B \\
		& $||e^{T}||_{dG}$  & \num[exponent-product=\ensuremath{\cdot}]{8.095e-05} & \num[exponent-product=\ensuremath{\cdot}]{4.094e-05} & \num[exponent-product=\ensuremath{\cdot}]{2.455e-05} & \num[exponent-product=\ensuremath{\cdot}]{1.553e-05} & \num[exponent-product=\ensuremath{\cdot}]{8.425e-06} \T\B \\
	\end{tabular}
	\caption{Robustness test vs $\mathbf{K} = k \mathbf{I}, \boldsymbol{\Theta} = \theta \mathbf{I}$ Section 5.2.3: $L^2$- and $dG$-errors of $\mathbf{u}_h$, $p_h$, and $T_h$ versus $N$. The results for the four different choices of the linearized form are reported. The polynomial degree of approximation and the number of elements are taken as $\ell = 2$ and $N=1000$.}
	\label{tab:RobQSTPE_RobThetaKappaErrors_p2}
\end{table}

\end{document}